\documentclass[11pt,a4paper]{amsart}  % AMS article document class, 12pt font
\usepackage[top=1in, bottom=1in, left=1in, right=1in]{geometry}  % Page margins
\usepackage{amsmath,amssymb,amsfonts,amscd,amsxtra,amsthm}  % AMS mathematical packages
\usepackage{mathtools,mathrsfs}  % Enhanced math tools and script fonts
\usepackage{graphicx,subfig,svg}  % Graphics and sub-figures
\usepackage{latexsym}  % LaTeX symbols
\usepackage{xcolor,color}  % Color support
\usepackage{epstopdf}  % EPS to PDF conversion
\usepackage{tikz, tikz-cd}  % TikZ graphics and commutative diagrams
\usepackage{bm, bbm}  % Bold math and blackboard bold fonts
\usepackage{enumitem}  % Customized lists
\usepackage{cite}  % Bibliography management
\usepackage{caption}  % Figure and table captions
\usepackage{multirow}  % Table multirow cells
\usepackage{algorithmic}  % Algorithmic environment
\usepackage{algorithm}  % Algorithm floating environment
\usepackage{rotating}  % Rotated content (sideways tables/figures)
\usepackage{verbatim}

\usepackage[toc,page]{appendix}  % Appendix support
\usepackage[colorlinks=true, linkcolor=blue, citecolor=magenta]{hyperref}  % Hyperlinks

\numberwithin{equation}{section}  % Number equations within sections

\usepackage{cleveref}  % Cross-referencing with lowercase labels

\newtheorem{theorem}{Theorem}[section]
\newtheorem{lemma}{Lemma}[section]

\newtheorem{proposition}{Proposition}[section]
\newtheorem{definition}{Definition}[section]
\newtheorem{assumption}{Assumption}[section]
\newtheorem{remark}{Remark}[section]

\crefname{theorem}{Theorem}{Theorems}
\Crefname{theorem}{Theorem}{Theorems}

\crefname{lemma}{Lemma}{Lemmas}
\Crefname{lemma}{Lemma}{Lemmas}

\crefname{corollary}{Corollary}{Corollaries}
\Crefname{corollary}{Corollary}{Corollaries}

\crefname{proposition}{Proposition}{Propositions}
\Crefname{proposition}{Proposition}{Propositions}

\crefname{definition}{Definition}{Definitions}
\Crefname{definition}{Definition}{Definitions}

\crefname{assumption}{Assumption}{Assumptions}
\Crefname{assumption}{Assumption}{Assumptions}

\crefname{remark}{Remark}{Remarks}
\Crefname{remark}{Remark}{Remarks}

\crefname{example}{Example}{Examples}
\Crefname{example}{Example}{Examples}

\crefname{section}{Appendix}{Appendices}
\DeclareMathOperator{\divg}{div}

\newcommand{\bR}{\mathbb{R}}

\newcommand{\velocity}{\theta}

\graphicspath{{pics/}} % Folder for images (ensure 'pics' directory exists)

\title[Broadband Superabsorption of Waves by Metascreens]{
Reduced Order Model for Broadband Superabsorption of Waves by Metascreens}

\author{Habib Ammari}
\address{(H. Ammari) Department of Mathematics, ETH Z\"urich, Rämistrasse 101, 8092 Z\"urich, Switzerland}
\email{habib.ammari@math.ethz.ch}

\author{Yu Gao}
\address{(Y. Gao) Department of Mathematics, ETH Z\"urich, Rämistrasse 101, 8092 Z\"urich, Switzerland}
\email{yugaoy@ethz.ch, yugaomath@gmail.com}

\author{Lara Vrabac}
\address{(L. Varbac) Department of Mathematics, ETH Z\"urich, Rämistrasse 101, 8092 Z\"urich, Switzerland}
\email{lvrabac@student.ethz.ch}

\begin{document}

\begin{comment}
\begin{abstract}
    This work presents a new design to achieve the broadband absorption of acoustic waves in the low-frequency regime. Using a thin coating composed of a graded system of subwavelength acoustic resonators repeated periodically on a reflective surface, we show both analytically and numerically that the structure behaves as an equivalent absorbing surface over a broadband of frequencies with wavelengths much larger than the size of the unit cell. We also optimize the absorption coefficient in terms of the shapes of the resonators.
\end{abstract}
\end{comment}

\begin{abstract}
    This work presents a new design for broadband absorption of low-frequency acoustic waves using a thin coating made of subwavelength acoustic resonators arranged periodically on a reflective surface. We first study the associated  scattering problem and the corresponding subwavelength resonance problem, and then derive analytical approximations for the resonant frequencies and the reflection coefficient in terms of the periodic capacitance matrix in a half-space with a Dirichlet boundary condition. These approximations yield an effective macroscopic description of the coating via an impedance boundary condition and clarify the mechanism of superabsorption through an approximate coupling condition. Moreover, they lead to a reduced order model that enables efficient evaluation of the scattered waves over a frequency band and accelerates broadband absorption design. Building on this reduced order model, we develop a gradient based shape optimization method using shape derivatives of the resonant quantities to achieve broadband absorption. Numerical experiments demonstrate the broadband performance and the effectiveness of the proposed design procedure.
\end{abstract}

\maketitle

\vspace{0.5cm}
    

%\tableofcontents

%%%%%%%%%%%%%%%%%%%%%%%%%%%%%%%%%%%%%%%%%%%%%%%%%%%%%%%%%%%%%%%%%%

\section{Introduction}

Sound absorption plays a central role in architectural and room acoustics, environmental and industrial noise remediation, and underwater acoustics, where it supports radiated noise reduction and underwater acoustic stealth. It has therefore been the subject of sustained research and broad interest for decades. There is an increasing demand for broadband absorbers because many real world noise sources are spectrally rich and operate under variable conditions. However, achieving high absorption over a broad band, especially at low frequencies, remains challenging because intrinsic material dissipation is weak in this regime \cite{sheng2015,sheng2012} and because passive absorbers are constrained by fundamental thickness bandwidth limits \cite{yang2017sound}. In recent decades, the use of acoustic metamaterials has opened up new possibilities to solve this problem \cite{review2020}, offering practical ways to improve  noise reduction \cite{noisereduction} and, in particular, the accuracy of underwater acoustic research \cite{leroy}. The aim of this paper is to provide a mathematical and numerical framework for the analysis and optimal design of acoustic metascreens that exhibit a broadband low-frequency sound absorption.

\subsection{Problem statement and background}

In this subsection, we first describe the mathematical setting of the model and introduce the notation that will be used in the sequel. For a point $x \in \bR^d$ with $d = 2$ or $3$, we denote $x = (x_{\ell}, x_d)$, where $x_{\ell} \in \mathbb{R}^{d-1}$ represents the first $(d - 1)$ lattice coordinates, and $x_d \in \bR$ denotes the last coordinate. The upper half-space is given by $\mathbb{R}^d_{+} := \{ x \in \mathbb{R}^d \mid x_d > 0 \}$, and the reflective plane is defined as $\partial \mathbb{R}^{d}_{+}:= \{ x \in \mathbb{R}^d \mid x_d = 0 \}$. Let $\Lambda := L\mathbb{Z}^{d-1}$ be a $(d-1)$ dimensional square lattice with unit cell $Y = [-L/2, L/2]^{d-1}$ for a given $L>0$. Furthermore, for later use, we define the fundamental domains $Y_{\infty}$ and $Y_h$ as $Y_{\infty} := Y \times \mathbb{R}_{+}$ and $Y_h := Y \times (0, h)$, respectively. Finally, the dual lattice $\Lambda^*$ is defined by $\Lambda^* := (2\pi / L)\Lambda$.

In this study, we consider the periodic Helmholtz scattering problem \eqref{eq: model} for a coating mounted on a reflective (Dirichlet) surface. The coating is formed by periodically repeating a unit cell that contains $N$ subwavelength resonators with high-contrast material parameters compared to those of the background medium (cf. \Cref{fig:metascreen}). Let $D_1,\ldots,D_N \subset \mathbb{R}^d$ be disjoint, connected domains contained in the unit cell domain $Y_{\infty}$, and assume that each boundary $\partial D_i$ is of class $C^2$. We denote the union of resonators in one unit cell by $D := \bigcup_{i=1}^N D_i$. Given a lattice $\Lambda$, the periodically repeated $i$\textsuperscript{th} resonator and the fully periodic set of resonators are defined by $\mathcal{D}_i := \bigcup_{\xi\in\Lambda}\bigl(D_i+(\xi,0)\bigr)$ and $\mathcal{D} := \bigcup_{i=1}^N \mathcal{D}_i$.
\begin{figure}[htbp]
    \centering
    \includegraphics[width=0.9\textwidth]{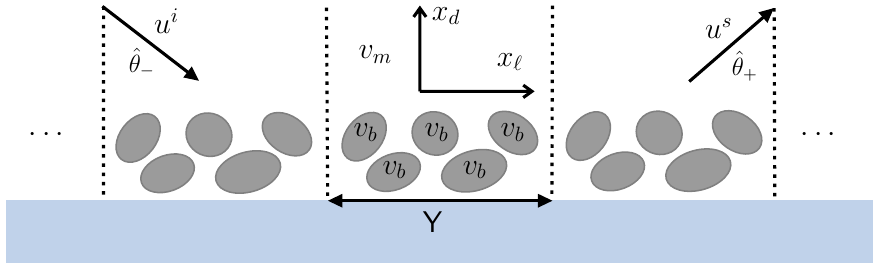}
    \caption{Illustration of a metascreen above a reflective surface.}
    \label{fig:metascreen}
\end{figure}

The material parameters (the density and the bulk modulus) inside the inclusions $D$ are denoted by $\rho_b$ and $\kappa_b$, whereas those in the surrounding medium are $\rho_m$ and $\kappa_m$. The corresponding wave speeds are given by $v_b = \sqrt{\kappa_b / \rho_b}$ and $v_m = \sqrt{\kappa_m / \rho_m}$. We assume that $\rho_m$ and $\kappa_m$ are given positive constants. The associated wave numbers and the dimensionless contrast parameter are defined as $k_m = \omega / v_m$, $k_b = \omega / v_b$, and $\delta = \rho_b / \rho_m$. We also assume that $\delta \ll 1$, with $v_b = \mathcal{O}(1)$ and $v_m = \mathcal{O}(1)$, ensuring a high contrast regime characteristic of subwavelength behavior. 

Let $\hat{\theta}_{\pm} \in \mathbb{R}^{d}$ be unit vectors satisfying  $|\hat{\theta}_{\pm}|=1$, defined by  $\hat{\theta}_{\pm}=(\hat{\theta}_{\ell}, \pm \hat{\theta}_{d})$ with $\hat{\theta}_{d} > 0$, where $\hat{\theta}_{\ell}$ denotes the components of $\hat{\theta}$ corresponding to the first $(d-1)$ lattice coordinates. The incident field is given by  $u^{i}(x)=e^{\mathrm{i}k_m \hat{\theta}_{-} \cdot x}$,  where $\hat{\theta}_{-}$ represents the incident direction. Since its last component is negative, the corresponding plane wave propagates downward, i.e., it is incident from above. The scattering problem of interest is governed by the system
\begin{equation}\label{eq: model}
    \begin{cases}
    \Delta u + k_m^2 u = 0 &  \text{ in } \mathbb{R}_+^d \setminus \mathcal{D}, \\[1mm]
    \Delta u + k_b^2 u = 0 & \text{ in } \mathcal{D},  \\[1mm]
    \left. u \right|_+ - \left. u \right|_- = 0 &  \text{ on } \partial \mathcal{D}, \\[1mm]
    \displaystyle \frac{1}{\rho_m} \left. \frac{\partial u}{\partial \nu} \right|_+ 
    - \frac{1}{\rho_b}\left. \frac{\partial u}{\partial \nu} \right|_- = 0 &
   \text{ on } \partial \mathcal{D}, \\[1mm]
     u = 0 & \text{ on } \partial \mathbb{R}^{d}_{+},\\[1mm]
    u^s(x) := u(x) - u^{i}(x) &\text{ satisfies the radiation condition as } x_d \to +\infty.
    \end{cases}
\end{equation}
Here, the subscripts $+$ and $-$ denote the limits of $u$ taken from the exterior and interior of $\mathcal{D}$, respectively. The outgoing radiation condition is formulated via the Rayleigh--Bloch expansion \cite{ammari2018mathematical,botten2013electromagnetic}. We seek a solution $u$ that is $\alpha$-quasiperiodic in $x_\ell$, with $\alpha = k_m \hat{\theta}_\ell$, i.e.,
\begin{equation*}
    u(x + (\xi, 0)) = e^{\mathrm{i}\alpha\cdot \xi}u(x), \quad \xi \in \Lambda.
\end{equation*}

The objective is to design an array of resonators arranged within a unit cell such that the reflectance  remains very small over a prescribed broadband range $[\omega_{\min},\omega_{\max}]$, which is equivalent to achieving near superabsorption across that band. This design task poses significant computational challenges. First, the governing scattering response reflectance  is frequency dependent, so broadband objectives require solving the forward problem at many discrete frequency samples. Second, since each unit cell comprises multiple resonators, multiple scattering interactions must be resolved, and the computational cost increases rapidly with the number of inclusions \cite{martin2006multiple}. Third, within an integral equation formulation, the evaluation of quasi-periodic Green’s functions can be expensive \cite{linton2010lattice,capolino2005efficient}, especially for repeated solvers across frequencies and geometries. Finally, the design task involves shape optimization over many geometric degrees of freedom, leading to a high dimensional, tightly coupled problem with substantial overall computational cost.

Recently, broadband optimal design of periodic materials has attracted increasing attention. For instance, researchers have proposed reduced basis strategies leveraging physics informed initial guesses \cite{gao2025optimal}, model order reduction techniques coupled with finite element and boundary element solvers \cite{chen2024broadband}, and multiscale approaches based on topological derivatives \cite{novotny2025topological}, all aimed at accelerating broadband computations. Despite significant advances in the experimental realization and numerical design of acoustic metamaterials, a rigorous mathematical understanding of subwavelength resonant scattering in high contrast metascreen has been developed primarily for configurations with one or two resonators \cite{ammari2021bound,ammari2017mathematical}. More recently, building on the capacitance matrix methodology developed in \cite{ammari2024functional}, we propose a reduced order model based on the capacitance matrix for the broadband design of periodic structures with multiple resonators.

\subsection{Main results and outline}

We study the associated resonance problem, that is, the configuration without an incident field $u^i=0$. Using the layer potential techniques developed in Section~\ref{sec:preliminaries}, the scattering solution admits the representation \eqref{eq: integral formulation of u}, which reduces the problem to the nonlinear resonance condition $\mathcal{F}(\omega)=0$ in \eqref{eq: simplied integral equation}. In the subwavelength regime, the scattering resonances can be characterized in terms of the periodic capacitance matrix introduced in \Cref{def:periodic_capacitance}. We then derive asymptotic expansions for the $N$ subwavelength resonant frequencies $\omega_j(\delta)$ with respect to the high contrast parameter $\delta$, and show that they satisfy $\Im\,\omega_j(\delta)<0$, as stated in \Cref{thm: asymptotic of resonance frequency}.

Our primary focus is on approximating the total field, with particular emphasis on the reflection coefficient (cf. \Cref{thm: approximation of total field}). We show that the total field admits a decomposition into three components: the incident wave, the reflected propagating wave, and an evanescent contribution that decays exponentially away from the structure. The reflected propagating wave is characterized by a reflection coefficient $r(\omega)$, which can be expressed as a superposition of multiple resonant modes. In particular, we derive an analytical formula for the reflection coefficient of this structure in \eqref{eq: asymptotic of reflection coefficient}, which is closely related to the expressions obtained in \cite{lanoy2018broadband,leroy2015superabsorption}.

We show that a thin coating patterned with subwavelength resonators can have a macroscopic impact on acoustic scattering over a broadband range of low frequencies. More precisely, the resulting total field is asymptotically equivalent to the solution of a half-space scattering problem subject to an impedance boundary condition (cf. \Cref{prop:impedance}). When $\lambda(\omega)$ is real and approaches $\lambda_j$, the reflection coefficient satisfies $r(\omega)\to 1$, so the structure behaves effectively as a sound-hard surface, that is, it approaches a Neumann boundary condition. In contrast, when $\lambda(\omega)$ is complex and there exists a frequency $\omega^{*}$ such that $\Re \lambda(\omega^{*}) \approx \lambda_j$ and $\Im \lambda(\omega^{*}) \approx \omega^{*}\lambda_{j,1}$, the reflection coefficient approaches zero, resulting in superabsorption near $\omega^{*}$. In this regime, the effective boundary condition becomes absorbing (cf. \Cref{rem:absorbing boundary condition}).

The approximate reflection coefficient not only provides a mathematical justification for the superabsorption effect, but also yields a reduced order model for broadband optimal design. Since the dominant computations in this reduced order model are frequency independent, they can be performed once and reused across the entire frequency band. The approximation remains accurate in the low frequency broadband regime. Broadband absorption requires structures that support multiple resonant peaks whose low reflection regions overlap within the frequency band of interest. With this in mind, we introduce two objective functionals, one based on resonant frequencies and the other based on the reflection coefficient. Because the resonant frequencies depend sensitively on the geometry, we derive shape derivatives for the resonant quantities and for the reflection coefficient, which enables gradient based optimal design. A detailed analysis is provided in \Cref{thm: shape derivative of eigenvalue}.

The paper is organized as follows. In the next section, we review the necessary preliminaries, including the quasi-periodic Green function, the layer potential operators, and the capacitance matrix associated with a half space resonance problem with a Dirichlet boundary condition. In Section~\ref{sec:approximation}, we analyze the subwavelength resonances and derive an asymptotic characterization of the reflection coefficient. Section~\ref{sec:broabband} is devoted to shape sensitivity analysis and the development of an optimal design algorithm for broadband absorption. Finally, extensive numerical results are presented in Section~\ref{sec:numerics} to illustrate the efficiency of the proposed design.

\section{Preliminaries}\label{sec:preliminaries}

In this section, we briefly review the quasi-periodic Green function, the layer potential operators, and the capacitance matrix that will be used in the subsequent analysis.

\subsection{Quasi-periodic Green's function}

In this subsection, we review the results on the quasi-periodic Green function established in \cite{ammari2017mathematical}. The quasi-periodic Green function $G_{\#}^{\alpha,k}(x)$ is defined as the solution to
\begin{align*}
    \Delta G_{\#}^{\alpha,k}(x) + k^2 G_{\#}^{\alpha,k}(x) 
    = \sum_{\xi \in \Lambda} \delta(x - (\xi, 0)) e^{\mathrm{i} \alpha \cdot \xi},
\end{align*}
where $\delta(x)$ denotes the Dirac delta distribution. It admits the explicit representation
\begin{align*}
    G_{\#}^{\alpha,k}(x) 
    = \sum_{\eta \in \Lambda^*} 
    \frac{e^{\mathrm{i} (\alpha + \eta) \cdot x_\ell} 
    e^{\mathrm{i} \sqrt{k^2 - |\alpha + \eta|^2} |x_d|}}
    {2 \mathrm{i} |Y| \sqrt{k^2 - |\alpha + \eta|^2}}.
\end{align*}
If $|\alpha| < k < \inf_{q \in \Lambda^* \setminus \{0\}} |\alpha + q|$, the Green function can be decomposed as
\begin{align}\label{eq: quasi-Green's fun}
    G_{\#}^{\alpha,k}(x)
    = \frac{e^{\mathrm{i} \alpha \cdot x_\ell} e^{\mathrm{i} \sqrt{k^2 - \alpha^2} |x_d|}}
    {2 \mathrm{i} \sqrt{k^2 - \alpha^2} |Y|}
    - \sum_{\eta \in \Lambda^* \setminus \{0\}}
    \frac{e^{\mathrm{i} (\alpha + \eta) \cdot x_\ell}
    e^{-\sqrt{|\alpha + \eta|^2 - k^2} |x_d|}}
    {2 |Y| \sqrt{|\alpha + \eta|^2 - k^2}}.
\end{align}
In the special case $\alpha = k = 0$, we denote the Green function by
\begin{align*}
    G_{\#}^{0,0}(x)
    = \frac{|x_d|}{2|Y|}
    - \sum_{\eta \in \Lambda^* \setminus \{0\}}
    \frac{e^{\mathrm{i} \eta \cdot x_\ell} e^{-|\eta| |x_d|}}
    {2 |Y| |\eta|},
\end{align*}
which serves as the fundamental solution to
\begin{align*}
    \Delta G_{\#}^{0,0}(x)
    = \sum_{\xi \in \Lambda} \delta(x - (\xi, 0)).
\end{align*}

We now introduce the low-frequency scaling, $\alpha = \omega a$ and  $\sqrt{k^2 - |\alpha|^2} = \omega \tau$, where $a$ is a fixed vector and $\tau$ is a scalar. Following \cite{ammari2017mathematical}, the Green function admits the asymptotic expansion with respect to $\omega$,
\begin{align}\label{eq: asymptotic of Green's fun}
    G_{\#}^{\alpha,k}(x) 
    = \sum_{n=-1}^{+\infty} \omega^n G^{a,\tau}_{\#,n}(x).
\end{align}
In particular, for $n = -1, 0,$ and $1$, we have the explicit formula
\begin{align*}
    G^{a,\tau}_{\#,-1}(x) = \frac{1}{2 \mathrm{i} \tau |Y|}, \
    G^{a,\tau}_{\#,0}(x) = G_{\#}^{0,0}(x) + \frac{a \cdot x_\ell}{2\tau|Y|}, \
    G^{a,\tau}_{\#,1}(x) = \frac{\mathrm{i}(a \cdot x_\ell + \tau|x_d|)^2}{4\tau|Y|} 
    + \mathrm{i} a \cdot g_{\#,1}(x),
\end{align*}
where $g_{\#,1}$ is a real vector-valued function independent of $a$ and $\tau$. Moreover, it satisfies the properties $g_{\#,1}(-x_\ell, x_d)=-g_{\#,1}(x_\ell, x_d)$  and $g_{\#,1}(x_\ell, -x_d)=g_{\#,1}(x_\ell, x_d)$. To derive the higher order terms, we observe that
\begin{align*}
    \left( \Delta + \frac{\omega^2}{v^2} \right) G_{\#}^{\alpha,k}(x) 
    = \sum_{\xi \in \Lambda} e^{\mathrm{i} \omega a \cdot \xi} \delta(x - (\xi, 0)) 
    = \sum_{n=0}^{+\infty} \omega^n 
    \sum_{\xi \in \Lambda} \frac{(\mathrm{i} a \cdot \xi)^n}{n!} 
       \delta(x - (\xi, 0)).
\end{align*}
Substituting the expansion \eqref{eq: asymptotic of Green's fun} into the above equation and matching coefficients of equal powers of $\omega$, we obtain the recurrence relations
\begin{align*}
    %\Delta G^{a,\tau}_{\#,0}(x) = \sum_{\xi \in \Lambda} \delta(x - (\xi, 0)), \
    \Delta G^{a,\tau}_{\#,n}(x) + \frac{1}{v^2} G^{a,\tau}_{\#,n-2}(x) 
    = \sum_{\xi \in \Lambda} \frac{(\mathrm{i} a \cdot \xi)^n}{n!} 
       \delta(x - (\xi, 0)), \ n \geq 1.
\end{align*}

Next, we introduce the quasi-periodic Green function associated with the sound-soft (Dirichlet) boundary condition, defined by
\begin{align*}
    G_s^{\alpha,k}(x, y) = G_{\#}^{\alpha,k}(x - y) - G_{\#}^{\alpha,k}(x - y^*), 
\end{align*}
where $y^* = (y_\ell, -y_d)$.  
Using the asymptotic expansion \eqref{eq: asymptotic of Green's fun}, we similarly obtain
\begin{align}\label{eq: asymptotic of soft Green's fun}
    G_{s}^{\alpha,k}(x, y) 
    = \sum_{n=-1}^{+\infty} \omega^n\, G_{s,n}^{a,\tau}(x, y),
\end{align}
with $G_{s,n}^{a,\tau}(x, y) = G_{\#,n}^{a,\tau}(x - y) - G_{\#,n}^{a,\tau}(x - y^*)$. Then we have
\begin{gather*}
    G_{s,-1}^{a,\tau}(x, y) = 0, 
    \quad G_{s,0}^{a,\tau}(x, y) = G_s^{0,0}(x, y), 
\end{gather*}
and hence $G_{s,0}^{a,\tau}(x, y)$ is independent of $a$ and $p$.  Furthermore, $G_{s,1}^{a,\tau}(x, y)$ can be written as
\begin{equation}\label{eq: Gs1}
    \begin{aligned}
        G_{s,1}^{a,\tau}(x, y) 
        = \frac{\mathrm{i}}{|Y|}\, \hat{G}_{s,1}^{a}(x, y) 
        - \frac{\mathrm{i} \tau}{|Y|} x_d y_d,
    \end{aligned}
\end{equation}
where $\hat{G}_{s,1}^{a}(x, y) =
a \cdot (x_\ell - y_\ell)\big(|x_d - y_d| - |x_d + y_d|\big)/2 
        + \, a \cdot g_{s,1}(x, y)\in \mathbb{R}$ depends only on $a$ and satisfies the anti-symmetry property $\hat{G}_{s,1}^{a}(x, y) = -\hat{G}_{s,1}^{a}(y, x)$.

\subsection{Layer potential techniques}

With the quasi-periodic Green’s function introduced in the previous subsection, we now define the associated layer potentials and boundary integral operators. The quasi-periodic single-layer potential $\mathcal{S}^{\alpha,k}_{D,s}$ associated with the sound-soft boundary condition is defined by
\begin{align*}
    \mathcal{S}^{\alpha, k}_{D, s}[\psi](x) 
    := \int_{\partial D} G^{\alpha, k}_{s}(x - y)\, \psi(y)\, \mathrm{d}\sigma(y), 
    \quad x \in \mathbb{R}^d,
\end{align*}
where $\psi \in L^2(\partial D)$ is the density function and $\mathrm{d}\sigma$ denotes the surface measure on $\partial D$.  The single-layer potential satisfies the following jump relations across the boundary
\begin{align}\label{eq: jump_formula}
    \left. \frac{\partial}{\partial \nu} 
    \mathcal{S}^{\alpha, k}_{D, s}[\psi] \right|_{\pm}(x)
    = \left( \pm \frac{1}{2}\mathcal{I}
    + (\mathcal{K}^{-\alpha, k}_{D, s})^* \right)[\psi](x), 
    \quad x \in \partial D,
\end{align}
where $\nu$ denotes the outward unit normal vector on $\partial D$. The associated boundary integral operator $(\mathcal{K}^{-\alpha,k}_{D,s})^*$ is given by
\begin{align*}
    (\mathcal{K}^{-\alpha, k}_{D, s})^*[\psi](x)
    := \int_{\partial D} 
    \frac{\partial G^{\alpha, k}_{s}(x - y)}{\partial \nu(x)}\, 
    \psi(y)\, \mathrm{d}\sigma(y),
    \quad x \in \partial D.
\end{align*}

According to the asymptotic expansion of the quasi-periodic Green’s function $G_{s}^{\alpha,k}(x,y)$ in~\eqref{eq: asymptotic of soft Green's fun}, the integral operators $\mathcal{S}^{\alpha,k}_{D,s}$ and $(\mathcal{K}^{-\alpha,k}_{D,s})^*$ admit the expansions
\begin{align}\label{eq: asymptotic of boundary operator}
    \mathcal{S}^{\alpha, k}_{D, s}
    = \sum_{n=-1}^{+\infty} \omega^n\, \mathcal{S}^{a,\tau}_{D, s, n},
    \quad
    (\mathcal{K}^{-\alpha, k}_{D, s})^*
    = \sum_{n=-1}^{+\infty} \omega^n\, (\mathcal{K}^{-a,\tau}_{D, s, n})^*,
\end{align}
where the operators $\mathcal{S}^{a,\tau}_{D,s,n}$ and $(\mathcal{K}^{-a,\tau}_{D,s,n})^*$ are associated with the kernel $G_{s,n}^{a,\tau}$. From the explicit expression of $G_{s,n}^{a,\tau}$, we have $\mathcal{S}^{a,\tau}_{D, s, -1}=0$ and $(\mathcal{K}^{-a,\tau}_{D, s, -1})^* =0$. Moreover, since $\mathcal{S}^{a,\tau}_{D,s,0}$ and $(\mathcal{K}^{-a,\tau}_{D,s,0})^*$ are independent of $a$ and $\tau$, we denote $\mathcal{S}^{0,0}_{D,s} := \mathcal{S}^{a,\tau}_{D,s,0}$ and $(\mathcal{K}^{0,0}_{D,s})^* := (\mathcal{K}^{-a,\tau}_{D,s,0})^*$, for simplicity and without ambiguity. Using the recurrence relations together with the divergence theorem, we obtain the following useful identities.
\begin{lemma}\label{lem: integral of Ksn}
For any $\psi \in L^2(\partial D)$, the following identities hold:
\begin{enumerate}
    \item[(i)] $\displaystyle 
        \int_{\partial D_i} 
        \left(-\frac{1}{2}\mathcal{I}
        + (\mathcal{K}^{0, 0}_{D, s})^*\right)[\psi]\, \mathrm{d}\sigma = 0,\quad  n=0$;
    \item[(ii)] $\displaystyle 
        \int_{\partial D_i} 
        (\mathcal{K}^{-a,\tau}_{D, s, n})^*[\psi]\, \mathrm{d}\sigma
        = -\frac{1}{v^2}
        \int_{D_i} \mathcal{S}^{a,\tau}_{D, s, n-2}[\psi]\, \mathrm{d}\sigma, 
        \quad n \geq 1.$
\end{enumerate}
\end{lemma}
Analogously to the boundary integral operators associated with the Laplace Green’s function in free space \cite{ammari2018minnaert}, similar results remain valid in the periodic setting in a half space resonance problem with a Dirichlet boundary condition.
\begin{lemma}\label{lem: properties of NP operator}
Consider a system of $N$ subwavelength resonators. Then, we have
\begin{enumerate}
    \item[(i)] The Laplace periodic single layer operator with sound-soft boundary condition
    $\mathcal{S}^{0, 0}_{D, s}: L^2(\partial D) \to H^1(\partial D)$ 
    is invertible;
    
    \item[(ii)] The nullspace satisfies
    \[
        \operatorname{Ker}\!\left(-\frac{1}{2}\mathcal{I} 
        + (\mathcal{K}^{0, 0}_{D, s})^*\right)
        = \operatorname{span}\{\psi_{1}, \ldots, \psi_{N}\},
    \]
    where 
    $\psi_{j} := (\mathcal{S}^{0, 0}_{D, s})^{-1}[\chi_{\partial D_j}]$, 
    and $\chi_{\partial D_j}$ denotes the characteristic function of 
    $\partial D_j$;
    
    \item[(iii)] The following direct sum decomposition holds:
    \begin{align*}
        L^2(\partial D)
        = L_0^2(\partial D)
        \oplus 
        \operatorname{Ker}\!\left(-\frac{1}{2}\mathcal{I} 
        + (\mathcal{K}^{0, 0}_{D, s})^*\right),
    \end{align*}
    where 
    \[
        L_0^2(\partial D)
        := \left\{
            \psi \in L^2(\partial D)
            \;\bigg|\;
            \int_{\partial D_i} \psi\, \mathrm{d}\sigma = 0,
            \; 1 \leq i \leq N
        \right\}.
    \]
    Furthermore, 
    $-\dfrac{1}{2}\mathcal{I} + (\mathcal{K}^{0, 0}_{D, s})^*$ 
    is invertible as an operator 
    $L_0^2(\partial D) \to L_0^2(\partial D)$.
\end{enumerate}
\end{lemma}

\begin{proof}
We provide the proof of part~(i); parts~(ii) and~(iii) follow from the corresponding arguments in the free space setting (cf. \cite{ammari2018minnaert}).  We first observe that $\mathcal{S}^{0,0}_{D,s}$ is a Fredholm operator of index zero. Therefore, to prove invertibility, it suffices to establish injectivity. Let $u = \mathcal{S}^{0,0}_{D,s}[\psi]$ and choose $h$ so that $x_d < h$ for all $x \in D$. Applying integration by parts over the inclusion $D$ and over the truncated cell domain $Y_h \setminus \overline{D}$, we obtain
\begin{align*}
    \int_{D} |\nabla u|^2\, \mathrm{d}\sigma 
        = \int_{\partial D} 
           \overline{u} 
           \frac{\partial u}{\partial \nu}\Big|_{-}
           \mathrm{d}\sigma, \
    \int_{Y_h \setminus \overline{D}} |\nabla u|^2\, \mathrm{d}\sigma 
        = -\int_{\partial D} 
           \overline{u} 
           \frac{\partial u}{\partial \nu}\Big|_{+}
           \mathrm{d}\sigma
           + \int_{Y \times \{h\}} 
             \overline{u} 
             \frac{\partial u}{\partial \nu}\, 
             \mathrm{d}\sigma.
\end{align*}
Here, we have used the fact that $u$ is quasi-periodic on $\partial Y \times (0, h)$. 
Combining these two expressions yields
\begin{align*}
    \int_{Y_h} |\nabla u|^2\, \mathrm{d}\sigma 
    = -\int_{\partial D} \overline{u}\, \psi\, \mathrm{d}\sigma
      + \int_{Y \times \{h\}} 
        \overline{u} 
        \frac{\partial u}{\partial \nu}\, \mathrm{d}\sigma.
\end{align*}
For $x_d > h$, we have
\begin{align*}
    u(x) = \frac{-1}{|Y|} 
      \int_{\partial D} y_d \psi\, \mathrm{d}\sigma
      - \sum_{\eta \in \Lambda^* \setminus \{0\}} \left(
      \int_{\partial D} 
      \frac{e^{-\mathrm{i} \eta \cdot y_\ell}\left(e^{|\eta| y_d} - e^{-|\eta| y_d}\right)}{2|Y||\eta|}
      \psi\, \mathrm{d}\sigma\right) e^{\mathrm{i} \eta \cdot x_\ell} e^{-|\eta| x_d}.
\end{align*}
Taking the limit as $h \to +\infty$, we obtain
\begin{align*}
    \int_{Y_{\infty}} |\nabla u|^2\, \mathrm{d}\sigma
    = -\int_{\partial D} \overline{u}\, \psi\, \mathrm{d}\sigma.
\end{align*}
Now, assume that $u = 0$ on $\partial D$. 
Then $u = 0$ in both $D$ and $Y_{\infty} \setminus \overline{D}$. 
It follows that 
$2\psi 
= \frac{\partial u}{\partial \nu}\!\Big|_{+} 
 - \frac{\partial u}{\partial \nu}\!\Big|_{-}
= 0$,
which completes the proof.
\end{proof}

\subsection{Periodic capacitance matrix}

Based on \Cref{lem: properties of NP operator}, we define the corresponding capacitance matrix and basis functions in the periodic setting with sound-soft boundary conditions.

\begin{definition}\label{def:periodic_capacitance}
For a system of $N$ resonators, the periodic capacitance matrix 
$C = (C_{ij}) \in \mathbb{R}^{N \times N}$ is defined by
\begin{align}\label{eq: def of capacity matrix by BIE}
    C_{ij} 
    := -\int_{\partial D_i} \psi_{j}\, \mathrm{d}\sigma,
    \quad i, j = 1, \ldots, N,
\end{align}
where $\psi_{j}$ is the basis function defined in \cref{lem: properties of NP operator}.
\end{definition}
\begin{remark}
     We also have the equivalent definition of the capacitance matrix \eqref{eq: def of capacity matrix by BIE} in variational form where $C_{ij}$ can also be expressed as
\begin{align}\label{eq: def of capacity matrix by PDE}
    C_{ij}(D) := \int_{Y_{\infty} \setminus \overline{D}} \nabla v_i \cdot \nabla v_j \, dx=\int_{\partial D_i} \frac{\partial v_j}{\partial \nu}\Big{|}_+d\sigma=\int_{\partial D_j} \frac{\partial v_i}{\partial \nu}\Big{|}_+d\sigma,
    \quad i,j = 1,\ldots, N,
\end{align}
where $v_i$ $(1 \le i \le N)$ are the solutions to the exterior periodic Laplace problem 
\begin{equation}\label{eq: PDE solution of capacity form}
\begin{cases}
\Delta v = 0 & \text{in } Y_{\infty} \setminus \overline{D}, \\[1mm]
v = v_{\partial D}  & \text{on } \partial D, \\[1mm]
v = 0  & \text{on } \partial \mathbb{R}^d_{+}, \\[1mm]
v(x + (\xi,0)) = v(x)  &  \forall \xi \in \Lambda, \\[1mm]
v(x)- v_{\infty}=O(e^{-|x_d|})  & \text{ as } x_d \to +\infty,
\end{cases}
\end{equation}
with the boundary condition $v_i = \chi_{\partial D_i}$ on $\partial D$. Note that the solution of $v_i$ is given by $v_i=\mathcal{S}^{0, 0}_{D, s}[\psi_i]$, where $\psi_i$ is defined in \cref{lem: properties of NP operator}. According to \eqref{eq: def of capacity matrix by PDE},  one can show that $C$ is symmetric and positive definite following similar arguments to those in \cite{ammari2024functional}. 

%Moreover, note that $C$ has $0$ as an eigenvalue with the associated eigenvector given by $(1,\dots, 1)^\top$.
\end{remark}

By the positive semidefiniteness of the capacitance matrix, the following result holds.
\begin{lemma} \label{lem:spC}
There exist $N$ eigenvalues $\{\lambda_j\}_{1 \leq j \leq N}$ and associated eigenvectors $\{u_j\}_{1 \leq j \leq N}$ that satisfy the generalized eigenvalue problem
\begin{align}\label{eq: eigenvalue problem of C}
    C u_j = \lambda_j V u_j, \qquad 1 \leq j \leq N,
\end{align}
where $V = \operatorname{diag}(|D_1|, \dots, |D_N|)$ denotes the diagonal volume matrix.  
The eigenvectors $\{u_j\}$ form an orthonormal basis of $\mathbb{C}^N$ with respect to the $V$ inner product, that is, $u_i^{\top} V u_j = \delta_{ij}$, where the superscript $\top$ denotes the transpose.
\end{lemma}

To avoid technical complications arising from possible eigenvalue multiplicities, 
we make the following simplifying assumption.

\begin{assumption}\label{asmp: eigenvalue}
All eigenvalues of~\eqref{eq: eigenvalue problem of C} are simple, 
meaning that each eigenvalue has algebraic multiplicity equal to one.
\end{assumption}

Next, we introduce a projection operator associated with the normalized basis, which will play a key role in the subsequent analysis.
\begin{definition}\label{def: projection}
Let $\{\hat{\psi}_{j}\}_{1 \le j \le N}$ be the basis of the subspace
$\operatorname{Ker}\left(-\frac{1}{2}\mathcal{I} + (\mathcal{K}^{0, 0}_{D, s})^*\right)$ defined by
\begin{align*}
    \hat{\psi}_{i} 
    := -\sum_{j=1}^{N} (C^{-1})_{ij}\, \psi_{j},
    \qquad 1 \le i \le N,
\end{align*}
where $C$ denotes the capacitance matrix in \eqref{eq: def of capacity matrix by BIE}.  We define the vectors $\hat{m}, m\in\mathbb{R}^N$ by
\begin{equation}\label{def:mj}
\hat{m}_i := \int_{\partial D} x_d\,\hat{\psi}_i(x)\,\mathrm{d}\sigma(x),\quad
m_i := -\int_{\partial D} x_d\,\psi_i(x)\,\mathrm{d}\sigma(x),\quad 1\le i\le N.
\end{equation} 
Finally, we define the projection operator
\begin{align*}
\mathcal{P}[\psi]
&:= \sum_{i=1}^{N}
\left( \int_{\partial D_i} \psi\, \mathrm{d}\sigma \right)
\hat{\psi}_{i}.
\end{align*}
\end{definition}

With these definitions, the basis $\{\hat{\psi}_j\}_{j=1}^N$ is normalized in the sense that $\int_{\partial D_i}\hat{\psi}_j\,\mathrm{d}\sigma=\delta_{ij}$ for $1\le i,j\le N$, the projection satisfies $\mathcal{P}[\hat{\psi}_i]=\hat{\psi}_i$ for $i=1,\dots,N$, and the moment vectors satisfy $m=C\hat{m}$.

\section{Approximation of the total field}\label{sec:approximation}

In this section, we investigate the periodic scattering problem \eqref{eq: model} using an integral formulation. Our objective is to analyze the resonance frequencies and the reflection coefficient in the subwavelength regime.

\subsection{Integral formulation}

We first note that, in the absence of resonators, the reflected field is 
given by $u^{r}(x) = -e^{\mathrm{i}k_m \hat{\theta}_{+} \cdot x}$. 
For convenience, we introduce $$\tilde{u}^{i}(x):=e^{k_m \hat{\theta}_{-} \cdot x}-e^{k_m \hat{\theta}_{+} \cdot x}=-2\mathrm{i}\sin(k_m\hat{\theta}_{d} x_{d})e^{\mathrm{i}k_m x_{\ell} \cdot \hat{\theta}_{\ell}}=
-2\mathrm{i}\sin(\omega \tau_m x_{d})e^{\mathrm{i}\omega a \cdot x_{\ell}},$$ which satisfies the sound-soft boundary condition on  $\partial \mathbb{R}^{d}_{+}$. Here, we have set $\alpha = k_m \hat{\theta}_{\ell} = \omega a$, $a = \hat{\theta}_{\ell}/v$, and $\tau_m = \hat{\theta}_{d}/v_m$. The total field $u$ solving \eqref{eq: model} admits the representation
\begin{equation}\label{eq: integral formulation of u}
    u(x) = \begin{cases}
    \begin{aligned}
        &\mathcal{S}^{\alpha, k_b}_{D, s}[\phi](x),& &x\in \mathcal{D},\\ 
        &\mathcal{S}^{\alpha, k_m}_{D, s}[\varphi](x)+\tilde{u}^{i}(x),& &x\in \mathbb{R}^{d}\setminus \overline{\mathcal{D}},
    \end{aligned}
    \end{cases}
\end{equation}
where $(\phi,\varphi)\in L^{2}(\partial D)\times L^{2}(\partial D)$ are the unknown boundary densities. Imposing the transmission conditions across $\partial D$ and applying the jump relations for the single-layer potential lead to the coupled system
\begin{equation}\label{eq: integral scattering problem}
	\begin{cases}
    \begin{aligned}
        &\mathcal{S}^{\alpha, k_b}_{D, s}[\phi] - \mathcal{S}^{\alpha, k_m}_{D, s}[\varphi] = \tilde{u}^{i}, \\
		&\left(-\dfrac{1}{2} \mathcal{I} + (\mathcal{K}^{\alpha, k_b}_{D, s})^*\right)[\phi]
		- \delta\left(\dfrac{1}{2} \mathcal{I} + (\mathcal{K}^{\alpha, k_m}_{D, s})^*\right)[\varphi]
		= \delta \dfrac{\partial \tilde{u}^{i}}{\partial \nu}.
    \end{aligned}
	\end{cases}
\end{equation}
Note that since $\mathcal{S}^{\alpha, k_m}_{D, s}$ is invertible \cite{ammari2024functional}, we  may express $\varphi=(\mathcal{S}^{\alpha, k_m}_{D, s})^{-1}(\mathcal{S}^{\alpha, k_b}_{D, s})[\phi]-(\mathcal{S}^{\alpha, k_m}_{D, s})^{-1}[\tilde{u}^{i}]$. Substituting this expression into \eqref{eq: integral scattering problem}, we obtain the boundary integral equation for $\phi$
\begin{align}\label{eq: simplied integral equation}
    \mathcal{A}(\omega,\delta )[\phi] = \delta \mathcal{F}(\omega).
\end{align}
Here, the operator $\mathcal{A}(\omega, \delta)$ and the right hand term $\mathcal{F}(\omega)$ are defined by
\begin{align*}
    \mathcal{A}(\omega,\delta )[\phi]=&\left(-\dfrac{1}{2} \mathcal{I} + (\mathcal{K}^{-\alpha, k_b}_{D, s})^*-\delta \left(\dfrac{1}{2} \mathcal{I} + (\mathcal{K}^{-\alpha, k_m}_{D, s})^*\right)(\mathcal{S}^{\alpha, k_m}_{D, s})^{-1}(\mathcal{S}^{\alpha, k_b}_{D, s}) \right)[\phi], \\
   \mathcal{F}(\omega) =& \dfrac{\partial \tilde{u}^{i}}{\partial \nu}- \left(\dfrac{1}{2} \mathcal{I} + (\mathcal{K}^{-\alpha, k_m}_{D, s})^*\right)(\mathcal{S}^{\alpha, k_m}_{D, s})^{-1}[\tilde{u}^{i}].
\end{align*}

Next, we solve the boundary integral equation \eqref{eq: simplied integral equation} using asymptotic analysis. From the expression of $\tilde{u}^i(x)$, we obtain
\begin{equation*}
    \begin{aligned}
        \tilde{u}^i(x)
        = -2\mathrm{i} \tau_m\, x_d\, \omega + O(\omega^2), \quad
        \frac{\partial \tilde{u}^i}{\partial \nu}(x)
        = -2\mathrm{i} \tau_m\, \nu_d\, \omega + O(\omega^2),
    \end{aligned}
\end{equation*}
where $\nu_d$ denotes the $d$\textsuperscript{th} component of the unit normal $\nu$. It follows that the right hand term $\mathcal{F}(\omega)$ admits the asymptotic expansion
\begin{align}\label{eq: asymptotic of F(omega)}
   \mathcal{F}(\omega)
   = -2\mathrm{i} \tau_m\, \nu_d\, \omega
    +2\mathrm{i} \tau_m\, \omega\left(\frac{1}{2}\mathcal{I}+ (\mathcal{K}^{0, 0}_{D, s})^*\right)(\mathcal{S}^{0, 0}_{D, s})^{-1}[x_d]+ O(\omega^2).
\end{align}
Using the asymptotic expansion of the boundary operators in \eqref{eq: asymptotic of boundary operator}, the operator $\mathcal{A}(\omega, \delta)$ admits the expansion
\begin{align}\label{eq: expansion of A}
    \mathcal{A}(\omega,\delta )
    =& \mathcal{A}_0+\omega \mathcal{A}_{1,0}+\omega^2 \mathcal{A}_{2,0}+\omega^3 \mathcal{A}_{3,0}+\delta \mathcal{A}_{0,1}+\omega \delta \mathcal{A}_{1,1}+O(\omega^4)+O(\delta \omega^2),
\end{align}
where each coefficient operator is given explicitly by
\begin{gather*}
    \mathcal{A}_0=-\dfrac{1}{2} \mathcal{I} + (\mathcal{K}^{0, 0}_{D, s})^*, \ \mathcal{A}_{1,0}=  (\mathcal{K}^{-a,\tau_b}_{D, s,1})^*, \ \mathcal{A}_{2,0}=  (\mathcal{K}^{-a,\tau_b}_{D, s,2})^*,\ \mathcal{A}_{3,0}=  (\mathcal{K}^{-a,\tau_b}_{D, s,3})^*, \\
    \mathcal{A}_{0,1} = -\dfrac{1}{2} \mathcal{I} - (\mathcal{K}^{0, 0}_{D, s})^*, \ \mathcal{A}_{1,1} = \left(\dfrac{1}{2} \mathcal{I} + (\mathcal{K}^{0, 0}_{D, s})^*\right)(\mathcal{S}^{0, 0}_{D, s})^{-1}\left(\mathcal{S}^{a,\tau_m}_{D, s, 1}-\mathcal{S}^{a,\tau_b}_{D, s, 1}\right)-(\mathcal{K}^{-a,\tau_m}_{D, s,1})^*,
\end{gather*}
with $\tau_b:= \sqrt{k_b^2-|\alpha|^2}/\omega$. According to \cref{lem: integral of Ksn}, the following integral identities hold.
\begin{lemma}\label{lem: integral of An}
For any $\psi\in L^2(\partial D)$, for $1\leq i\leq N$, we have
    \begin{gather*}
        \int_{\partial D_i} \mathcal{A}_0[\psi] d\sigma=0,\quad \int_{\partial D_i} \mathcal{A}_{1,0}[\psi] d\sigma=0, \quad \int_{\partial D_i} \mathcal{A}_{0,1}[\psi] d\sigma=-\int_{\partial D_i} \psi,  \\
         \int_{\partial D_i} \mathcal{A}_{2,0}[\psi] d\sigma=-\frac{1}{v_b^2}\int_{D_i} \mathcal{S}^{0, 0}_{D, s}[\psi]dx,\quad \int_{\partial D_i}\mathcal{A}_{3,0}[\psi] d\sigma=-\frac{1}{v_b^2}\int_{D_i} \mathcal{S}^{a,\tau_b}_{D, s, 1}[\psi]dx, \\
         \int_{\partial D_i} \mathcal{A}_{1,1}[\psi] d\sigma=\dfrac{\mathrm{i}(\tau_m-\tau_b)}{|Y|} m_i\int_{\partial D} y_d\psi d\sigma.
    \end{gather*}
\end{lemma}
\begin{proof}
    We only prove the last identity, since the others follow directly from \cref{lem: integral of Ksn}. From the expression of $G^{a,\tau}_{s, 1}$ in \eqref{eq: Gs1}, we have
    \begin{align*}
            G^{a,\tau_m}_{s, 1} - G^{a,\tau_b}_{s, 1}
    = -\frac{\mathrm{i}(\tau_m-\tau_b)}{|Y|}\, x_d y_d,
    \end{align*}
which implies
\begin{align*}
        \mathcal{S}^{a,\tau_m}_{D, s, 1}[\psi]
    - \mathcal{S}^{a,\tau_b}_{D, s, 1}[\psi]
    = -\frac{\mathrm{i}(\tau_m-\tau_b)}{|Y|}\, x_d
      \int_{\partial D} y_d \psi\, \mathrm{d}\sigma.
\end{align*}
Applying \cref{lem: integral of Ksn}, we obtain
\begin{align*}
        \int_{\partial D_i} \mathcal{A}_{1,1}[\psi] \, \mathrm{d}\sigma
        = -\frac{\mathrm{i}(\tau_m-\tau_b)}{|Y|}
            \int_{\partial D_i}
              (\mathcal{S}^{0, 0}_{D, s})^{-1}[x_d]
              \, \mathrm{d}\sigma
            \int_{\partial D} y_d \psi\, \mathrm{d}\sigma.
\end{align*}
Using \cref{def: projection}, we derive the useful identity
\begin{align}\label{eq: integral of Sinverse_xd}
\int_{\partial D_i} 
    \left(\mathcal{S}^{0, 0}_{D, s}\right)^{-1}[x_d]\, \mathrm{d}\sigma
= \int_{\partial D} 
     x_d\, 
     \left(\mathcal{S}^{0, 0}_{D, s}\right)^{-1}
     [\chi_{\partial D_i}]\, \mathrm{d}\sigma 
= \int_{\partial D} 
     x_d\, \psi_{i}\, \mathrm{d}\sigma
 =- m_i. 
\end{align}
The desired result now follows from \eqref{eq: integral of Sinverse_xd}.
\end{proof}

In order to solve \eqref{eq: simplied integral equation}, we note from the expansion of $\mathcal{A}(\omega, \delta)$
in \eqref{eq: expansion of A} that the leading-order operator $\mathcal{A}_0$ is not invertible. To overcome this difficulty, we introduce the modified operator $\widetilde{\mathcal{A}}_{0}:= \mathcal{A}_0 + \mathcal{P}$. Its inverse is explicitly given by
\begin{align}\label{eq: inverse of A0}
    \widetilde{\mathcal{A}}_{0}^{-1}[\psi]
    = \left(-\dfrac{1}{2}\mathcal{I}
        + (\mathcal{K}^{0, 0}_{D, s})^*\right)^{-1}
        [\psi - \mathcal{P}[\psi]]
      + \mathcal{P}[\psi],
      \quad
      \psi \in L^2(\partial D).
\end{align}
Using $\widetilde{\mathcal{A}}_{0}$, equation~\eqref{eq: simplied integral equation} can be rewritten as
\begin{align}\label{eq: modified simplified equation}
    \big(\widetilde{\mathcal{A}}_{0}
        + \mathcal{B}(\omega, \delta)
        - \mathcal{P}\big)[\phi]
    = \delta\, \mathcal{F}(\omega),
\end{align}
where $\mathcal{B}(\omega, \delta):= \mathcal{A}(\omega, \delta) - \mathcal{A}_0$. To determine the boundary density $\phi$, we decompose it into its projected and complementary parts:
\begin{align*}
        \phi = \phi_0 + \phi_1,
    \qquad
    \phi_0 := \mathcal{P}[\phi],
    \quad
    \phi_1 := \phi - \mathcal{P}[\phi].
\end{align*}
Substituting this decomposition into \eqref{eq: modified simplified equation} and applying $\big(\widetilde{\mathcal{A}}_{0}+ \mathcal{B}(\omega, \delta)\big)^{-1}$ yields
\begin{align}
    \phi_0 + \phi_1
      - \big(\widetilde{\mathcal{A}}_{0}
              + \mathcal{B}(\omega, \delta)\big)^{-1}[\phi_0]
    = \delta\, 
      \big(\widetilde{\mathcal{A}}_{0}
          + \mathcal{B}(\omega, \delta)\big)^{-1}
        \mathcal{F}(\omega).
\end{align}

Assume that the projected component admits the expansion $\phi_0 = \sum_{j=1}^{N} z_j\, \hat{\psi}_j$, where the coefficients $z_j$ are to be determined. Integrating the above equation over the boundary of each resonator $\partial D_i$ gives
\begin{align*}
    \sum_{j=1}^{N}
      \Big(
        \delta_{ij}
        - \int_{\partial D_i}
            \mathcal{H}(\omega, \delta)[\hat{\psi}_j]
            \, \mathrm{d}\sigma
      \Big) z_j
    = \delta
      \int_{\partial D_i}
        \mathcal{H}(\omega, \delta)\,
        \mathcal{F}(\omega)
        \, \mathrm{d}\sigma,
\end{align*}
where $\mathcal{H}(\omega, \delta):= \big(\widetilde{\mathcal{A}}_{0}+ \mathcal{B}(\omega, \delta)\big)^{-1}$. Consequently, the coefficient vector $z = (z_1, \dots, z_N)^{\mathsf{T}}$ satisfies the finite dimensional system
\begin{align}\label{eq: finite dimensional equation}
    A(\omega, \delta)\, z
    = \delta\, F(\omega, \delta),
\end{align}
where the matrix $A(\omega, \delta)\in \mathbb{C}^{N\times N}$ and the vector $F(\omega, \delta) \in \mathbb{C}^N$ are defined by
\begin{align}\label{eq: def of A and F}
    A_{ij}(\omega, \delta)
      := \delta_{ij}
        - \int_{\partial D_i}
            \mathcal{H}(\omega, \delta)[\hat{\psi}_j]
            \, \mathrm{d}\sigma,
        \quad
    F_j(\omega, \delta)
      := \int_{\partial D_j}
            \mathcal{H}(\omega, \delta)
            \mathcal{F}(\omega)
            \, \mathrm{d}\sigma.
\end{align}
If $A(\omega, \delta)$ is invertible, then the solution to \eqref{eq: integral scattering problem} is given by
\begin{equation}\label{eq: solution of densities}
    \begin{aligned}
        \phi_0 &= \sum_{j=1}^{N} z_j \hat{\psi}_j,
    \
    z = \delta\, A(\omega, \delta)^{-1} F(\omega, \delta), \\[4pt]
    \phi &= \mathcal{H}(\omega, \delta)[\phi_0]+\delta\, 
            \mathcal{H}(\omega, \delta)\mathcal{F}(\omega), \\[4pt]
    \varphi
       &= (\mathcal{S}^{\alpha, k_m}_{D, s})^{-1}
           (\mathcal{S}^{\alpha, k_b}_{D, s})[\phi]
          - (\mathcal{S}^{\alpha, k_m}_{D, s})^{-1}[\tilde{u}^{i}].
    \end{aligned}
\end{equation}

\subsection{Subwavelength resonances}

In the absence of an incident field (i.e., $u^i = 0$), the resonances of the $N$ resonators in the scattering problem
\eqref{eq: model} are defined as complex frequencies $\omega$ with $\Im \omega < 0$ for which there exists a nontrivial solution to the integral equation $\mathcal{A}(\omega,\delta )[\phi] = 0$. Equivalently, the resonances correspond to complex numbers $\omega$ for which the finite dimensional system $A(\omega,\delta ) z = 0$ admits a nontrivial solution. In what follows, we analyze this nonlinear eigenvalue problem for $A(\omega,\delta)$ in the subwavelength regime.

The  matrix $A(\omega,\delta)$ in \eqref{eq: def of A and F} depends on $\mathcal{H}(\omega,\delta)$. It is therefore natural to derive its asymptotic expansion. Using the expansion of $\mathcal{A}(\omega,\delta)$
in \eqref{eq: expansion of A}, we obtain
\begin{equation}\label{eq: asymptotic of H}
    \begin{aligned}
         \mathcal{H}(\omega,\delta ) &= \widetilde{\mathcal{A}}_{0}^{-1}+\sum_{j=1}^3 (-1)^j\widetilde{\mathcal{A}}_{0}^{-1}\left( \mathcal{B}(\omega,\delta )\widetilde{\mathcal{A}}_{0}^{-1}\right)^j+O((\omega^2+\delta)^2),\\
    & = \mathcal{H}_0+\omega \mathcal{H}_{1,0}+\omega^2 \mathcal{H}_{2,0}+\omega^3 \mathcal{H}_{3,0}+\delta \mathcal{H}_{0,1}+\omega \delta \mathcal{H}_{1,1}+O((\omega^2+\delta)^2),
    \end{aligned}
\end{equation}
where the first few terms $\mathcal{H}_{n,m}$ in the expansion are defined as follows:
\begin{gather*}
    \mathcal{H}_0 = \widetilde{\mathcal{A}}_{0}^{-1},\ \mathcal{H}_{1,0} = -\widetilde{\mathcal{A}}_{0}^{-1} \mathcal{A}_{1,0}\widetilde{\mathcal{A}}_{0}^{-1}, \
    \mathcal{H}_{2,0} = \widetilde{\mathcal{A}}_{0}^{-1}\left( -\mathcal{A}_{2,0}+ \mathcal{A}_{1,0}\widetilde{\mathcal{A}}_{0}^{-1}\mathcal{A}_{1,0}\right)\widetilde{\mathcal{A}}_{0}^{-1}, \\
    \mathcal{H}_{3,0} = \widetilde{\mathcal{A}}_{0}^{-1} \left(-\mathcal{A}_{3,0}+ \mathcal{A}_{1,0}\widetilde{\mathcal{A}}_{0}^{-1} \mathcal{A}_{2,0}+ \mathcal{A}_{2,0}\widetilde{\mathcal{A}}_{0}^{-1}\mathcal{A}_{1,0}-\mathcal{A}_{1,0}\widetilde{\mathcal{A}}_{0}^{-1}\mathcal{A}_{1,0}\widetilde{\mathcal{A}}_{0}^{-1}\mathcal{A}_{1,0}\right)\widetilde{\mathcal{A}}_{0}^{-1},\\
    \mathcal{H}_{0,1} = -\widetilde{\mathcal{A}}_{0}^{-1}\mathcal{A}_{0,1}\widetilde{\mathcal{A}}_{0}^{-1}, \ \mathcal{H}_{1,1} = \widetilde{\mathcal{A}}_{0}^{-1}\left(-\mathcal{A}_{1,1}+\mathcal{A}_{1,0}\widetilde{\mathcal{A}}_{0}^{-1}\mathcal{A}_{0,1}+\mathcal{A}_{0,1}\widetilde{\mathcal{A}}_{0}^{-1}\mathcal{A}_{1,0}\right)\widetilde{\mathcal{A}}_{0}^{-1}.
\end{gather*}
We now present several integral identities that will be used in the analysis of $A_{ij}(\omega, \delta)$.
\begin{lemma}\label{lem: integral of Hn}
For any $1 \le i,j \le N$, the following relations hold:
\begin{gather*}
    \int_{\partial D_i} \mathcal{H}_{0}[\hat{\psi}_{j}]\, \mathrm{d}\sigma = \delta_{ij}, 
    \quad
    \int_{\partial D_i} \mathcal{H}_{1,0}[\hat{\psi}_{j}]\, \mathrm{d}\sigma = 0,
    \quad
    \int_{\partial D_i} \mathcal{H}_{0,1}[\hat{\psi}_{j}]\, \mathrm{d}\sigma = \delta_{ij}, \\
    \int_{\partial D_i} \mathcal{H}_{2,0}[\hat{\psi}_{j}]\, \mathrm{d}\sigma 
    = -\frac{|D_i|}{v_b^{2}}\,(C^{-1})_{ij}, 
    \quad
    \int_{\partial D_i} \mathcal{H}_{1,1}[\hat{\psi}_{j}]\, \mathrm{d}\sigma
    = -\frac{\mathrm{i}(\tau_m-\tau_b)}{|Y|}
      m_i \hat{m}_j, \\
    \int_{\partial D_i} \mathcal{H}_{3,0}[\hat{\psi}_{j}]\, \mathrm{d}\sigma
    = \mathrm{i}\,\frac{|D_i|}{v_b^{2}|Y|}
      \Big( Q_{ij} - \tau_b\, \hat{m}_i \hat{m}_j \Big).
\end{gather*}
Here, $C = (C_{ij})$ denotes the capacitance matrix introduced in \cref{def:periodic_capacitance}, $|D_i|$ is the volume of the $i$\textsuperscript{th} inclusion, and $|Y|$ is the volume of the periodic cell. The coefficients $Q_{ij}$ are defined by
\begin{equation} \label{def:qij}
    Q_{ij} 
    = \int_{\partial D_i}\int_{\partial D_j}
        \hat{G}^{a}_{s,1}(x,y)\,
        \hat{\psi}_{i}(x)\,\hat{\psi}_{j}(y)\, 
        \mathrm{d}\sigma_x\, \mathrm{d}\sigma_y
    \in \mathbb{R},
\end{equation}
and satisfy the antisymmetry relation $ Q_{ij} + Q_{ji} = 0$.
\end{lemma}

\begin{proof}
By the explicit expression of $\widetilde{\mathcal{A}}_{0}^{-1}$ in \eqref{eq: inverse of A0}, we have, for any $\psi \in L^2(\partial D)$, 
\begin{align}\label{eq: integral of A0 inverse}
    \int_{\partial D_i}
        \widetilde{\mathcal{A}}_{0}^{-1}[\psi]
        \,\mathrm{d}\sigma
    = \int_{\partial D_i}
        \psi
        \,\mathrm{d}\sigma.
\end{align}
In particular, since $\widetilde{\mathcal{A}}_{0}^{-1}[\hat{\psi}_j]
= \hat{\psi}_j$, the integrals of $\mathcal{H}_{0}[\hat{\psi}_{j}]$,
$\mathcal{H}_{1,0}[\hat{\psi}_{j}]$, $\mathcal{H}_{0,1}[\hat{\psi}_{j}]$,
and $\mathcal{H}_{2,0}[\hat{\psi}_{j}]$ over $\partial D_i$ follow directly from
\cref{lem: integral of An}. It remains to compute the integrals of $\mathcal{H}_{1,1}[\hat{\psi}_{j}]$ and $\mathcal{H}_{3,0}[\hat{\psi}_{j}]$ over $\partial D_i$. We begin with the term $\mathcal{H}_{1,1}$ over $\partial D_i$. By \cref{lem: integral of An},
\begin{align*}
    &\int_{\partial D_i}
        \mathcal{H}_{1,1}[\hat{\psi}_j]
        \,\mathrm{d}\sigma
    =  \int_{\partial D_i}
         \Big(
           -\mathcal{A}_{1,1}
           + \mathcal{A}_{0,1}\widetilde{\mathcal{A}}_{0}^{-1}
             \mathcal{A}_{1,0}
         \Big)[\hat{\psi}_j]
         \,\mathrm{d}\sigma \\
    = & -\frac{\mathrm{i}(\tau_m-\tau_b)}{|Y|}
          m_i \hat{m}_j-
      \int_{\partial D_i}
\widetilde{\mathcal{A}}_{0}^{-1}
             \mathcal{A}_{1,0}[\hat{\psi}_j]
         \,\mathrm{d}\sigma .
\end{align*}
We evaluate the second integral by computing the action of $\widetilde{\mathcal{A}}_{0}^{-1}\mathcal{A}_{1,0}$ on $\hat{\psi}_j$. We let $ h_j :=\big(\mathcal{S}^{0,0}_{D,s}\big)^{-1}\mathcal{S}^{a,\tau_b}_{D,s,1}[\hat{\psi}_j]$, so that $\mathcal{S}^{0,0}_{D,s}[h_j]= \mathcal{S}^{a,\tau_b}_{D,s,1}[\hat{\psi}_j]$.
By the jump relation, we obtain
\begin{align*}
     \left(-\frac{1}{2}\mathcal{I}
          + (\mathcal{K}^{0,0}_{D,s})^*\right)[h_j]
      = (\mathcal{K}^{-a,\tau_b}_{D,s,1})^*[\hat{\psi}_j].
\end{align*}
It then follows that
\begin{align*}
    \Big(-\frac{1}{2}\mathcal{I}
         + (\mathcal{K}^{0,0}_{D,s})^*
         + \mathcal{P}\Big)^{-1}
    (\mathcal{K}^{-a,\tau_b}_{D,s,1})^*[\hat{\psi}_j]
    = h_j - \mathcal{P}[h_j],
\end{align*}
which is equivalent to $\widetilde{\mathcal{A}}_{0}^{-1}\mathcal{A}_{1,0}[\hat{\psi}_j]= h_j - \mathcal{P}[h_j]$. Since
$h_j - \mathcal{P}[h_j] \in L^2_0(\partial D)$,
its integral over each connected component
$\partial D_i$ vanishes.
Therefore,
\begin{align*}
     \int_{\partial D_i}
        \widetilde{\mathcal{A}}_{0}^{-1}
        \mathcal{A}_{1,0}[\hat{\psi}_j]
        \,\mathrm{d}\sigma= 0.
\end{align*}
This completes the computation of the integral of $\mathcal{H}_{1,1}[\hat{\psi}_j]$.

We now compute the integral of $\mathcal{H}_{3,0}[\hat{\psi}_j]$ over $\partial D_i$.
Again by \cref{lem: integral of An},
\begin{align*}
    &\int_{\partial D_i}
        \mathcal{H}_{3,0}[\hat{\psi}_j]
        \,\mathrm{d}\sigma
    = \int_{\partial D_i}
         (
           -\mathcal{A}_{3,0}
           + \mathcal{A}_{2,0}\widetilde{\mathcal{A}}_{0}^{-1}
             \mathcal{A}_{1,0}
         )[\hat{\psi}_j]
         \,\mathrm{d}\sigma 
    = \frac{1}{v_b^2}
       \int_{D_i}
         (
           \mathcal{S}^{a,\tau_b}_{D,s,1}
         - \mathcal{S}^{0,0}_{D,s}
           \widetilde{\mathcal{A}}_{0}^{-1}\mathcal{A}_{1,0}
         )[\hat{\psi}_j]
         \,\mathrm{d}x.
\end{align*}
Using $\mathcal{S}^{a,\tau_b}_{D,s,1}[\hat{\psi}_j]=\mathcal{S}^{0,0}_{D,s}[h_j]$ and $\mathcal{S}^{0,0}_{D,s}
           \widetilde{\mathcal{A}}_{0}^{-1}\mathcal{A}_{1,0}
[\hat{\psi}_j]=\mathcal{S}^{0,0}_{D,s}[h_j]-\mathcal{S}^{0,0}_{D,s}
            \mathcal{P}[h_j]$, we obtain
\begin{align*}
    \int_{\partial D_i}
        \mathcal{H}_{3,0}[\hat{\psi}_j]
        \,\mathrm{d}\sigma=\frac{1}{v_b^2}
        \int_{D_i}
            \mathcal{S}^{0,0}_{D,s}
            \mathcal{P}[h_j]
            \,\mathrm{d}x.
\end{align*}           
By the definition of the projection operator $\mathcal{P}$,
\begin{align*}
    \mathcal{S}^{0,0}_{D,s}\mathcal{P}[h_j]
    = \sum_{l=1}^{N}
        \Big( \int_{\partial D_l} h_j \,\mathrm{d}\sigma \Big)
        \mathcal{S}^{0,0}_{D,s}[\hat{\psi}_l] 
     =-\sum_{l=1}^{N}\sum_{t=1}^{N}
        \Big( \int_{\partial D_l} h_j \,\mathrm{d}\sigma \Big)
        (C^{-1})_{lt}\int_{\partial D_i} \chi_{\partial D_t}\,\mathrm{d}x,
\end{align*}
so that
\begin{align*}
    \int_{\partial D_i}
        \mathcal{H}_{3,0}[\hat{\psi}_j]
        \,\mathrm{d}\sigma
    = -\frac{|D_i|}{v_b^2}
      \sum_{l=1}^{N}
        \Big( \int_{\partial D_l} h_j \,\mathrm{d}\sigma \Big)
        (C^{-1})_{li}.
\end{align*}
To evaluate the integral of $h_j$ over $\partial D_i$, we note that
\begin{align*}
    \int_{\partial D_l} h_j \,\mathrm{d}\sigma
    &= \big\langle
          (\mathcal{S}^{0,0}_{D,s})^{-1}
          \mathcal{S}^{a,\tau_b}_{D,s,1}[\hat{\psi}_j],
          \chi_{\partial D_l}
       \big\rangle
     = -\sum_{t=1}^{N} C_{lt}
        \big\langle 
          \mathcal{S}^{a,\tau_b}_{D,s,1}[\hat{\psi}_j],
          \hat{\psi}_t
        \big\rangle.
\end{align*}
Since
$\sum_{l=1}^{N}(C^{-1})_{li} C_{lt} = \delta_{it}$, we conclude that
\begin{align*}
    \int_{\partial D_i}
        \mathcal{H}_{3,0}[\hat{\psi}_j]
        \,\mathrm{d}\sigma
    = \frac{|D_i|}{v_b^2}
       \left\langle
         \mathcal{S}^{a,\tau_b}_{D,s,1}[\hat{\psi}_j],
         \hat{\psi}_i
       \right\rangle 
    = \frac{|D_i|}{v_b^2}
       \int_{\partial D}\!\int_{\partial D}
         G_{s,1}^{a,\tau_b}(x,y)\,
         \hat{\psi}_j(y)\,
         \hat{\psi}_i(x)\,
         \mathrm{d}\sigma(y)\,
         \mathrm{d}\sigma(x).
\end{align*}
Using the decomposition of $G_{s,1}^{a,\tau_b}$ in~\eqref{eq: Gs1}, we obtain
\begin{align*}
        \int_{\partial D_i}
        \mathcal{H}_{3,0}[\hat{\psi}_j]
        \,\mathrm{d}\sigma
    = \mathrm{i}\,\frac{|D_i|}{v_b^2|Y|}
      \Big( Q_{ij} - \tau_b\, \hat{m}_i \hat{m}_j \Big).
\end{align*}
Finally, since  $\hat{G}^{a}_{s,1}(x,y) = -\hat{G}^{a}_{s,1}(y,x)$, a change of variables yields $Q_{ij} = -Q_{ji}$.
\end{proof}

According to the integral identity in \cref{lem: integral of Hn}, the matrix $A(\omega,\delta)_{ij}$ admits the following asymptotic expansion with respect to $\omega$ and $\delta$:
\begin{equation*}
    A(\omega,\delta)_{ij} =\, \omega^2 \frac{|D_i|}{v_b^2}(C^{-1})_{ij} 
    + \mathrm{i} \omega^3 \frac{|D_i|}{v_b^2|Y|} \left(\tau_b \hat{m}_i \hat{m}_j - Q_{ij}\right) 
    - \delta \delta_{ij}+ 
     \mathrm{i} \omega \delta \frac{(\tau_m-\tau_b)}{|Y|} m_i \hat{m}_j 
    + O((\omega^2 + \delta)^2).
\end{equation*}
In matrix form, this expansion can be rewritten as
\begin{align*}
    A(\omega,\delta) = \frac{\omega^2}{v_b^2} V C^{-1} 
    + \mathrm{i} \frac{\omega^3}{v_b^2|Y|} V \left(\tau_b \hat{m}\hat{m}^{\top} - Q\right) 
     - \delta I 
    + \mathrm{i} \omega \delta \frac{(\tau_m-\tau_b)}{|Y|} m\hat{m}^{\top} 
    + O((\omega^2 + \delta)^2),
\end{align*}
where $Q$ is the matrix with the entries $Q_{ij}$ defined in \eqref{def:qij}, and
$m$ and $\hat{m}$are  given by \eqref{def:mj}. To analyze the resonances of $A(\omega,\delta)$,
we introduce the following scaled matrix:
\begin{align}\label{eq: new A}
    \tilde{A}(\omega,\delta)
    :=  A(\omega,\delta) C/\delta
    = \lambda(\omega) V - C
      + \mathrm{i}\omega(\lambda B + E)
      + \mathcal{O}(\omega^2+\delta),
\end{align}
where the scaling parameter $\lambda$ and the matrices $B$ and $E$ are defined by
\begin{align}\label{eq: expression of s B E}
     \lambda = \frac{\omega^2}{\delta v_b^2},
    \qquad
    B = \frac{\tau_b V\hat{m}m^\top  - VQ C}{|Y|} ,
    \qquad
    E = \frac{(\tau_m-\tau_b) m m^\top }{|Y|} .
\end{align}
We now establish the following result concerning the low-frequency resonances of \eqref{eq: model}.

\begin{theorem}\label{thm: asymptotic of resonance frequency}
Under \cref{asmp: eigenvalue}, in the subwavelength regime, there exist $N$ resonant frequencies with the asymptotic expansions
\begin{align}
    \omega_j(\delta)
    = v_b\sqrt{\delta\lambda_j}
      - \mathrm{i}\,
        \frac{\tau_m\,v_b^2\, (m^{\top}u_j)^2}
             {2|Y|}\,\delta
      + \mathcal{O}(\delta^{3/2}),
    \quad 1 \le j \le N.
\end{align}
The associated eigenvectors $u_j(\delta)$ of $\tilde{A}(\omega_j(\delta), \delta)$ satisfy
\begin{equation*}
     u_{j}(\delta)
      = u_j-\mathrm{i}v_b\sqrt{\delta\lambda_j}\sum_{i \neq j}\frac{\tau_m(m^{\top} u_j)(m^{\top}u_i)- \lambda_j \lambda_i\,u_j^{\top} V Q V u_i}{(\lambda_j - \lambda_i)\,|Y|}\,u_i+\mathcal{O}(\delta).
\end{equation*}
Here, $(\lambda_j, u_j)$ denotes the eigenpair of the capacitance matrix defined in Lemma~\ref{lem:spC}.
\end{theorem}

\begin{proof}
From the expression of $\tilde{A}(\omega,\delta)$ in \eqref{eq: new A}, we observe that $(\lambda_j,u_j)$ satisfies the leading-order. Applying the implicit function theorem as in \cite{feppon2022modal}, we obtain analytic functions $\omega_j(\delta)$ and $u_j(\delta)$, defined for $\delta$ sufficiently small, which solve the nonlinear eigenvalue problem
\begin{align}\label{eq: modified nonlinear eigenvalue}
    \tilde{A}(\omega_j({\delta}), \delta)[u_j({\delta})] = 0,\quad \delta \to 0.
\end{align}
We seek asymptotic expansions of the form
\begin{align*}
    \omega_j(\delta)= \omega_j \sqrt{\delta}-\mathrm{i}\omega_{j,1}\delta+\mathcal{O}(\delta^{3/2}),\quad
    u_j(\delta)= u_j - \mathrm{i}u_{j,1}\sqrt{\delta} + \mathcal{O}(\delta),
\end{align*}
where $\omega_j = v_b\sqrt{\lambda_j}$. Observe that $\lambda_j(\delta)=\omega^2_j(\delta)/(v_b^2\delta)= \lambda_j+ 2\mathrm{i}\omega_j\omega_{j,1}/v_b^2 \sqrt{\delta}+\mathcal{O}(\delta)$. Substituting these expansions into \eqref{eq: modified nonlinear eigenvalue} and collecting terms of equal order, we obtain
\begin{align*}
     (\lambda_j V - C)u_j
    + \mathrm{i}\!\left[\omega_j(\lambda_j B + E)u_j-
        (\lambda_j V - C)u_{j,1}
        - (2\omega_j\omega_{j,1}/v_b^2)V u_j
      \right]\sqrt{\delta}
    + \mathcal{O}(\delta)
    = 0.
\end{align*}
The leading-order equation is satisfied by construction. Projecting the first-order equation onto $u_j$ and expanding $u_{j,1}$ in the eigenbasis $\{u_i\}_{i=1}^N$, we obtain the first-order corrections
\begin{align*}
    \omega_{j,1}
      = \frac{v_b^2u_j^{\top}(\lambda_j B + E)u_j}{2}, \quad 
    u_{j,1}
      = \omega_j\sum_{i \neq j}
          \frac{
            u_j^{\top}(\lambda_j B + E)u_i
          }{\lambda_j - \lambda_i}
          \,u_i.
\end{align*}
Substituting the expressions of $B$ and $E$ from \eqref{eq: expression of s B E} into the above formulas gives
\begin{align*}
      \omega_{j,1}
      = \frac{v_b^2}{2}\frac{\lambda_ju_j^{\top} \left[\tau_b V \hat{m} m^\top  - VQ C\right] u_j + u_j^{\top} \left[(\tau_m-\tau_b) m m^\top \right] u_j }{|Y|}= \frac{\tau_m v_b^2\, (m^{\top} u_j)^2}{2|Y|},
\end{align*}
where we used the antisymmetry of $Q$, which implies $u_j^{\top}Q u_j = 0$. Similarly, the correction term $u_{j,1}$ can be written explicitly as
\begin{equation*}
     u_{j,1}
      = \omega_j\sum_{i \neq j}
          \frac{\tau_m (m^{\top} u_j)(m^{\top} u_i)
            - \lambda_j\lambda_i \,
              u_j^{\top} V Q V u_i
          }{
            (\lambda_j - \lambda_i)\,|Y|
          }\,u_i.
\end{equation*}
This completes the proof.
\end{proof}

\begin{remark}
In the case of a single resonator ($N=1$), the resonant frequency admits the asymptotic expansion
\begin{align*}
    \omega = v_b \sqrt{\frac{\mathrm{Cap}(D)}{|D|}} \delta^{1/2} 
    - \mathrm{i}  \frac{\tau_m v_b^2 m^2}{2|Y| |D|} \delta 
    + O(\delta^{3/2}),
\end{align*}
where $\mathrm{Cap}(D)$ denotes the periodic capacity of the resonator $D$, as defined in \cref{def:periodic_capacitance} (with $N=1$), and $|D|$ is the volume of $D$.
\end{remark}

\subsection{Approximation of the scattering problem}

We are now in a position to derive an approximation of the scattered field. As a first step, we solve the finite dimensional system \eqref{eq: finite dimensional equation} to determine the unknown coefficient vector $z$.

\begin{proposition}\label{prop: asymptotic of solution z}
Let $\lambda_j(\omega)$ and $u_j(\omega)$ denote, respectively, the eigenvalue and eigenvector associated with the nonlinear eigenvalue problem
\begin{align*}
    \tilde{A}(\omega, \lambda_j(\omega))[u_j(\omega)] = 0.
\end{align*} 
Assume that $0 < \omega < K \sqrt{\delta}$ for some constant $K > 0$, and that \cref{asmp: eigenvalue} holds. Then, the solution to~\eqref{eq: finite dimensional equation} admits the asymptotic representation
\begin{align}\label{eq: asymptotic of z}
    z  = 2\mathrm{i} \tau_m \omega \sum_{j=1}^N 
          \frac{m^{\top}u_{j}}
               {\lambda_j(\omega)-\lambda(\omega)}\,C u_j\,\big(1+\mathcal{O}(\omega)\big).
\end{align}
Moreover, the eigenvalue $\lambda_j(\omega)$ admits the following expansion:
\begin{align}\label{eq: lambdaj1 omega}
    \lambda_j(\omega) = \lambda_j-\omega \lambda_{j,1}+\mathcal{O}(\omega^2),\quad \lambda_{j,1}=\frac{\tau_m (m^{\top} u_j)^2}{|Y|},
\end{align}
where the vector $m$ is defined in~\eqref{def:mj}, and the matrix $V$ is introduced in \cref{lem:spC}.
\end{proposition}

\begin{proof}
The finite-dimensional system~\eqref{eq: finite dimensional equation} can be rewritten equivalently as
\begin{align}\label{eq: modified finite dimensional system}
    \tilde{A}(\omega,\delta)\,(C^{-1}z)=F(\omega,\delta),
\end{align}
where $\tilde{A}(\omega,\delta)$ is defined in~\eqref{eq: new A}. Following the approach in~\cite{feppon2022modal}, we analyze this reformulated system. From \cref{thm: asymptotic of resonance frequency}, the resonance frequencies $\omega_j(\delta)$ were obtained asymptotically as $\delta \to 0$. Conversely, one may consider the inverse relation $\delta_j(\omega)$ as $\omega \to 0$. Equivalently, this amounts to studying the behavior of $\lambda_j(\omega)$ as $\omega \to 0$, through the relation $\lambda = \omega^2/(\delta v_b^2)$. For convenience, we introduce the similar notation $\tilde{A}(\omega,\lambda(\omega))= A(\omega,\delta)C/\delta$.

Following similar arguments to those in \cref{thm: asymptotic of resonance frequency}, we assume the asymptotic expansions
\begin{align}
   \lambda_j(\omega)= \lambda_j -\mathrm{i}\omega \lambda_{j,1}+ \mathcal{O}(\omega^2), \quad
    u_j(\omega) = u_j - \mathrm{i}\omega u_{j,1} +\mathcal{O}(\omega^2) .
\end{align}
Substituting these expansions into the nonlinear eigenvalue problem $\tilde{A}(\omega,\lambda_j(\omega))[u_j(\omega)] = 0$ and collecting terms of order $\omega$, we obtain
\begin{align*}
    (\lambda_j B + E)u_j-\lambda_{j,1}V u_j-(\lambda_j V - C)u_{j,1}=0.
\end{align*}
Projecting onto $u_j$ yields
\begin{align*}
    \lambda_{j,1}
    = u_j^{\top}(\lambda_j B + E)u_j=\frac{\tau_m\, (m^{\top}  u_j)^2}{|Y|}.
\end{align*}
By continuity, the family $\{u_j(\omega)\}_{j=1}^N$ forms a basis of $\mathbb{C}^N$ for $\omega$ sufficiently small. Hence, we may expand
\begin{align*}
    C^{-1}z
    = \sum_{j=1}^N c_j(\omega,\delta)\,u_j(\omega),
\end{align*}
where $c_j(\omega,\delta)$ are scalar coefficients. Using the property $\tilde{A}(\omega,\lambda_j(\omega))\,u_j(\omega)=0$, we obtain
\begin{align*}
    \tilde{A}(\omega,\lambda(\omega))(C^{-1}z)
    = \sum_{j=1}^N c_j(\omega,\delta)
      \big(
        \tilde{A}(\omega,\lambda(\omega))
        - \tilde{A}(\omega,\lambda_j(\omega))
      \big)
      u_j(\omega).
\end{align*}
Observe that
\begin{align*}
    \tilde{A}(\omega,\lambda(\omega))
    - \tilde{A}(\omega,\lambda_j(\omega))
    =
    \big(\lambda(\omega) - \lambda_j(\omega)\big)
    \big(
        V + \mathrm{i}\omega B
        + \mathcal{O}(\omega^2)
    \big).
\end{align*}
Therefore,
\begin{align*}
    \tilde{A}(\omega,\lambda(\omega))(C^{-1}z)
    =
    \sum_{j=1}^N
    c_j(\omega,\delta)
    \big(\lambda(\omega) - \lambda_j(\omega)\big)
    V u_j
    \big(1+\mathcal{O}(\omega)\big).
\end{align*}
Multiplying both sides of~\eqref{eq: modified finite dimensional system}
by $u_i^{\top}$ yields
\begin{align*}
    c_i(\omega,\delta)
    \big(\lambda(\omega)-\lambda_i(\omega)\big)
    \big(1+\mathcal{O}(\omega)\big)
    =
    u_i^{\top}F(\omega,\delta).
\end{align*}
Accordingly,
\begin{align*}
    c_i(\omega,\delta)
      =
      \frac{u_i^{\top}F(\omega,\delta)}
           {\lambda(\omega)-\lambda_i(\omega)}
      \big(1+\mathcal{O}(\omega)\big).
\end{align*}
Substituting this into the expansion for $C^{-1}z$ gives
\begin{align*}
    z
    = \sum_{j=1}^N
       c_j(\omega,\delta)\,
       C u_j(\omega) 
    = \sum_{j=1}^N
       \frac{
             u_j^{\top}F(\omega,\delta)}
            {\lambda(\omega)-\lambda_j(\omega)}
       \,C u_j
       \big(1+\mathcal{O}(\omega)\big).
\end{align*}

Next, we determine the asymptotic form of $F(\omega,\delta)$. From the expansions of $\mathcal{H}(\omega,\delta)$ and $\mathcal{F}(\omega)$ in
\eqref{eq: asymptotic of H} and \eqref{eq: asymptotic of F(omega)}, respectively, we obtain
\begin{align*}
    F_i(\omega,\delta)
      =& \int_{\partial D_i}\mathcal{H}_{0}\Big( -2\mathrm{i} \tau_m \nu_d \omega+ 2\mathrm{i} \tau_m\omega\Big(\frac{1}{2} \mathcal{I}+ (\mathcal{K}^{0,0}_{D,s})^*\Big)(\mathcal{S}^{0,0}_{D,s})^{-1}[x_d]\Big)\, d\sigma+ \mathcal{O}(\omega^2) \\
      =& 2\mathrm{i} \tau_m \omega\int_{\partial D_i}\Big(\frac{1}{2} \mathcal{I} + (\mathcal{K}^{0,0}_{D,s})^*\Big)(\mathcal{S}^{0,0}_{D,s})^{-1}[x_d] \, d\sigma+ \mathcal{O}(\omega^2).
\end{align*}
By \cref{lem: integral of Ksn} and \eqref{eq: integral of Sinverse_xd}, this reduces to
\begin{align}\label{eq: asymptotic of vector F}
    F(\omega,\delta)=-2\mathrm{i} \tau_m \omega\, m+ \mathcal{O}(\omega^2).
\end{align}
Substituting \eqref{eq: asymptotic of vector F} into the expression for $z$ immediately yields the asymptotic representation
\eqref{eq: asymptotic of z}, which completes the proof.
\end{proof}

\begin{remark}
Keldysh’s theorem (see, e.g., \cite{mennicken2003non,guttel2017nonlinear}) provides an alternative approach to describing the behavior of the resolvent $(\tilde{A}(\omega,s))^{-1}$ in a neighborhood of the eigenvalues
$s_j(\omega)$. In particular, the resolvent admits a decomposition into a singular part associated with each eigenmode and a holomorphic regular part.
\end{remark}

Building on \cref{prop: asymptotic of solution z}, we proceed to derive asymptotic analysis for the total field and the reflection coefficient.

\begin{theorem}\label{thm: approximation of total field}
The total field $u$ in $Y \times (h, +\infty)$ admits the decomposition
\begin{align*}
    u(x) = u^i(x) + u^p(x) + u^e(x),
\end{align*}
where $h := \min\{x_d \mid x \in \partial D\}$, and $u^p$ and $u^e$ denote the propagating and evanescent components, respectively.  
The propagating part $u^p$ consists of a single propagating mode, which can be represented as
\begin{align}
    u^p(x)
      = r(\omega)\,
        e^{\mathrm{i}k_m \hat{\theta}_{+} \cdot x},\quad x\in Y \times (h, +\infty),
\end{align}
where $r(\omega)$ is the reflection coefficient, given asymptotically by
\begin{align}\label{eq: asymptotic of reflection coefficient}
    r(\omega)
    = -1
      - \sum_{j=1}^{N}
        \frac{2\mathrm{i}\omega \lambda_{j,1}}
             {\lambda_j - \mathrm{i}\omega \lambda_{j,1}-\lambda(\omega)}
        \big(1 + \mathcal{O}(\omega)\big)
      + \mathcal{O}(\omega),
\end{align}
with $\lambda(\omega)$ and $\lambda_{j,1}$ being defined in \eqref{eq: expression of s B E} and \eqref{eq: lambdaj1 omega}, respectively. Furthermore, there exists a constant $ c := \inf_{\eta \in \Lambda^* \setminus \{0\}}
\sqrt{|\alpha + \eta|^2 - k^2}$ such that
\begin{align}
    |u^e(x)|
    \lesssim e^{-c(x_d - h)} \, \|\varphi\|_{L^2(\partial D)},\quad x\in Y \times (h, +\infty).
\end{align}
\end{theorem}

\begin{proof}
By the definition of the quasi-periodic Green function \eqref{eq: quasi-Green's fun}, the scattered field $u^s$ in $Y \times (h,+\infty)$ admits the decomposition
\begin{align*}
    u^{s}(x)
    = \mathcal{S}^{\alpha, k_m}_{D, s}[\varphi](x)
    = u^{s,p}(x) + u^{s,e}(x),
\end{align*}
where $u^{s,p}$ and $u^{s,e}$ denote the propagating and evanescent components, respectively. The propagating component consists of a single mode, which can be written as
\begin{align*}
    u^{s,p}(x)
      &= \int_{\partial D} 
          \left(\frac{e^{\mathrm{i}\alpha \cdot (x_\ell - y_\ell)} e^{\mathrm{i}\sqrt{k_m^2-\alpha^2} |x_d - y_d|}}{2\mathrm{i}\sqrt{k_m^2-\alpha^2} |Y|}
            -\frac{e^{\mathrm{i}\alpha \cdot (x_\ell - y_\ell)} e^{\mathrm{i}\sqrt{k_m^2-\alpha^2} |x_d + y_d|}}{2\mathrm{i}\sqrt{k_m^2-\alpha^2} |Y|}\right) \varphi(y)\, \mathrm{d}\sigma(y) \\
      &= -\left(\int_{\partial D} \frac{\sin(\omega \tau_m y_d)}{\omega \tau_m |Y|}e^{-\mathrm{i}\alpha \cdot y_\ell}\varphi(y)\, \mathrm{d}\sigma(y)\right)e^{\mathrm{i}k_m \hat{\theta}_{+} \cdot x}.
\end{align*}
The total propagating wave is the sum of the reflected and scattered contributions
\begin{align*}
    u^p(x)
      = u^{r}(x)
        + u^{s,p}(x)
      = r(\omega)\,
        e^{\mathrm{i}k_m  \hat{\theta}_{+} \cdot x},
\end{align*}
where the reflection coefficient $r(\omega)$ is expressed as
\begin{align}\label{eq: formula of reflection coefficient}
    r(\omega)
      = -1- \int_{\partial D}\frac{\sin(\omega \tau_m y_d)}{\omega \tau_m |Y|}e^{-\mathrm{i}\alpha \cdot y_\ell}\varphi(y)\, \mathrm{d}\sigma(y).
\end{align}
From~\eqref{eq: solution of densities} and the asymptotic forms of  
$\mathcal{H}(\omega, \delta)$ and $\mathcal{F}(\omega)$ in~\eqref{eq: asymptotic of H}  
and~\eqref{eq: asymptotic of F(omega)}, the density $\phi$ satisfies
\begin{align*}
    \phi
      = \mathcal{O}(\omega \delta)
        + \sum_{j=1}^N z_j\, \mathcal{H}(\omega, \delta)[\hat{\psi}_j]
      = \sum_{j=1}^N z_j\, \hat{\psi}_j(1 + \mathcal{O}(\omega))
        + \mathcal{O}(\omega \delta).
\end{align*}
Hence, by~\eqref{eq: solution of densities},
\begin{align*}
    \varphi
      = (\mathcal{S}^{\alpha, k_m}_{D, s})^{-1}(\mathcal{S}^{\alpha, k_b}_{D, s})[\phi]
         - (\mathcal{S}^{\alpha, k_m}_{D, s})^{-1}[\tilde{u}^{i}]
      = \phi(1 + \mathcal{O}(\omega))
         + 2\mathrm{i}\tau_m \omega
           (\mathcal{S}^{0, 0}_{D, s})^{-1}[x_d]
           (1 + \mathcal{O}(\omega)).
\end{align*}
Substituting the expression of $\phi$ yields
\begin{align}\label{eq: asymptotic of varphi}
    \varphi
      = \sum_{j=1}^N z_j \hat{\psi}_j(1 + \mathcal{O}(\omega))
        + \mathcal{O}(\omega).
\end{align}
Substituting~\eqref{eq: asymptotic of varphi} into~\eqref{eq: formula of reflection coefficient} yields
\begin{align*}
    r(\omega)
      &= -1
         - \frac{1}{|Y|}
           \int_{\partial D} y_d\, \varphi(y)\, \mathrm{d}\sigma(y)
           (1 + \mathcal{O}(\omega))
         + \mathcal{O}(\omega)
       = -1
         - \frac{\hat{m}^{\top}z}{|Y|} (1 + \mathcal{O}(\omega))
         + \mathcal{O}(\omega).
\end{align*}
Using the asymptotic expansion of $z$ from~\cref{prop: asymptotic of solution z}, we obtain
\begin{align*}
    r(\omega)
      &= -1
         - \frac{2\mathrm{i}\omega \tau_m}{|Y|}
           \sum_{j=1}^N
           \frac{(m^{\top}u_j)^2}
                {\lambda_j(\omega)-\lambda(\omega)}
           (1 + \mathcal{O}(\omega))
         + \mathcal{O}(\omega).
\end{align*}
Using the relation for $\lambda_j(\omega)$ in~\eqref{eq: lambdaj1 omega}, this reduces to
\begin{align*}
    r(\omega)
      = -1
        - \sum_{j=1}^N
          \frac{2\mathrm{i}\omega \lambda_{j,1}}
               {\lambda_j - \mathrm{i}\omega \lambda_{j,1}-\lambda(\omega)}
          (1 + \mathcal{O}(\omega))
        + \mathcal{O}(\omega).
\end{align*}

Finally, consider the evanescent part $u^e(x) = u^{s,e}(x)$. From the definition of $G^{\alpha, k}_{s}(x,y)$, we have
\begin{align*}
    u^{e}(x)
      = - \sum_{\eta \in \Lambda^* \setminus \{0\}}
          \int_{\partial D}
            \frac{
            e^{\mathrm{i} (\alpha + \eta) \cdot (x_\ell-y_{\ell})}
            (e^{-\sqrt{|\alpha + \eta|^2 - k_m^2}\, (x_d-y_d)}-e^{-\sqrt{|\alpha + \eta|^2 - k_m^2}\, (x_d+y_d)})}
            {2|Y| \sqrt{|\alpha + \eta|^2 - k_m^2}}
            \varphi(y)\, \mathrm{d}\sigma(y).
\end{align*}
Thus,
\begin{align*}
    |u^e(x)|
      &\le e^{-c(x_d - h)}
        \int_{\partial D}
        \sum_{\eta \in \Lambda^* \setminus \{0\}}
          \frac{
            e^{-\sqrt{|\alpha + \eta|^2 - k_m^2}(h - y_d)}
            + e^{-\sqrt{|\alpha + \eta|^2 - k_m^2}(h + y_d)}}
               {2|Y| \sqrt{|\alpha + \eta|^2 - k_m^2}}
          |\varphi(y)|\, \mathrm{d}\sigma(y).
\end{align*}
By the Cauchy–Schwarz inequality, we finally deduce that
\begin{align*}
    |u^e(x)|
      \le C\, e^{-c(x_d - h)}\, \|\varphi\|_{L^2(\partial D)},
\end{align*}
where $C$ is a constant depending on the geometry of $D$ and the parameters $\alpha$, $k$, and $h$.
\end{proof}

As a direct consequence of \cref{thm: approximation of total field}, we derive the following approximate scattering problem with impedance boundary condition.
\begin{proposition} \label{prop:impedance}
Let $u_{\mathrm{app}}(x) = u^{i}(x) + u^{p}(x)$ denote the approximate total field, since the evanescent part $u^{e}(x)$ decays exponentially as $x_d \to +\infty$.  
Then $u_{\mathrm{app}}$ satisfies the following Helmholtz equation in the upper half space with impedance boundary condition:
\begin{equation}
\begin{cases}
\begin{aligned}
    &\Delta u_{\mathrm{app}} + k_m^2 u_{\mathrm{app}} = 0 \quad
    \text{in } \mathbb{R}^d_+, \\[0.25em]
    &u_{\mathrm{app}} + \gamma(\omega)\, 
      \dfrac{\partial u_{\mathrm{app}}}{\partial x_d} = 0 \quad
     \text{on } \partial \mathbb{R}^d_+, \\[0.25em]
    &u_{\mathrm{app}} - u^{i} \text{ satisfies the outgoing radiation condition as } x_d \to +\infty.
\end{aligned}
\end{cases}
\end{equation}
Here, the impedance parameter $\gamma(\omega)$ is given by
\begin{align}
    \gamma(\omega)
      = \frac{1}{\mathrm{i}\omega \tau_m }\frac{1+r(\omega)}{1-r(\omega)},
\end{align}
where $r(\omega)$ denotes the reflection coefficient defined in~\cref{thm: approximation of total field} and is approximately given by \eqref{eq: formula of reflection coefficient}.

\begin{remark}\label{rem:absorbing boundary condition}

If the material parameters $\kappa_b$ and $\rho_b$ are real, then at the resonant frequencies $\omega_j(\delta)$ given in Theorem~\ref{thm: asymptotic of resonance frequency}, or equivalently when $\lambda(\omega)$ approaches $\lambda_j$ in \cref{thm: approximation of total field}, the reflection coefficient satisfies $r(\omega_j) \approx 1$. Consequently, the periodic system of resonators on a reflective surface (with Dirichlet boundary condition) behaves effectively as a surface satisfying a Neumann boundary condition. Indeed, when $r(\omega)\to 1$, we have $\gamma(\omega)\to \infty$, and in this limit the approximate field $u_{\mathrm{app}}$ satisfies the sound-hard (Neumann) boundary condition.

If the material parameters $\kappa_b$ and $\rho_b$ are complex, the situation changes. In particular, when $\Re \lambda(\omega^{*}) \approx \lambda_j$ and $\Im \lambda(\omega^{*}) \approx \omega^{*} \lambda_{j,1}$ for some frequency $\omega^{*}$, the reflection coefficient approaches zero according to \eqref{eq: asymptotic of reflection coefficient} in \cref{thm: approximation of total field}. This corresponds to a regime of superabsorption \cite{regime}. For example, let $v^2_b=\sigma-\mathrm{i}\zeta$, where $\sigma>0$ and $\zeta \ge 0$. Then we decompose
\begin{align*}
    \lambda_j- \mathrm{i}\omega \lambda_{j,1} - \lambda(\omega)
    = \lambda_j-\frac{\sigma}{\sigma^2+\zeta^2}\frac{\omega^2}{\delta}
    -\mathrm{i}\left(\omega\lambda_{j,1}+\frac{\zeta}{\sigma^2+\zeta^2}\frac{\omega^2}{\delta}\right).
\end{align*}
By appropriately choosing the parameters $\sigma$ and $\zeta$, one can enforce $\Re \lambda(\omega^{*}) \approx \lambda_j$ and $\Im \lambda(\omega^{*}) \approx \omega^{*} \lambda_{j,1}$, thereby achieving the desired condition. This provides a rigorous explanation of the superabsorption phenomenon observed in \cite{lanoy2018broadband,leroy2015superabsorption}. Moreover, when $r(\omega)\to 0$, we have $\gamma(\omega)=1/(\mathrm{i}\omega \tau_m)$, and the approximate field $u_{\mathrm{app}}$ satisfies the absorbing boundary condition
\begin{align*}
    \frac{\partial u_{\mathrm{app}}}{\partial x_d}+ \mathrm{i}\omega \tau_m u_{\mathrm{app}}= 0.
\end{align*}

\begin{comment}
If the reflection coefficient $r(\omega)$ approaches $1$, then $\gamma(\omega)\to \infty$. In this limit, $u_{\mathrm{app}}$ satisfies the sound-hard (Neumann) boundary condition. If $r(\omega)$ approaches $0$, then
$\gamma(\omega)=1/(\mathrm{i}\omega \tau_m )$. In this case, $u_{\mathrm{app}}$ satisfies the absorbing boundary condition
\begin{align*}
    \frac{\partial u_{\mathrm{app}}}{\partial x_d}+ \mathrm{i}\omega \tau_m u_{\mathrm{app}}= 0.
\end{align*}
\end{comment}

\end{remark}

\begin{comment}
    \begin{equation}
\label{app:r}
 r(\omega)
      = -1 - \sum_{j=1}^N
          \frac{2\mathrm{i}\omega \lambda_{j,1}}
               {\lambda - \lambda_j - \mathrm{i}\omega \lambda_{j,1}}.
\end{equation}
\end{comment}

\end{proposition}

\section{Reduced order model for broadband absorption} \label{sec:broabband}

In this section, we introduce two reduced order models for broadband absorption on a prescribed frequency interval $[\omega_{\min},\omega_{\max}]$. Both are based on the analytical approximation of the reflection coefficient $r(\omega)$ derived in \Cref{thm: approximation of total field}. In particular, \eqref{eq: asymptotic of reflection coefficient} shows that the evaluation of $r(\omega)$ depends primarily on the modal quantities $\lambda_j$ and $\lambda_{j,1}$. Since these quantities are frequency independent, they can be computed once for a given geometry and then reused to efficiently evaluate $r(\omega)$ over the target band.

We first consider a global in frequency criterion that directly penalizes reflection over $[\omega_{\min},\omega_{\max}]$ through the band averaged reflectance
\begin{equation}\label{def:objective_simple}
J^{\mathrm{ref}}
=\frac{1}{\omega_{\max}-\omega_{\min}}
\int_{\omega_{\min}}^{\omega_{\max}} | r(\omega)|^2 \, d\omega.
\end{equation}
This functional represents the mean reflected energy over the target band. Minimizing $J^{\mathrm{ref}}$ promotes broadband absorption in an integrated sense, rather than enforcing low reflectance only at isolated frequencies.

A complementary viewpoint is provided by the resonance mechanism. Broadband absorption can be enhanced when several resonance induced low reflection features are distributed across the band and overlap. This suggests that the real parts of the resonant frequencies should be well spread within $[\omega_{\min},\omega_{\max}]$, and that the reflectance should be small near these resonances. Motivated by the superabsorption mechanism discussed in \Cref{rem:absorbing boundary condition}, we introduce a second objective that enforces the matching condition at a prescribed set of target frequencies $\{\omega_j^{*}\}_{j=1}^M \subset [\omega_{\min},\omega_{\max}]$:
\begin{equation}\label{def:objective_res}
J^{\mathrm{res}}
=\frac{1}{M}\sum_{j=1}^M\left[
\left(\frac{\lambda_j}{\Re\lambda(\omega^{*}_j)}-1\right)^2
+
\left(\frac{\omega_j^{*}\lambda_{j,1}}{\Im\lambda(\omega^{*}_j)}-1\right)^2
\right].
\end{equation}
For broadband performance, a natural choice is to distribute the targets throughout $[\omega_{\min},\omega_{\max}]$, for instance, uniformly as
\begin{align}
\omega_j^{*}
=\omega_{\min}+\frac{j}{M+1}\,(\omega_{\max}-\omega_{\min}),
\quad j=1,\dots,M.
\end{align}

Overall, $J^{\mathrm{ref}}$ and $J^{\mathrm{res}}$ capture broadband performance from complementary perspectives. The functional $J^{\mathrm{ref}}$ is a global measure that reduces reflection on average across the band, whereas $J^{\mathrm{res}}$ is a resonance driven, local in frequency criterion that promotes near zero reflectance at selected frequencies and enhances overlapping superabsorption features.

Finally, we note that both $J^{\mathrm{ref}}$ and $J^{\mathrm{res}}$ depend on the resonant quantities $\lambda_j$ and $\lambda_{j,1}$. Therefore, to perform gradient-based shape optimization for either objective, we require the corresponding shape derivatives of $\lambda_j$ and $\lambda_{j,1}$. We then turn to the derivation of these shape derivatives.

\subsection{Shape sensitivity analysis}

We begin by recalling the definition of transformations generated by the velocity method (see \cite{delfour2011shapes,sokolowski1992introduction,haslinger2003introduction}
for further details).  

\begin{definition}
Let $D \subset \mathbb{R}^d$ be a $C^2$ domain, and let $u : \partial D \to \mathbb{R}$ be a $C^1$ function. The tangential gradient of $u$ is defined by
\begin{align}\label{eq: def of tangential gradient}
    \nabla_T u = \nabla \tilde{u} - \frac{\partial \tilde{u}}{\partial \nu}\nu,
\end{align}
where $\tilde{u}$ is any $C^1$ extension of $u$ to a neighborhood of $\partial D$. The definition is independent of the particular choice of the extension~$\tilde{u}$. Let $F : \partial D \rightarrow \mathbb{R}^d$ be a vector field.   Its tangential component is given by
\begin{align*}
    F_T = F - (F \cdot \nu)\nu = F - F_{\nu}\nu,
\end{align*}
and its tangential divergence is defined as
\begin{align*}
    \operatorname{div}_T F = \operatorname{div} \tilde{F} - (\nabla \tilde{F}\,\nu)\cdot \nu,
\end{align*}
where $\tilde{F}$ is any $C^1$ extension of $F$ to a neighborhood of $\partial D$.
\end{definition}

Next, let $T_t: \mathbb{R}^d \rightarrow \mathbb{R}^d$, $t \in [0, t_0)$, be a family of transformations describing the deformation of a reference domain $D \subset \mathbb{R}^d$. The deformed domain is given by
\begin{align*}
    D_t := T_t(D) = \{T_t(x) : x \in D\},
\end{align*}
and its boundary satisfies
\begin{align*}
    \partial D_t := T_t(\partial D) = \{T_t(x) : x \in \partial D\}.
\end{align*} 
Throughout this section, we consider small perturbations of the identity mapping of the form
\begin{align*}
    x_t := T_t(x) = x + t\,\velocity(x),
\end{align*}
where $\velocity$ is a smooth vector field. Consequently, $x_t'=\velocity(x_t)$ with initial condition $x_0=x$, where $x_t'$ denotes the derivative with respect to $t$. The volume Jacobian is defined by $\gamma_t=\det(\nabla T_t)$, so that $\gamma_0=1$, where $\nabla T_t$ is the Jacobian matrix with entries $(\nabla T_t)_{ij}=\partial_j T_{t,i}$. The associated surface Jacobian is $s_t:=\gamma_t\,\lvert(\nabla T_t)^{-\top}\nu\rvert$, hence $s_0=1$, where $\nu$ denotes the outward unit normal vector on $\partial D$.

The following results and definitions in shape sensitivity analysis hold (cf. \cite{delfour2011shapes,sokolowski1992introduction,haslinger2003introduction}).
\begin{lemma}\label{lem: derivative of s}
Let $\velocity$ be a smooth vector field, then the derivative of Jacobian determinant and surface Jacobian determinant are given by
\begin{align*}
    \gamma^{\prime}= \lim_{t \rightarrow 0} \frac{\gamma_t - \gamma_0}{t} = \operatorname{div}\velocity,\quad
    s^\prime = \lim_{t \rightarrow 0} \frac{s_t - s_0}{t} 
    = \operatorname{div}\velocity - (\nabla \velocity\, \nu)\cdot \nu
    = \operatorname{div}_T \velocity.
\end{align*}
\end{lemma}

\begin{definition}
Let $D_t \mapsto u_t \in H^1(D_t)$ be a function defined on the deforming domains $D_t$.  
We say that $u_t$ admits a \emph{material (Lagrangian) derivative} $\dot{u}_t \in H^1(D)$ at $D$ if the transported function is well defined on $D$.  
The material derivative of $u$ in the direction $\velocity$ is given by
\begin{align}
    \dot{u} = \dot{u}(D; \velocity) := \lim_{t \rightarrow 0} \frac{u_t(x_t) - u(x)}{t}.
\end{align}
The corresponding \emph{shape derivative} (Eulerian derivative) of $u$ in the direction $\velocity$ is
\begin{align}
    u' = \dot{u} - \nabla u \cdot \velocity.
\end{align}
If $u_t := u(\partial D_t)$ is defined on $\partial D_t$, the material and shape derivatives, denoted by $\dot{u}$ and $u'$, satisfy
\begin{align}
    u' = \dot{u} - \nabla_T u \cdot \velocity.
\end{align}
\end{definition}

\begin{definition}
For any vector field $\velocity \in C([0, t_0]; \mathbb{R}^d)$, the Eulerian derivative of a shape functional $J(D)$ at $D$ in the direction $\velocity$ is defined by
\begin{align*}
    J^{\prime}(D; \velocity) = \lim_{t \rightarrow 0} \frac{J(D_t) - J(D)}{t}.
\end{align*}
\end{definition}

Before deriving the shape derivatives of the eigenvalues and eigenvectors, we first state several auxiliary results. The proof of \Cref{lem: shape derivative for general framework} is given in \Cref{appdx: proof of shape derivative for general form}.

\begin{lemma}\label{lem: shape derivative for general framework}
    We assume that, when restricted to the exterior domain $Y_{\infty}\setminus\overline{D}$, the functions $u$ and $v$ solve the exterior periodic Laplace problem \eqref{eq: PDE solution of capacity form} and satisfy the boundary conditions $u=f$ and $v=g$ on $\partial D$. Likewise, when restricted to the interior domain $D$, the functions $u$ and $v$ solve the interior Laplace problem and satisfy the same boundary conditions $u=f$ and $v=g$ on $\partial D$. Let $J_1$ and $J_2$ be two shape functionals defined in the domains $Y_{\infty} \setminus \overline{D}$ and $D$, respectively, and given by
    \begin{align}\label{eq: def of J1 and J2}
        J_1(D) := \int_{Y_{\infty} \setminus \overline{D}} 
            \nabla u \cdot \nabla v\, dx, \qquad
        J_2(D) := \int_{D}
            \nabla u \cdot \nabla v\, dx.
    \end{align}
     The shape derivative of $J_1$ and $J_2$ are given by
    \begin{align}
        J'_1(D;\theta) =\; & \int_{\partial D}
        \left(
            \left(\frac{\partial u}{\partial \nu}\frac{\partial v}{\partial \nu}\right)\Big|_{+}
            -\left(\frac{\partial u}{\partial \nu}\frac{\partial g}{\partial \nu}+\frac{\partial v}{\partial \nu}\frac{\partial f}{\partial \nu}\right)\Big|_{+} -\left(\nabla_T u \cdot \nabla_T v\right)
        \right)(\theta \cdot \nu)\, d\sigma, \label{eq: shape derivative for J1} \\
        J'_2(D;\theta) =\; & \int_{\partial D}
        \left(
            \left(\nabla_T u \cdot \nabla_T v\right)+\left(\frac{\partial u}{\partial \nu}\frac{\partial g}{\partial \nu}+\frac{\partial v}{\partial \nu}\frac{\partial f}{\partial \nu}\right)\Big|_{-}
            - \left(\frac{\partial u}{\partial \nu}\frac{\partial v}{\partial \nu}\right)\Big|_{-}
        \right)(\theta \cdot \nu)\, d\sigma. \label{eq: shape derivative for J2}
    \end{align}
    Furthermore, if $u$ and $v$ are defined in $Y_{\infty} \setminus \overline{D}$ and $D$, we have
    \begin{align}\label{eq: shape derivative for J12}
        J'_1(D;\theta)+J'_2(D;\theta) = \int_{\partial D}
        \left(\left[\frac{\partial u}{\partial \nu}\frac{\partial v}{\partial \nu}\right]
            -\left[\frac{\partial u}{\partial \nu}\frac{\partial g}{\partial \nu}+\frac{\partial v}{\partial \nu}\frac{\partial f}{\partial \nu}\right]
        \right)(\theta \cdot \nu)\, d\sigma,
    \end{align}
    where, for any quantity $w$ admitting interior and exterior traces on $\partial D$, we denote its jump by $[w]:=w|_{+}-w|_{-}$.
\end{lemma}

We now derive the shape derivatives of the eigenvalues and eigenvectors based on the variational characterization of the capacitance matrix in \eqref{eq: def of capacity matrix by PDE}. The proof of \Cref{thm: shape derivative of eigenvalue} is given in \Cref{appdx: shape derivative of eigenvalue}.

\begin{theorem}\label{thm: shape derivative of eigenvalue}
The shape derivatives of the capacitance matrix $C(D)$ in \eqref{eq: def of capacity matrix by PDE} and the volume matrix $V(D)$ admit the following representations:
\begin{align}
\label{eq: shape derivative of C in varaitional form}
C'(D;\theta)&=\int_{\partial D}g^C \, (\theta\cdot\nu)\, d\sigma,\quad g^C_{ij}=\left(\frac{\partial v_i}{\partial \nu}\frac{\partial v_j}{\partial \nu}\right)\Big{|}_{+},\\
\label{eq: shape derivative of V}
V'(D;\theta)&=\int_{\partial D}g^V \, (\theta\cdot\nu)\, d\sigma,\quad 
 g^V_{ij}=\chi_{\partial D_i}\chi_{\partial D_j},
\end{align}
where $g^C$ and $g^V$ are matrix valued functions defined on $\partial D$. Let $\lambda_j$ and $u_j$ depend on the shape of the domain $D$ and satisfy \eqref{eq: eigenvalue problem of C}. Then, their shape derivatives admit the boundary integral representations
\begin{equation}\label{eq: shape derivative of eigenpair}
    \lambda_j^\prime(D;\theta)= \int_{\partial D} g^{\lambda,0}_j (\theta\cdot\nu)\, d\sigma,
    \quad u_j^\prime(D;\theta)= \int_{\partial D} g^{u}_j (\theta\cdot\nu)\, d\sigma,
\end{equation}
where the scalar function $g^{\lambda,0}_j$ and the vector valued function $g^{u}_j$ are given by
\begin{align}
    g^{\lambda,0}_j=u_j^{\top}(g^C-\lambda_j g^V)u_j,\quad 
    g^{u}_j = \sum_{i\neq j}\frac{u_j^{\top}(g^C-\lambda_j g^V)u_i}{\lambda_j-\lambda_i} u_i-\frac{1}{2}(u_j^\top g^Vu_j)u_j.
\end{align}
Furthermore, the derivative of $\lambda_{j,1}(D)$, as defined in \eqref{eq: lambdaj1 omega}, is given by
\begin{align}\label{eq: shape derivative of lambdaj1}
\lambda_{j,1}^\prime(D;\theta) = \int_{\partial D}  g^{\lambda,1}_j (\theta\cdot\nu)\, d\sigma,\quad  g^{\lambda,1}_j=\frac{2\tau_m(m^{\top}u_j)}{|Y|}  \left( m^{\top}g^{u}_j+u_j^{\top} g^{m}\right).
\end{align}
Here, the shape derivative $m_j^{\prime}(D;\theta)$ is given by
\begin{align}\label{eq: shape derivative of m in varaitional form}
    m_j^{\prime}(D;\theta) = \int_{\partial D}
       g^{m}_j (\theta\cdot\nu)  \mathrm{d}\sigma,\quad g^{m}_j=
       \left[\frac{\partial w}{\partial \nu} \frac{\partial v_j}{\partial \nu}\right]-\nu_d\left[\frac{\partial v_j}{\partial \nu}\right],
\end{align}
where $w$ solves the exterior periodic Laplace problem \eqref{eq: PDE solution of capacity form} in $Y_{\infty}\setminus\overline{D}$ and the interior Laplace problem in $D$, with the Dirichlet boundary condition $w=x_d$ on $\partial D$ in both cases.
\end{theorem}

\begin{remark}
Let $\widetilde{\psi}$ be the solution of the boundary integral equation $\mathcal{S}^{0,0}_{D,s}[\widetilde{\psi}]=x_d$ on $\partial D$, and then $w$ is given by $w=\mathcal{S}^{0,0}_{D,s}[\widetilde{\psi}]$. Then the shape derivatives of $C_{ij}(D)$ and $m_j(D)$ admit boundary density representations obtained by applying the jump relations for $v_i$ and $w$:
\begin{align}\label{eq: shape derivative of C in BIE form}
C_{ij}^\prime(D;\theta) =& \int_{\partial D} \Big( \psi_i (\mathcal{K}^{0,0}_{D,s})^*[\psi_j] +\psi_j (\mathcal{K}^{0,0}_{D,s})^*[\psi_i] \Big)(\velocity \cdot \nu)\, d\sigma,\\
\label{eq: shape derivative of m in BIE form}
m_j^\prime(D;\theta) =&\int_{\partial D}
\Big( \psi_j (\mathcal{K}^{0,0}_{D,s})^*[\widetilde{\psi}] +\widetilde{\psi} (\mathcal{K}^{0,0}_{D,s})^*[\psi_j] -\nu_d\psi_j \Big) (\velocity \cdot \nu)\, d\sigma.
\end{align}
Here, we have used in the derivation of \eqref{eq: shape derivative of C in BIE form} the fact that $\left(-\frac{1}{2}\mathcal{I}+(\mathcal{K}^{0,0}_{D,s})^*\right)[\psi_i]=0$.
\end{remark}

\begin{remark}
    In the single particle case ($N=1$), the general shape derivative formula simplifies considerably. More precisely, the variational representation reduces to
    \begin{align*}
    \lambda^\prime(D;\theta)
    =
    \frac{1}{|D|}
    \int_{\partial D}
    \left(
    \left(\frac{\partial v}{\partial \nu}\Big{|}_{+}\right)^2
    -
    \frac{\mathrm{Cap}(D)}{|D|}
    \right)
    (\theta\cdot\nu)\, d\sigma.
    \end{align*}
    Equivalently, one obtains the boundary integral representation
    \begin{equation*}
    \lambda^\prime(D;\theta)
    =
    \frac{1}{|D|}
    \int_{\partial D}
    \left(
    2 \psi (\mathcal{K}^{0, 0}_{D, s})^*[\psi]
    -
    \frac{\mathrm{Cap}(D)}{|D|}
    \right)
    (\theta\cdot\nu)\, d\sigma.
    \end{equation*}
\end{remark}

\begin{proposition}
Let $r(\omega,D;\theta)$ denote the approximate reflection coefficient given in \eqref{eq: asymptotic of reflection coefficient}. Then, its shape derivative in the direction $\theta$ admits the boundary integral representation
\begin{align}
r^{\prime}(\omega,D;\theta)= \int_{\partial D} g^{r}(\omega)\,(\theta\cdot\nu)\, d\sigma,
\end{align}
where
\begin{align*}
g^{r}(\omega)
= - \sum_{j=1}^{N}\frac{2\mathrm{i}\omega \left((\lambda_{j}-\lambda(\omega))\,g^{\lambda,1}_j- \lambda_{j,1}\,g^{\lambda,0}_j\right)}{\bigl(\lambda_j - \mathrm{i}\omega \lambda_{j,1}-\lambda(\omega)\bigr)^2}\bigl(1 + \mathcal{O}(\omega)\bigr)
+ \mathcal{O}(\omega).
\end{align*}
Moreover, the shape derivatives of $J^{\mathrm{ref}}$ and $J^{\mathrm{res}}$ can be written in the form
\begin{equation}\label{eq: shape derivative of Jref}
(J^{\mathrm{ref}})^{\prime}(D;\theta)
= \int_{\partial D} g^{\mathrm{ref}}\,(\theta \cdot \nu)\, d\sigma,
\qquad
g^{\mathrm{ref}}
=\frac{2}{\omega_{\max}-\omega_{\min}}
\int_{\omega_{\min}}^{\omega_{\max}} \Re\!\left( \overline{r(\omega)}\, g^{r}(\omega) \right) \, d\omega ,
\end{equation}
and
\begin{equation}\label{eq: shape derivative of Jres}
(J^{\mathrm{res}})^{\prime}(D;\theta)
= \int_{\partial D} g^{\mathrm{res}}\,(\theta \cdot \nu)\, d\sigma,
\end{equation}
with
\begin{align}
g^{\mathrm{res}}
=\frac{2}{M}\sum_{j=1}^M\Bigg[
\left(\frac{\lambda_j}{\Re{\lambda(\omega^{*}_j)}}-1\right)\frac{g^{\lambda,0}_j}{\Re{\lambda(\omega^{*}_j)}}
+
\left(\frac{\omega_j^{*}\lambda_{j,1}}{\Im{\lambda(\omega^{*}_j)}}-1\right)\frac{\omega_j^{*}g^{\lambda,1}_j}{\Im{\lambda(\omega^{*}_j)}}
\Bigg].
\end{align}
\end{proposition}

\subsection{Optimal design algorithm}

In the previous sections, we have derived shape derivative formulas in a general, nonparametric framework. To obtain a practical optimization algorithm, we now let $d=2$ and restrict the admissible shapes to a finite dimensional, parametrized family.

For each resonator $D_j\subset\mathbb{R}^{2}$, we assume that $D_j$ is star shaped with respect to a reference point $x_{0,j}\in\mathbb{R}^{2}$. Its boundary $\partial D_j$ is parametrized in polar coordinates as
\begin{align*}
    x(t;p_j)=x_{0,j}+r(t;p_j)(\cos t,\sin t), \qquad t \in[0,2\pi].
\end{align*}
We approximate the radial function $r(t;p_j)$ by a truncated Fourier series,
\begin{align*}
    r(t;p_j)=a_{0,j}\left(1+\frac{1}{2M}\sum_{i=1}^{M}\left(a_{i,j}\cos(it)+b_{i,j}\sin(it)\right)\right),
    \qquad t\in[0,2\pi],
\end{align*}
where $a_{0,j}>0$ controls the overall size and the remaining Fourier coefficients describe shape perturbations. The associated parameter vector is
\begin{align*}
    p_j=\bigl(x_{0,j},a_{0,j},a_{1,j},\dots ,a_{M,j},\,b_{1,j},\dots ,b_{M,j}\bigr)^{\top}.
\end{align*}
Collecting the parameters for all resonators, we define the global design vector
\begin{align*}
    p=\bigl(p_1,\dots,p_M\bigr)^{\top}.
\end{align*}
With this parametrization, the shape optimization problem reduces to a finite dimensional optimization problem in the design variable $p$.

To compute gradients with respect to the parameters, we relate parametric derivatives to the general shape derivative. For a given parameter $p_{i,j}$, we choose the deformation field as the parametric sensitivity of the boundary mapping,
$\theta=\frac{\partial x}{\partial p_{i,j}}$. By the chain rule, the corresponding partial derivative of the objective satisfies $\frac{\partial J}{\partial p_{i,j}} = J'(D;\theta)$, which can be evaluated using the Hadamard boundary integral representation derived earlier. This provides a direct way to compute the gradient $\nabla_p J$ and thus enables the use of standard gradient based optimization methods. The resulting optimal shape design procedure is summarized in \Cref{algo: optimal design}.

\begin{algorithm}[!h]
\caption{Optimal shape design.}
\label{algo: optimal design}
\begin{algorithmic}[1]
\REQUIRE Material parameters $\rho_m$, $\rho_b$, $\kappa_m$, $\kappa_b$; periodic cell $Y$; incident direction $\hat{\theta}$; frequency interval $(\omega_{\min}, \omega_{\max})$; initial design $p^{(0)}$; objective functional $J$.
\FOR{$n = 0,1,\dots,N-1$}
    \STATE Construct the capacitance matrix $C$ using the boundary integral formulation \eqref{eq: def of capacity matrix by BIE};
    \STATE Solve the generalized eigenvalue problem \eqref{eq: eigenvalue problem of C} to obtain eigenpairs $(\lambda_j, u_j)$;
    \STATE Compute the shape derivatives of $\lambda_j$ and $\lambda_{j,1}$ using \eqref{eq: shape derivative of eigenpair} and \eqref{eq: shape derivative of lambdaj1};
    \STATE Evaluate the shape gradient of $J$ using \eqref{eq: shape derivative of Jref} or \eqref{eq: shape derivative of Jres};
    \STATE Update the design from $p^{(n)}$ to $p^{(n+1)}$ using a gradient based method.
\ENDFOR
\ENSURE Optimized design parameters $p^{(N)}$.
\end{algorithmic}
\end{algorithm}

\section{Numerical experiments} \label{sec:numerics}

In this section, we present numerical examples to validate the accuracy of the reduced order model in the subwavelength regime and to demonstrate broadband absorption optimal design based on the reduced order model. For simplicity, we restrict ourselves to the two-dimensional case $d=2$. Unless otherwise stated, we adopt the following parameter settings throughout:
\begin{itemize} 
    \item Incident and reflection directions: $\hat{\theta}_{\pm}=(0,\pm 1)$;
    \item Unit cell: $Y=[-L/2,L/2]$ with period $L=20$;
    \item Material parameters: $v_m=1$, $v_b=1-0.05\mathrm{i}$, and $\delta=0.001$;
    \item Frequency band of interest: $\omega_{\min}=0.01$ and $\omega_{\max}=0.1$;
    \item Shape parameters of the $j$\textsuperscript{th} resonator: $a_{0,j}\in[0.1,1]$ and $a_{i,j}, b_{i,j}\in[-1,1]$ for $1\le j\le N$ (number of resonators) and $1\le i\le M$ (number of shape parameters).
\end{itemize}

\subsection{Efficiency of the reduced order model}

To evaluate the accuracy and robustness of the reduced order model (ROM), we benchmark it against the exact full order solution. Specifically, we solve \eqref{eq: integral scattering problem} and then evaluate the reflection coefficient $r(\omega)$ using \eqref{eq: formula of reflection coefficient}. We begin with the simplest setting of a single resonator, and then consider configurations with multiple resonators, including arrangements with three and nine resonators. This progression allows us to examine whether the ROM retains its predictive accuracy as the number of geometric features increases and modal interactions become more pronounced.

Figure~\ref{fig:rom_accuracy_1_vs_3} compares the exact and ROM results for the single-resonator and three-resonator unit cells. For both configurations, the ROM reproduces the geometric description of the unit cell (top row) and yields absorptance spectra (bottom row) that closely match the exact solution across the frequency range of interest, capturing both the resonance location and the overall spectral trend. This agreement indicates that the reduced order model is already sufficiently expressive in the low-complexity regime and remains stable when the problem size increases from one to three resonators.

The assessment is further extended to more complex settings in \Cref{fig:rom_accuracy_3_vs_9}, where the three-resonator case is contrasted with a nine-resonator unit cell. Despite the substantially increased geometric complexity and potentially stronger coupling effects, the ROM continues to show excellent agreement with the exact solution in terms of both geometry representation and absorptance response. Importantly, the spectral comparison confirms that the ROM preserves the key physics governing the resonant behavior, namely accurate resonance frequencies and the corresponding absorptance levels, without introducing noticeable spurious features. Although small discrepancies remain between the ROM and the exact solution, they are within an acceptable range for practical prediction and design.

Overall, \Cref{fig:rom_accuracy_1_vs_3} and \Cref{fig:rom_accuracy_3_vs_9} demonstrate that the ROM maintains good accuracy as the geometry progresses from a simple single-resonator setting to more complex multi-resonator layouts with three and nine resonators. This consistent performance supports the use of the ROM as an efficient surrogate model for subsequent parametric sweeps and optimization studies, where repeated evaluations of the spectral response would otherwise be computationally prohibitive with the full order solver.

\begin{figure}[htbp]
    \centering
    % --- top row: geometry ---
    \subfloat[Single resonator: geometry in the unit cell]%
    {\includegraphics[width=.45\textwidth]{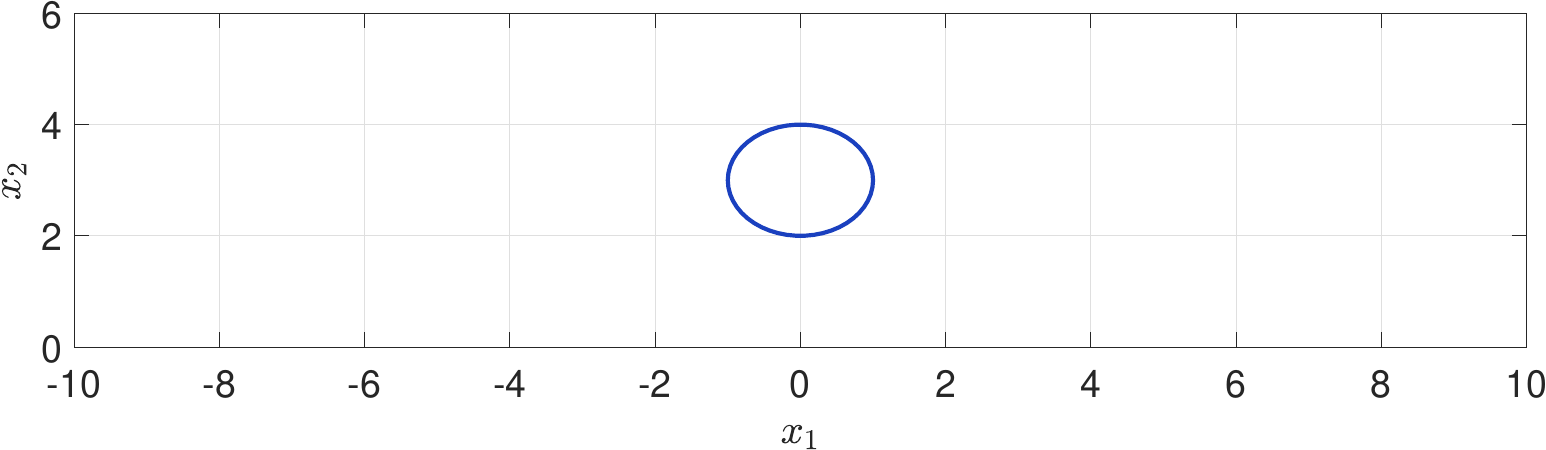}}
    \hfill
    \subfloat[Three resonators: geometry in the unit cell]%
    {\includegraphics[width=.45\textwidth]{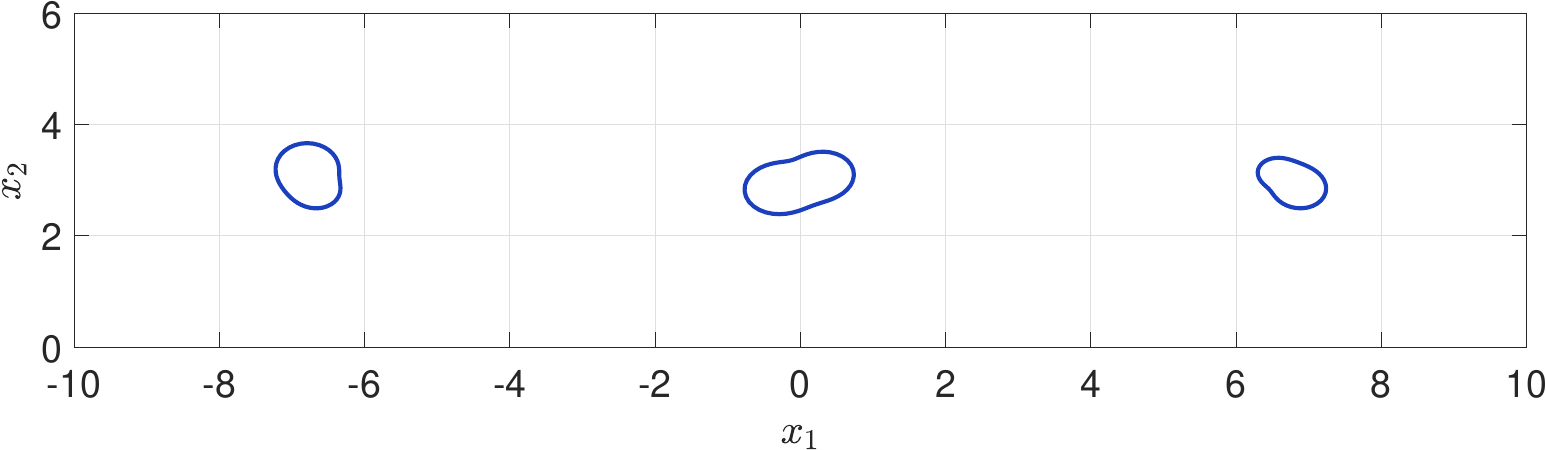}}\\[0.5ex]
    % --- bottom row: absorptance ---
    \subfloat[Single resonator: absorptance (exact vs.\ reduced order model)]%
    {\includegraphics[width=.45\textwidth]{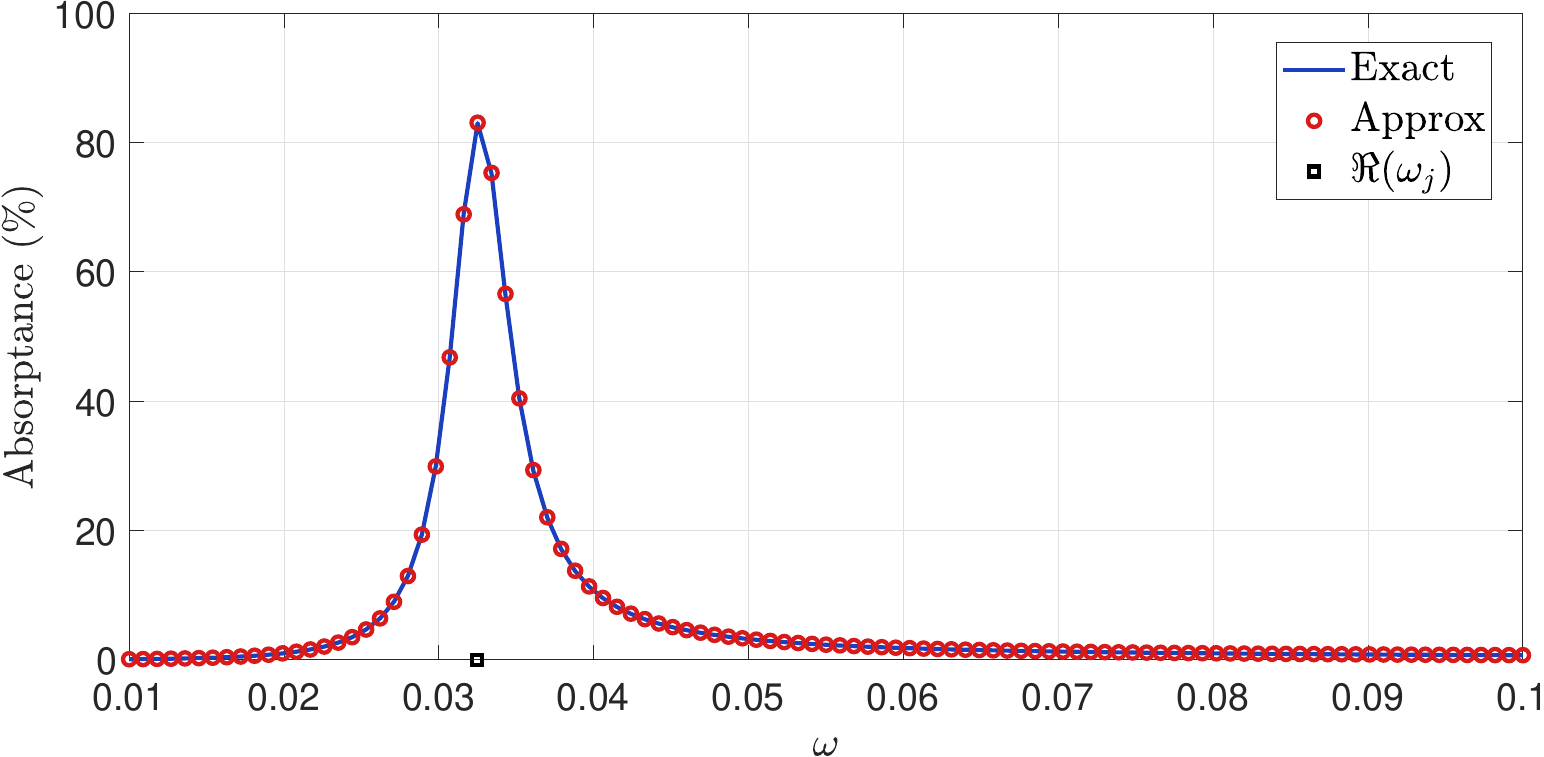}}
    \hfill
    \subfloat[Three resonators: absorptance (exact vs.\ reduced order model)]%
    {\includegraphics[width=.45\textwidth]{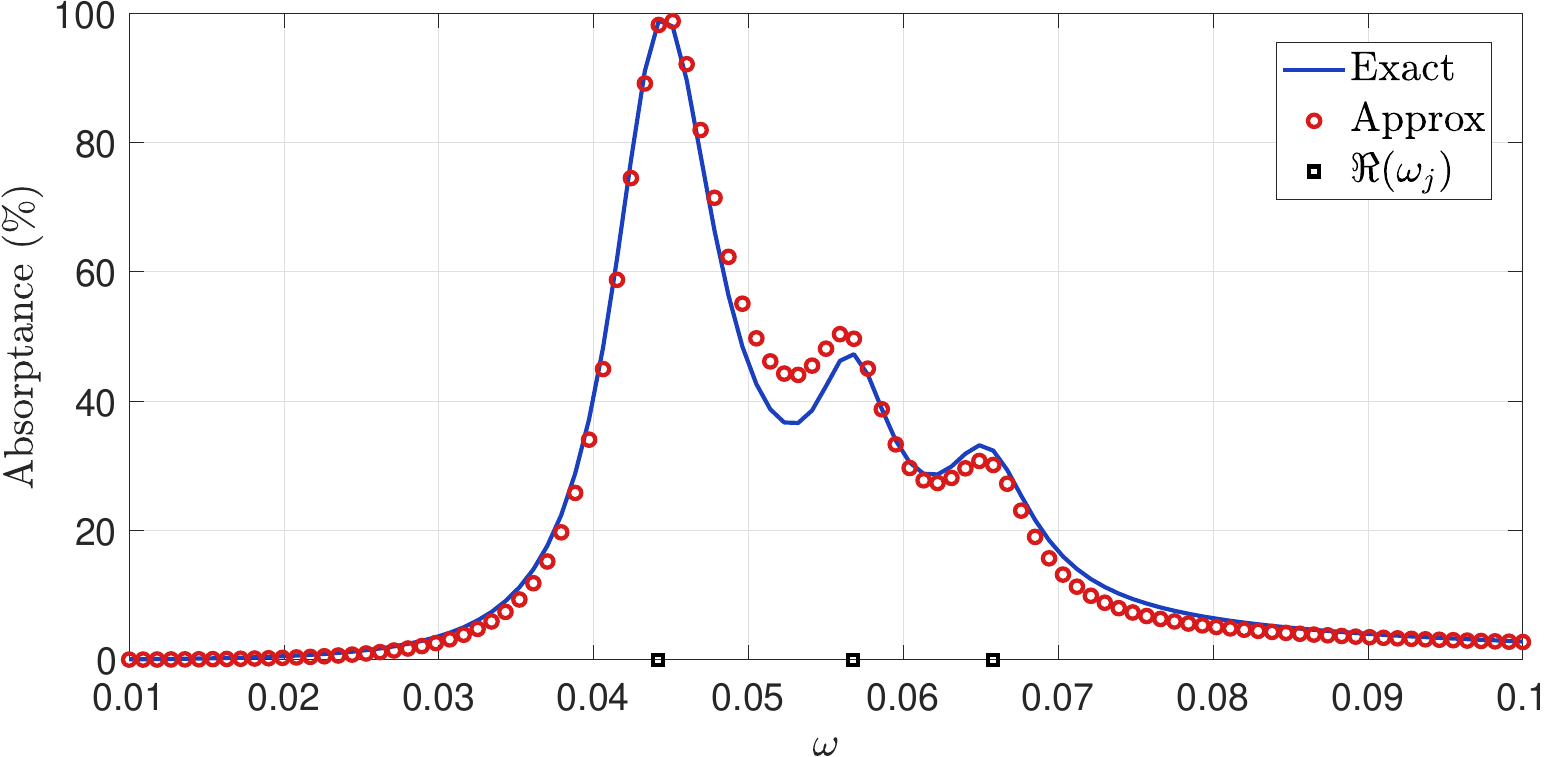}}
    \caption{Geometry (top) and absorptance spectra (bottom) for single and three resonator configurations (per unit cell), comparing the exact solution and the reduced order model.}
    \label{fig:rom_accuracy_1_vs_3}
\end{figure}
    
\begin{figure}[htbp]
    \centering
    % --- top row: geometry ---
    \subfloat[Three resonators: geometry in the unit cell]%
    {\includegraphics[width=.45\textwidth]{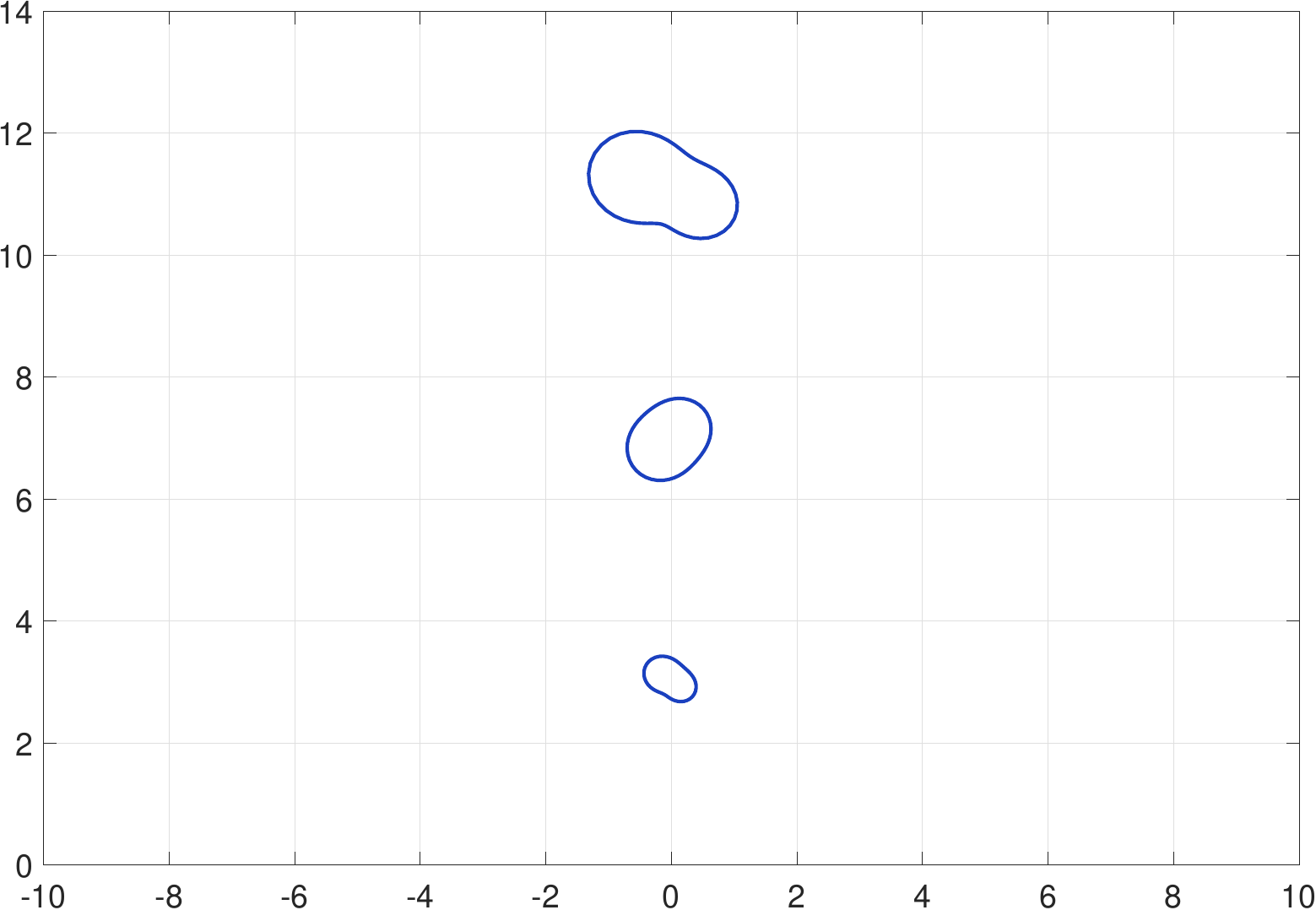}}
    \hfill
    \subfloat[Nine resonators: geometry in the unit cell]%
    {\includegraphics[width=.45\textwidth]{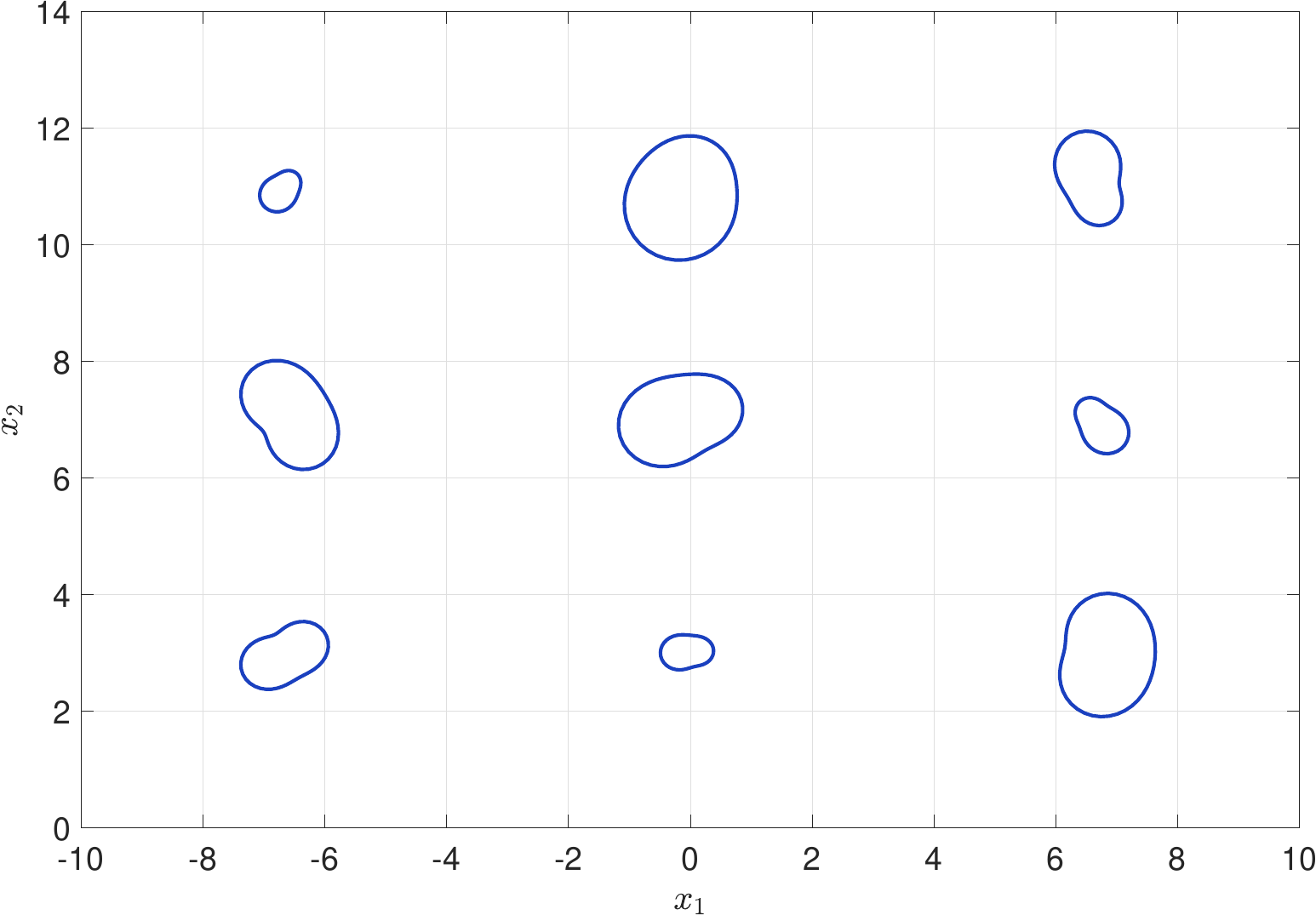}}\\[0.5ex]
    % --- bottom row: absorptance ---
    \subfloat[Three resonators: absorptance (exact vs.\ reduced order model)]%
    {\includegraphics[width=.45\textwidth]{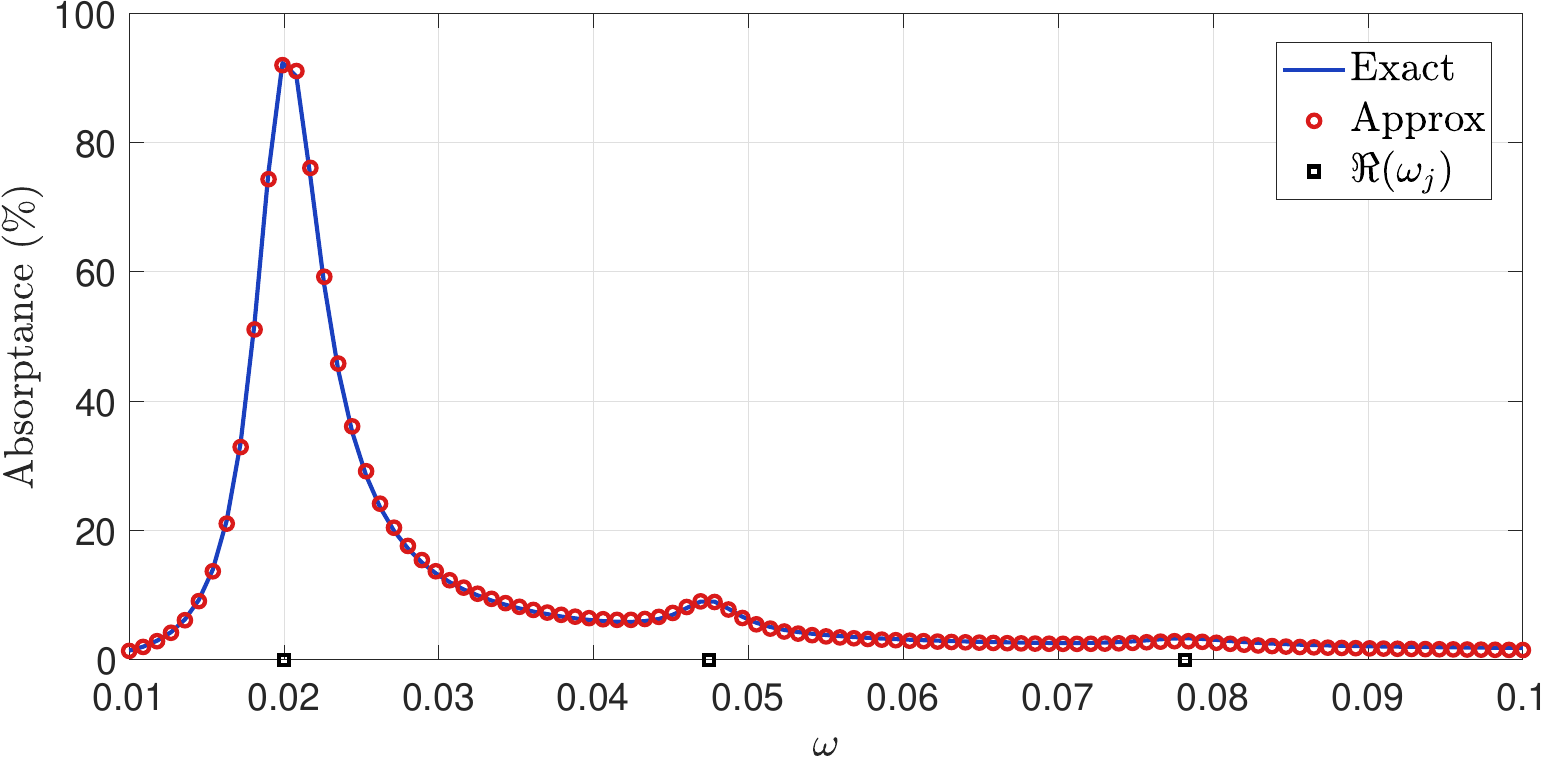}}
    \hfill
    \subfloat[Nine resonators: absorptance (exact vs.\ reduced order model)]%
    {\includegraphics[width=.45\textwidth]{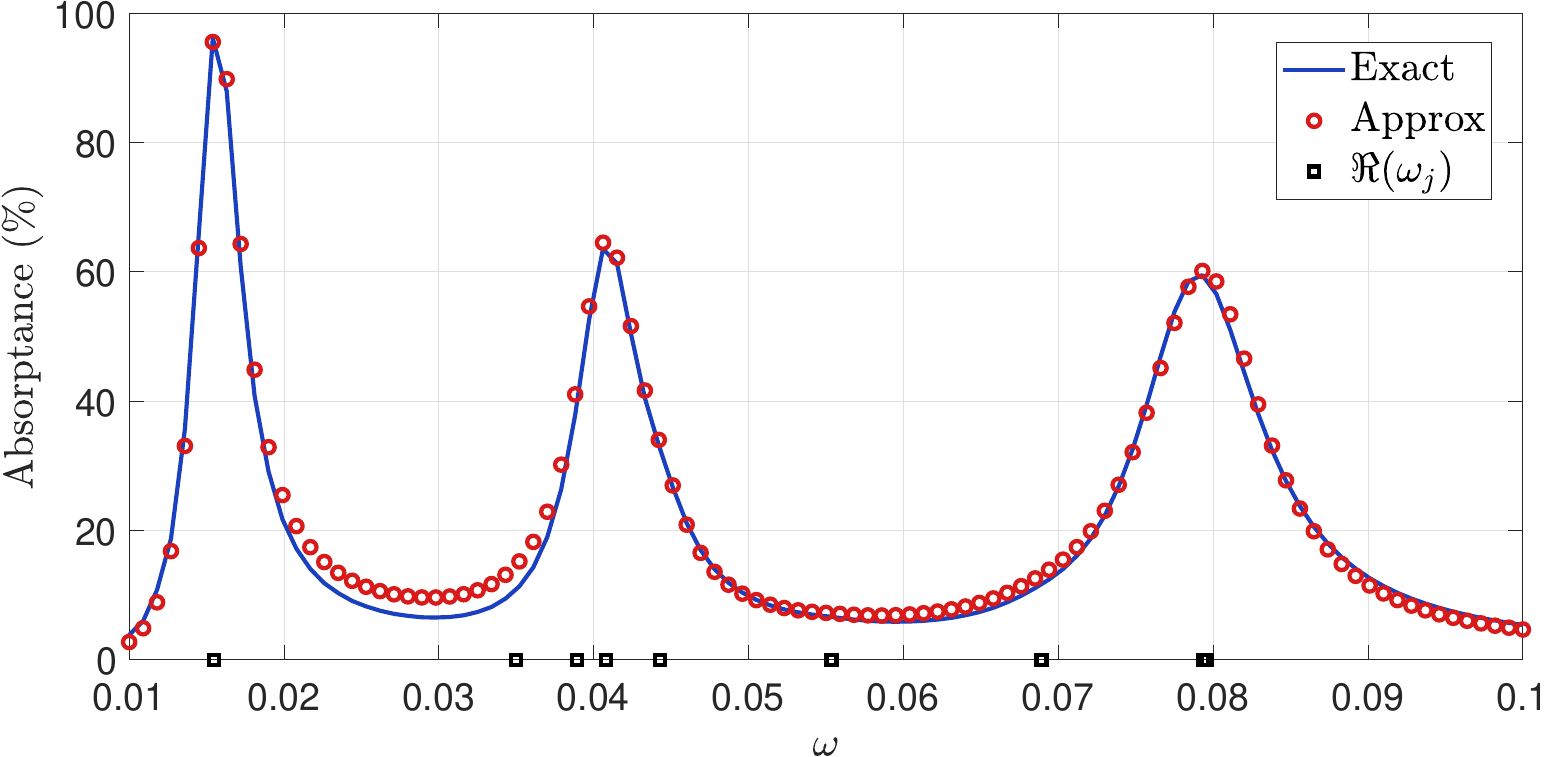}}
    \caption{Geometry (top) and absorptance spectra (bottom) for three and nine resonator configurations (per unit cell), comparing the exact solution and the reduced order model.}
    \label{fig:rom_accuracy_3_vs_9}
\end{figure}

\subsection{Shape optimization results}

We next present the shape optimization results for different layouts with increasing complexity: a single resonator ($1\times1$), three resonators arranged horizontally ($3\times1$) and vertically ($1\times3$), and nine resonators ($3\times3$) (see \Cref{fig:shape_opt_1x1_Jres_vs_Jref}--\Cref{fig:shape_opt_3x3_Jres_vs_Jref}). In each figure, the left column corresponds to the resonance driven objective $J^{\mathrm{res}}$ and the right column corresponds to the broadband reflectance objective $J^{\mathrm{ref}}$. The three panels respectively display the convergence history, the optimized deformation, and the absorptance spectra before and after optimization.

For the simplest configurations ($1\times1$, and to a large extent the $3$ resonator layouts), optimizing $J^{\mathrm{res}}$ yields a pronounced improvement in the spectral response. This is consistent with the role of $J^{\mathrm{res}}$ as a mode by mode enforcement of the critical coupling matching conditions: by driving $\Re \lambda(\omega_j^*)$ towards $\lambda_j$ and $\Im \lambda(\omega_j^*)$ towards $\omega_j^*\lambda_{j,1}$ at selected target frequencies, the optimization can produce sharp absorption peaks, and in several instances the absorptance approaches $100\%$ at the prescribed frequencies. In contrast, the functional $J^{\mathrm{ref}}$ measures the average of $|r(\omega)|^2$ over the entire band. As a result, its minimization tends to allocate the optimization effort to those portions of the spectrum that contribute most to the integral. In particular, low-frequency resonant features may be comparatively underemphasized if their contribution to the band averaged reflectance is small, so that improvements at the lowest resonances can be less pronounced even when the overall average reflectance decreases.

As the number of resonators increases, the optimization landscape becomes more intricate due to stronger inter resonator coupling and the resulting hybridization of resonant modes. This trend is most evident for the $3\times3$ configuration in \Cref{fig:shape_opt_3x3_Jres_vs_Jref}, where minimizing $J^{\mathrm{res}}$ becomes noticeably more challenging: the objective attempts to enforce multiple critical-coupling conditions simultaneously, but the corresponding resonant quantities are no longer weakly interacting. Consequently, adjusting the shape to improve the matching for one targeted resonance can deteriorate the matching for another, leading to a more constrained and less uniformly successful optimization outcome. In the same regime, the broadband objective $J^{\mathrm{ref}}$ is often more effective in producing an overall reduction of reflectance across the band, because it does not require a one-to-one alignment between prescribed target frequencies and individual resonance branches; instead, it exploits the collective response of the coupled system to reduce the integrated reflectance.

To avoid suboptimal local minima we use UniformAdam \cite{nicolet2021large} on the shape gradient density. We observe that the initial conditions strongly affect the result of the optimization of $J^\mathrm{ref}$ and $J^\mathrm{res}$. In \Cref{fig:shape_opt_comparison_J_ref} for the objective $J^\mathrm{ref}$, we can see that for an initial mesh consisting only of circles, the initial radii can lead to very different local minima. Similarly, we show that more vertical structures with adjusted, shorter periods tend to produce a broader range of high absorption compared to other configurations.
With $J^\mathrm{res}$, the dependence on initial conditions is even more pronounced in \Cref{fig:shape_opt_comparison_J_res}, particularly with different initial radii, where a larger initial radius produces a significantly better result than smaller radii.

Finally, these observations also clarify an important practical point regarding broadband design. While the motivating discussion suggests that having many resonance frequencies with real parts ``densely distributed'' in $[\omega_{\min},\omega_{\max}]$ can promote peak overlap, the numerical results indicate that a strictly uniform distribution of $\Re\omega_j(\delta)$ is not always necessary (and can even be suboptimal) once multiple resonators interact. In strongly coupled configurations, the relevant absorption features arise from collective modes whose locations and widths are shaped by coupling and loss, so that broadband superabsorption can be achieved through an appropriate coupled spectral arrangement rather than by enforcing an \emph{a priori} uniform spacing of the resonance real parts.

\begin{figure}[htbp]
    \centering
    \subfloat[Convergence history: $J^{\mathrm{res}}$ (left) and $J^{\mathrm{ref}}$ (right)]%
    {\includegraphics[width=.90\textwidth]{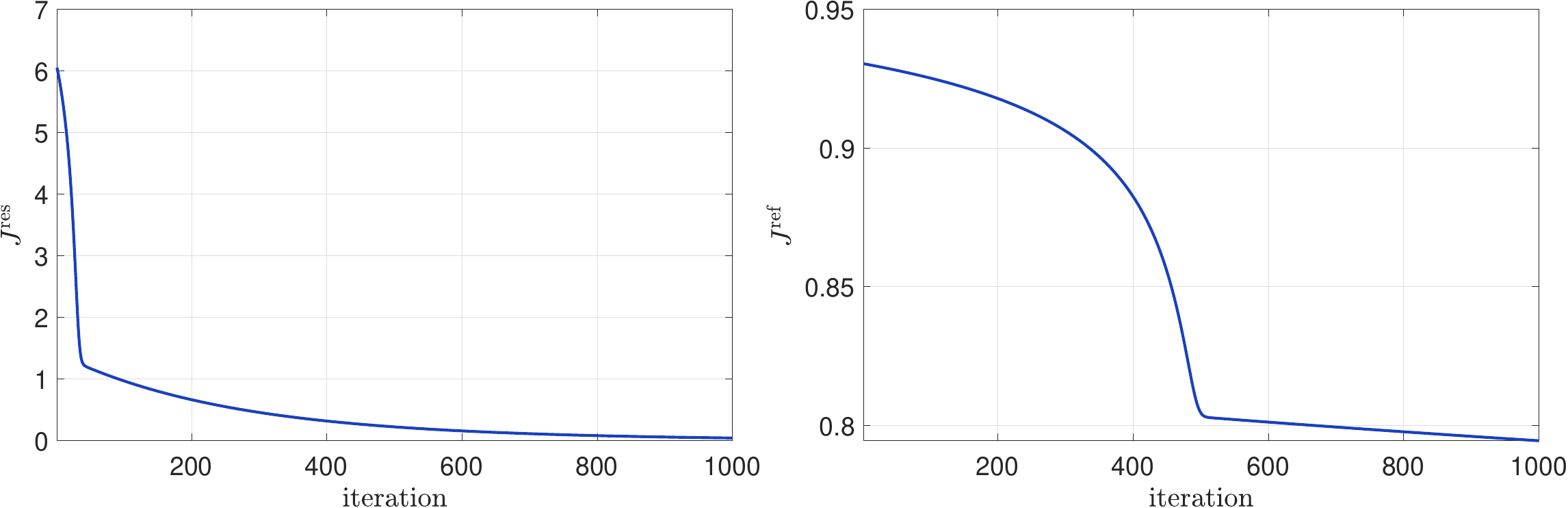}}\\[0.5ex]
    \subfloat[Shape deformation: optimized shapes for $J^{\mathrm{res}}$ (left) and $J^{\mathrm{ref}}$ (right)]%
    {\includegraphics[width=.90\textwidth]{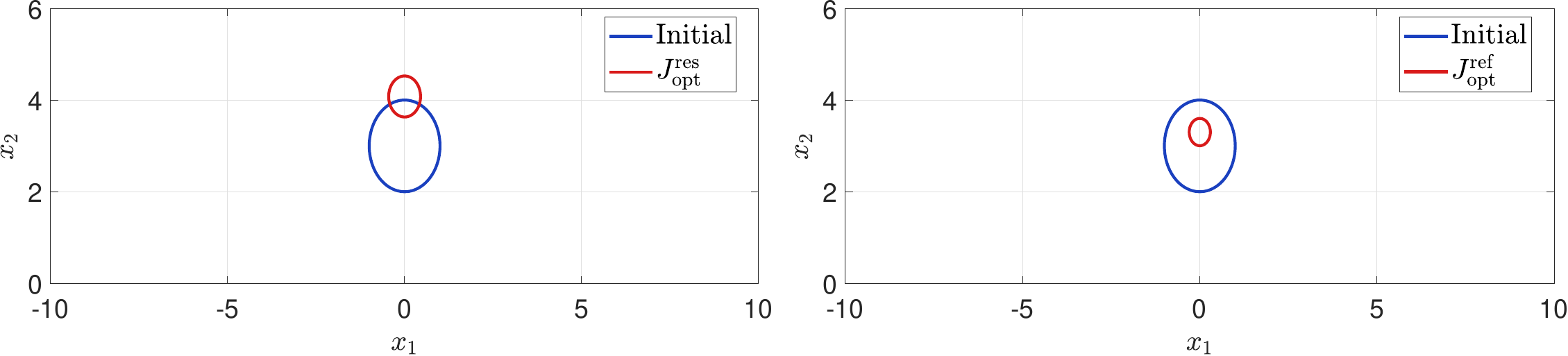}}\\[0.5ex]
    \subfloat[Absorptance spectra: initial vs.\ optimized for $J^{\mathrm{res}}$ (left) and $J^{\mathrm{ref}}$ (right)]%
    {\includegraphics[width=.90\textwidth]{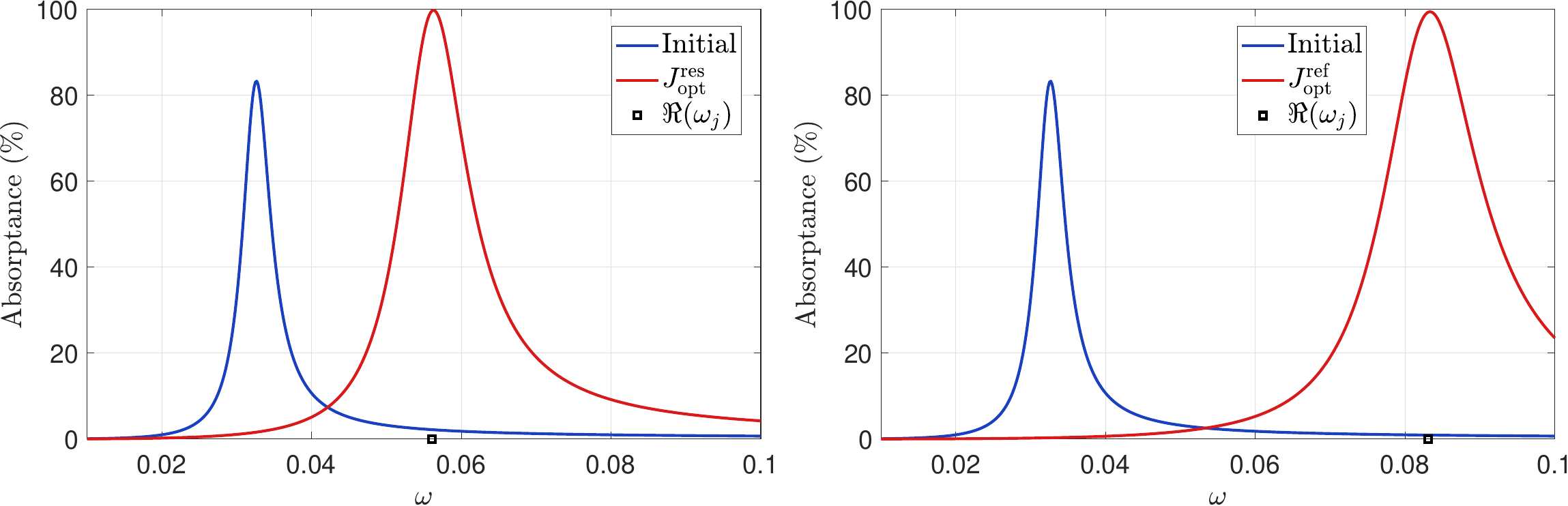}}
    \caption{Shape optimization results for a single resonator configuration ($1\times 1$). The left column corresponds to $J^{\mathrm{res}}$ and the right column to $J^{\mathrm{ref}}$.}
    \label{fig:shape_opt_1x1_Jres_vs_Jref}
\end{figure}

\begin{figure}[htbp]
    \centering
    \subfloat[Convergence history: $J^{\mathrm{res}}$ (left) and $J^{\mathrm{ref}}$ (right)]%
    {\includegraphics[width=.90\textwidth]{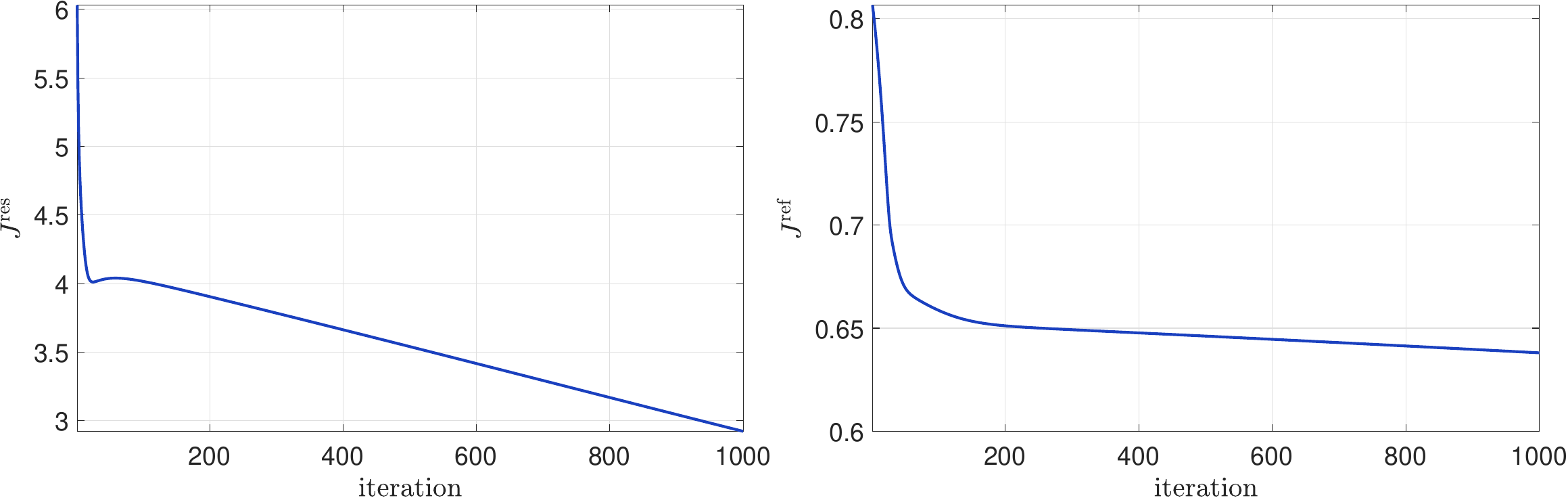}}\\[0.5ex]
    \subfloat[Shape deformation: optimized shapes for $J^{\mathrm{res}}$ (left) and $J^{\mathrm{ref}}$ (right)]%
    {\includegraphics[width=.90\textwidth]{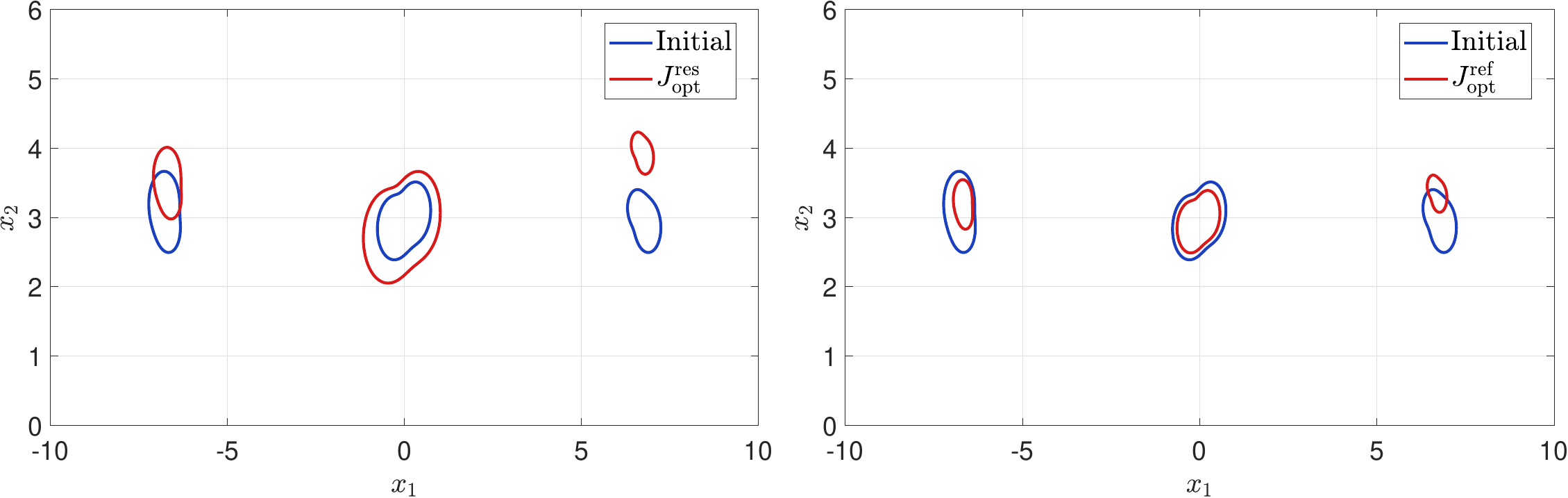}}\\[0.5ex]
    \subfloat[Absorptance spectra: initial vs.\ optimized for $J^{\mathrm{res}}$ (left) and $J^{\mathrm{ref}}$ (right)]%
    {\includegraphics[width=.90\textwidth]{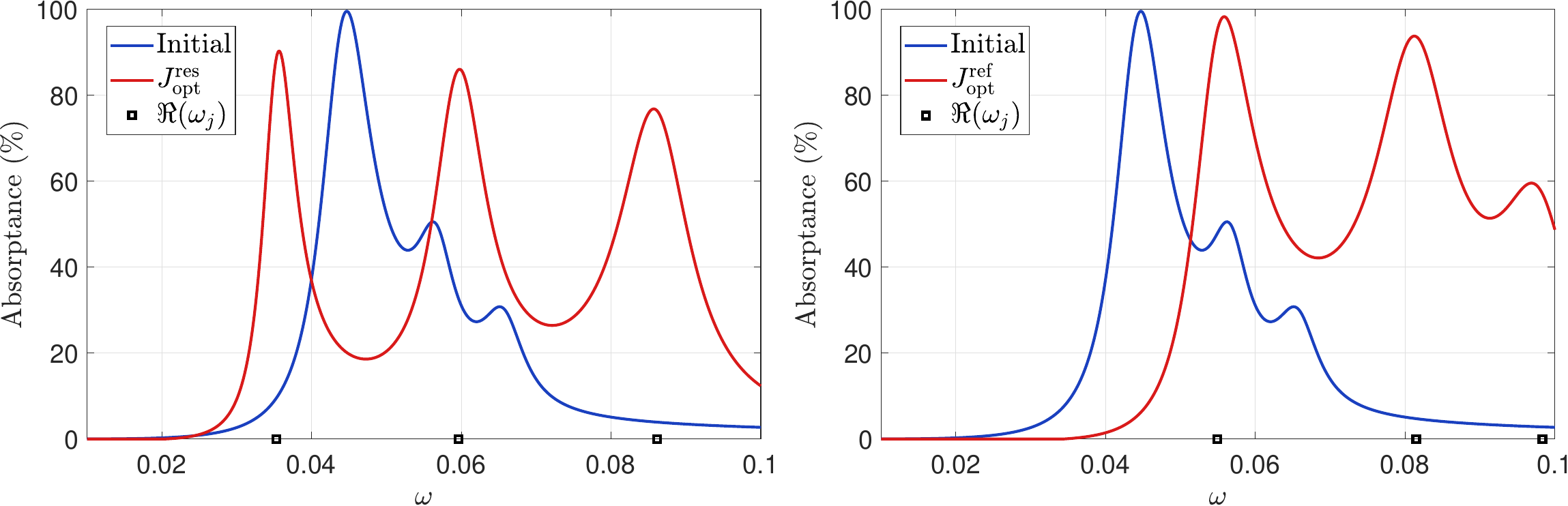}}
    \caption{Shape optimization results for a three resonators configuration arranged horizontally ($3\times 1$). The left column corresponds to $J^{\mathrm{res}}$ and the right column to $J^{\mathrm{ref}}$.}
    \label{fig:shape_opt_3x1_Jres_vs_Jref}
\end{figure}

\begin{figure}[htbp]
    \centering
    \subfloat[Convergence history: $J^{\mathrm{res}}$ (left) and $J^{\mathrm{ref}}$ (right)]%
    {\includegraphics[width=.90\textwidth]{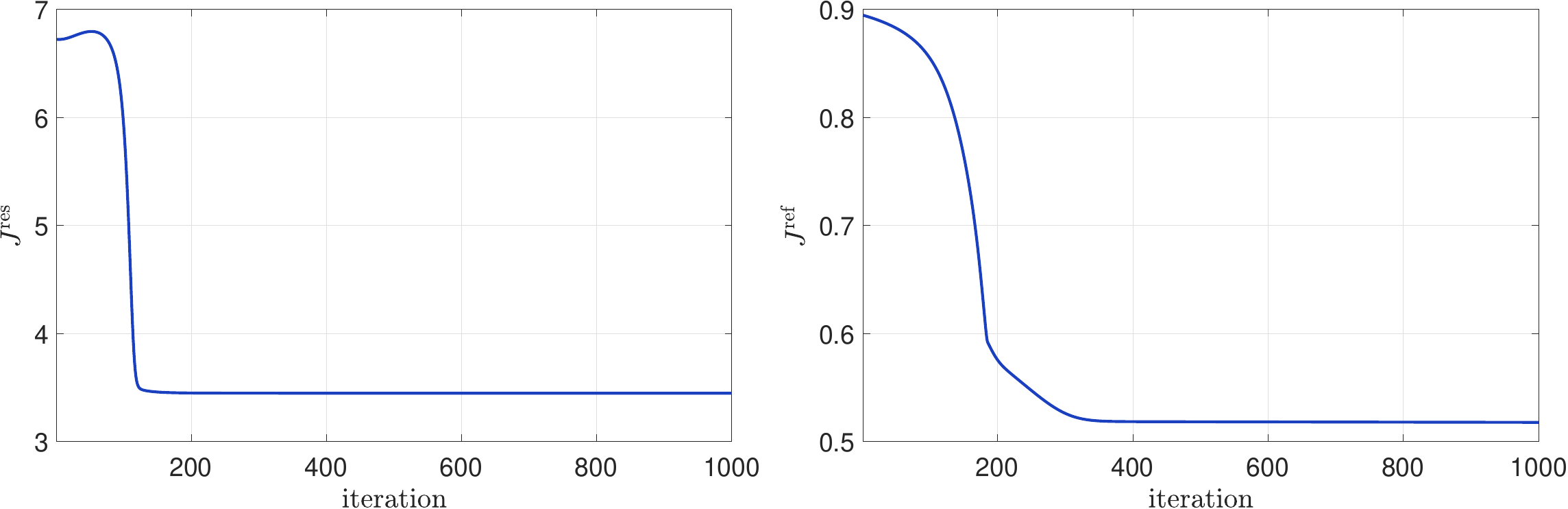}}\\[0.5ex]
    \subfloat[Shape deformation: optimized shapes for $J^{\mathrm{res}}$ (left) and $J^{\mathrm{ref}}$ (right)]%
    {\includegraphics[width=.90\textwidth]{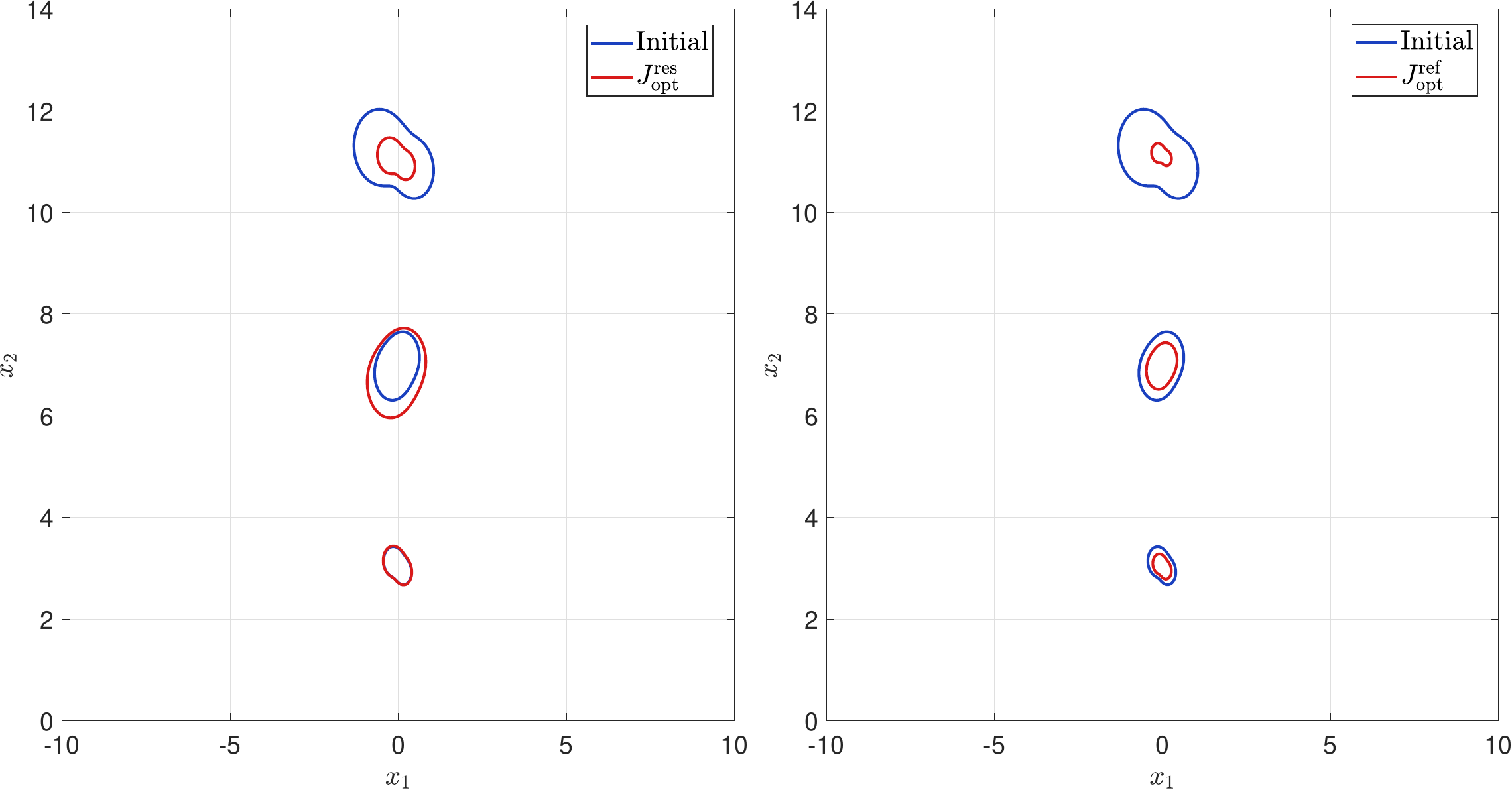}}\\[0.5ex]
    \subfloat[Absorptance spectra: initial vs.\ optimized for $J^{\mathrm{res}}$ (left) and $J^{\mathrm{ref}}$ (right)]%
    {\includegraphics[width=.90\textwidth]{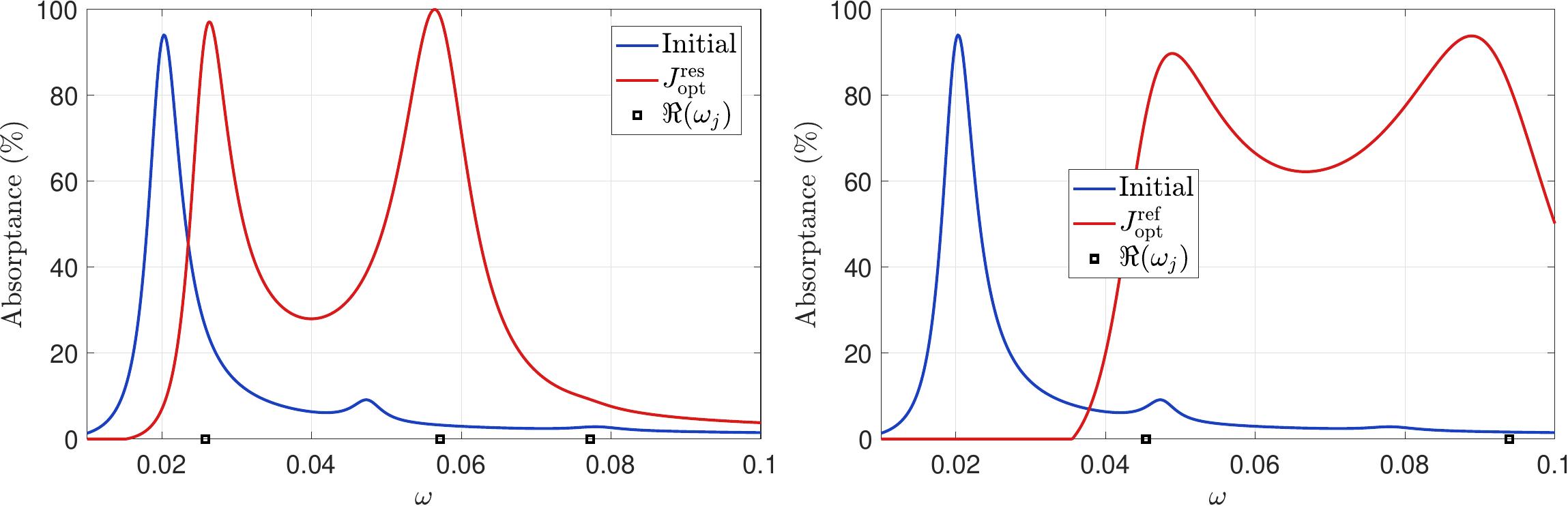}}
    \caption{Shape optimization results for a three resonators configuration arranged vertically ($1\times 3$). The left column corresponds to $J^{\mathrm{res}}$ and the right column to $J^{\mathrm{ref}}$.}
    \label{fig:shape_opt_1x3_Jres_vs_Jref}
\end{figure}

\begin{figure}[htbp]
    \centering
    \subfloat[Convergence history: $J^{\mathrm{res}}$ (left) and $J^{\mathrm{ref}}$ (right)]%
    {\includegraphics[width=.90\textwidth]{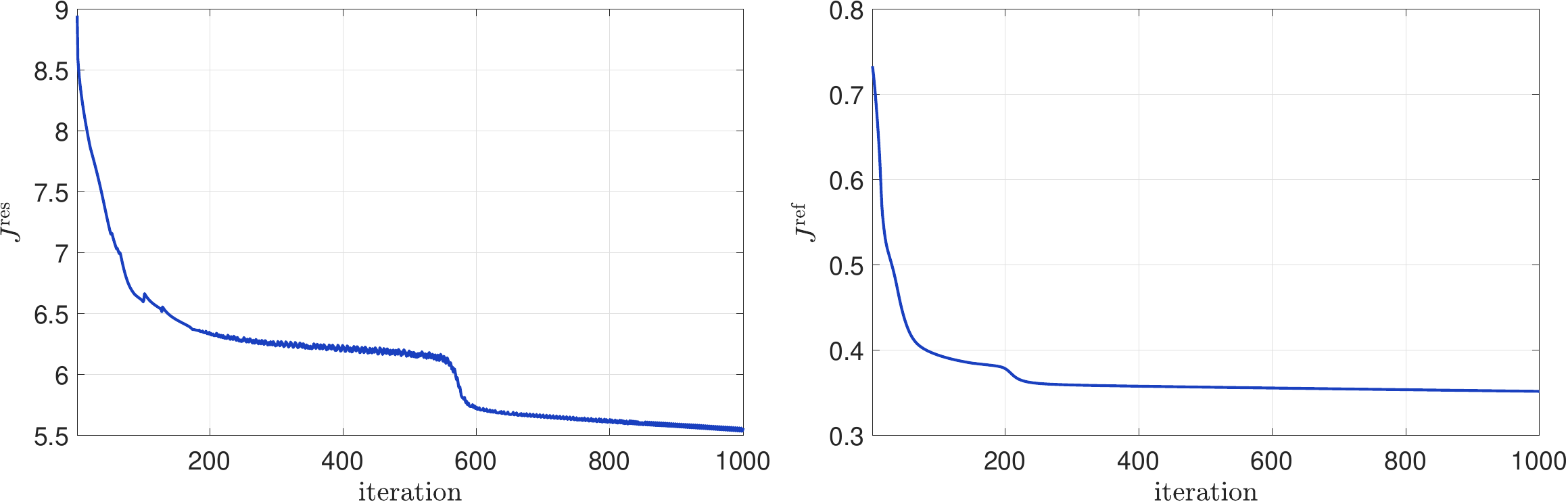}}\\[0.5ex]
    \subfloat[Shape deformation: optimized shapes for $J^{\mathrm{res}}$ (left) and $J^{\mathrm{ref}}$ (right)]%
    {\includegraphics[width=.90\textwidth]{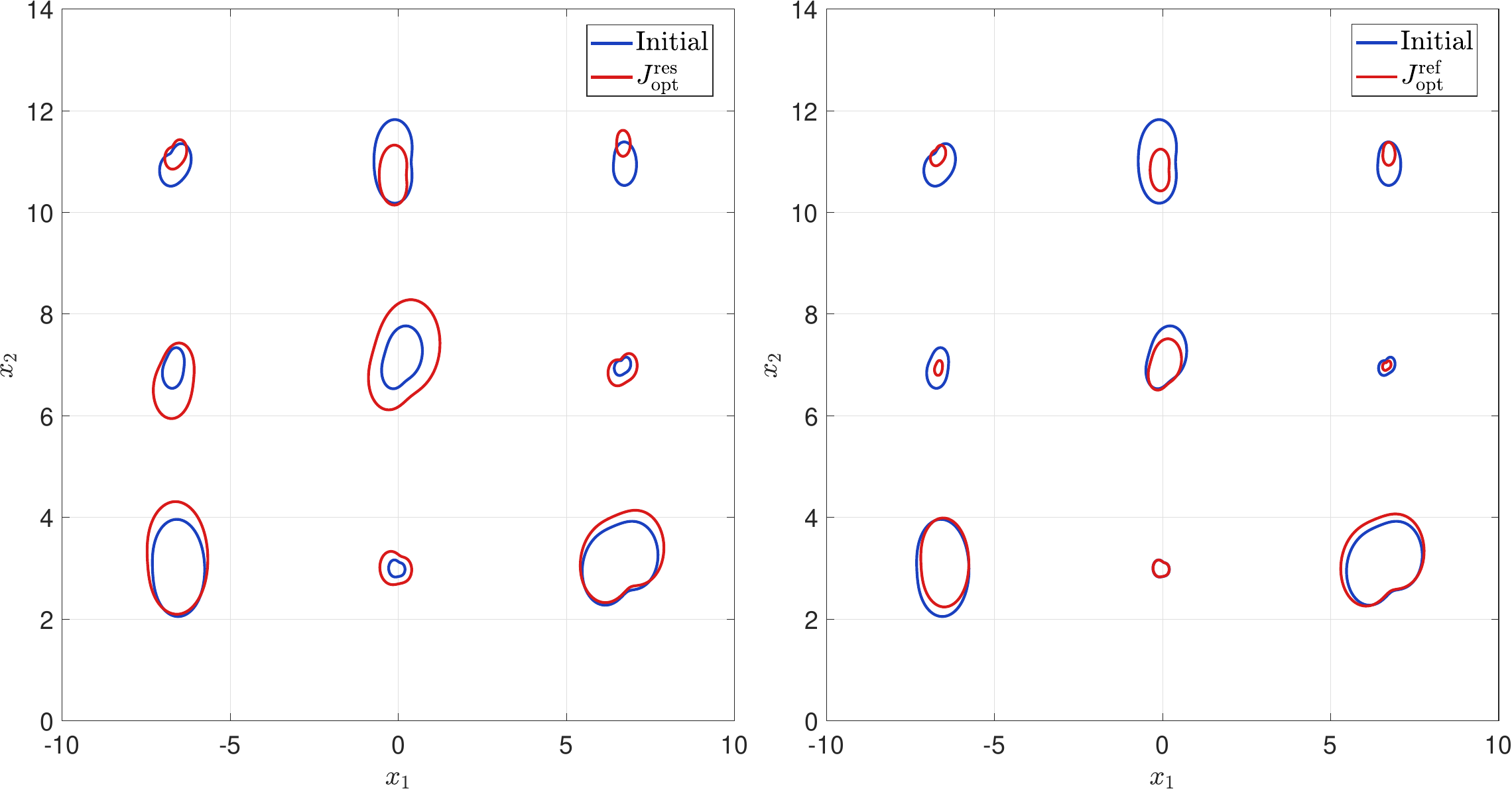}}\\[0.5ex]
    \subfloat[Absorptance spectra: initial vs.\ optimized for $J^{\mathrm{res}}$ (left) and $J^{\mathrm{ref}}$ (right)]%
    {\includegraphics[width=.90\textwidth]{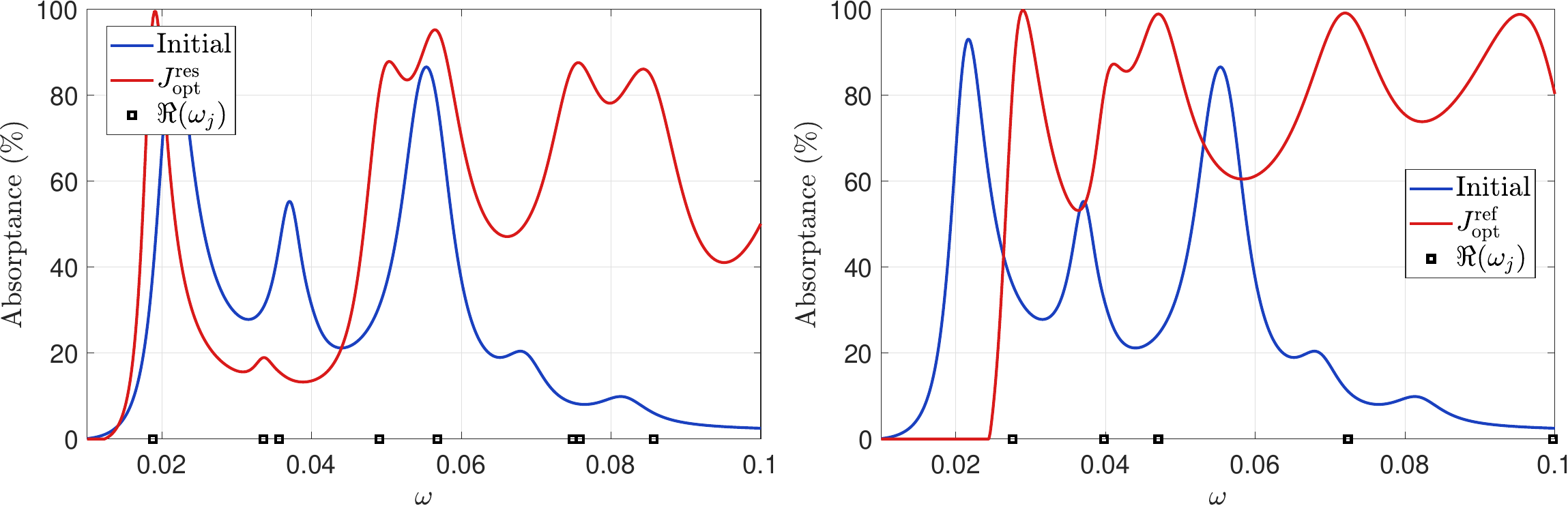}}
    \caption{Shape optimization results for a nine resonators configuration arranged vertically ($3\times 3$). The left column corresponds to $J^{\mathrm{res}}$ and the right column to $J^{\mathrm{ref}}$.}
    \label{fig:shape_opt_3x3_Jres_vs_Jref}
\end{figure}

\begin{figure}[htbp]
    \centering
    \subfloat[Absorptance spectra after optimization of $J^\mathrm{ref}$ for a grid of 9 resonators with circles of radius $r$ as the initial shape.]%
    {\includegraphics[width=.90\textwidth]{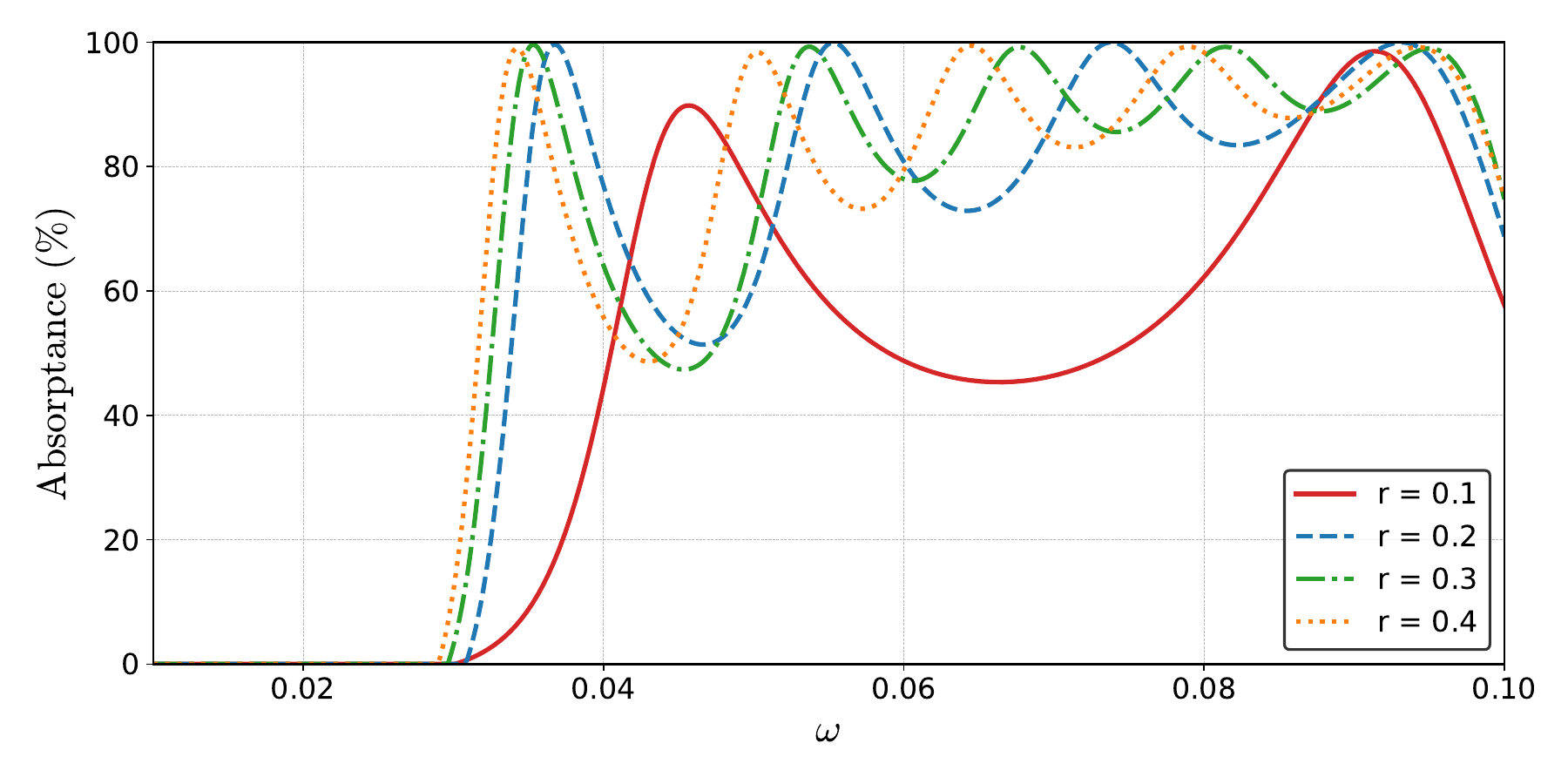}}\\[0.5ex]
    \subfloat[Absorptance spectra after optimization of $J^\mathrm{ref}$ for different grid types with 8 resonators. The period is adjusted according to the number of rows $L = 4(8/N_\text{rows} + 1).$]%
    {\includegraphics[width=.90\textwidth]{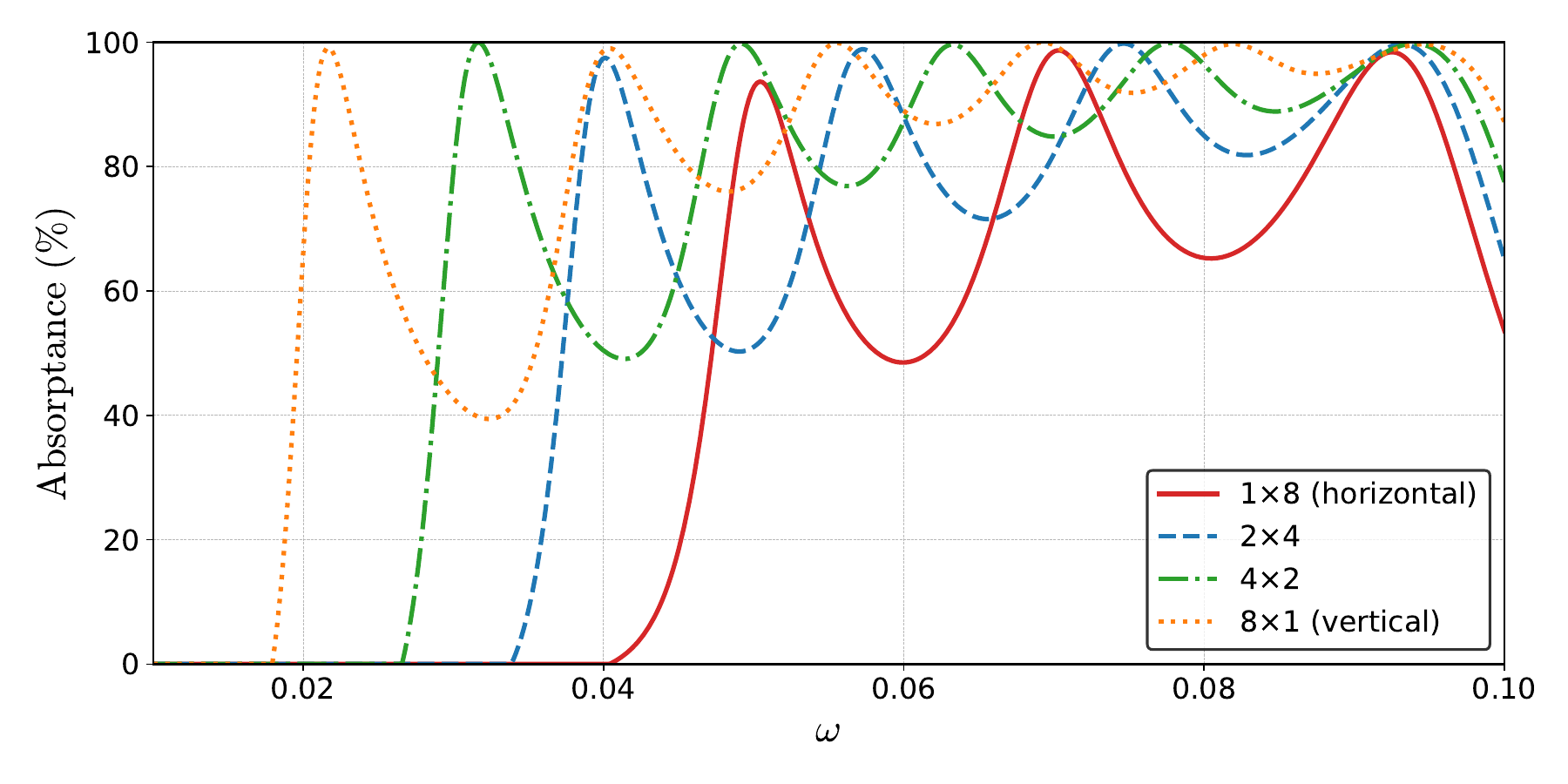}}
    \caption{Reflectance results after shape optimization of $J^\mathrm{ref}$ with different initial shape and grid configurations.}
    \label{fig:shape_opt_comparison_J_ref}
\end{figure}

\begin{figure}[htbp]
    \centering
    \subfloat[Absorptance spectra after optimization of $J^\mathrm{res}$ for a grid of 9 resonators with circles of radius $r$ as the initial shape.]%
    {\includegraphics[width=.90\textwidth]{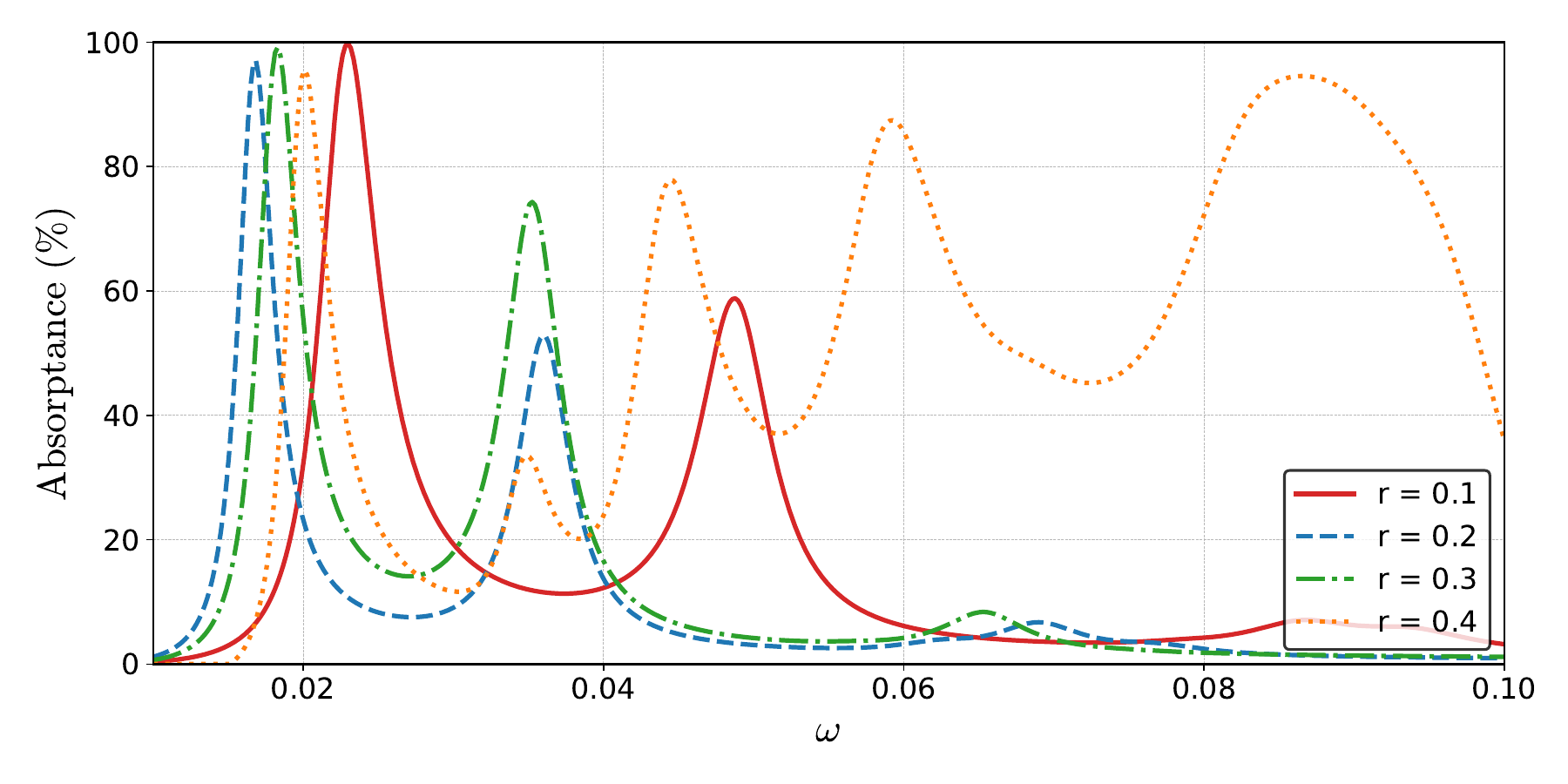}}\\[0.5ex]
    \subfloat[Absorptance spectra after optimization of $J^\mathrm{res}$ for different grid types with 8 resonators. The period is adjusted according to the number of rows $L = 4(8/N_\text{rows} + 1).$]%
    {\includegraphics[width=.90\textwidth]{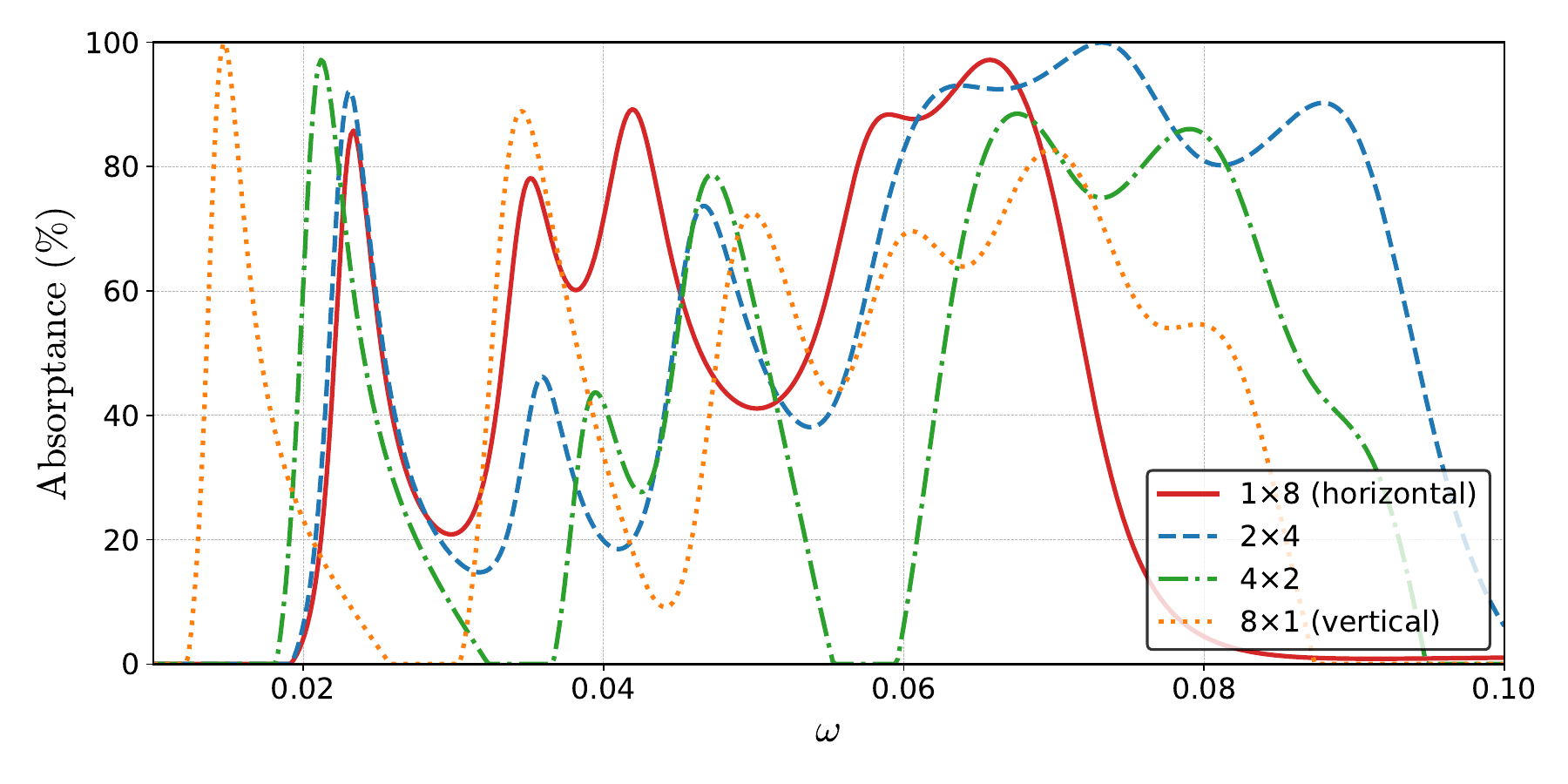}}
    \caption{Reflectance results after shape optimization of $J^\mathrm{res}$ with different initial shape and grid configurations.}
    \label{fig:shape_opt_comparison_J_res}
\end{figure}

\section{Concluding remarks}

In this work,  a reduced order model framework for the broadband optimal design of metascreens with multiple resonators has been developed. In the low-frequency regime, we have established asymptotic formulas for the subwavelength resonant frequencies and for the reflection coefficient. Building on the resulting analytical approximation of the reflection coefficient, we have proposed reduced order objective functionals formulated in terms of the resonant frequencies and the reflection coefficient, which significantly accelerate broadband optimal design. A key feature of the proposed approach is that the dominant computations are frequency independent and therefore need to be carried out only once; the resulting quantities can then be reused efficiently across the entire target frequency band. Finally, to enable gradient-based optimization, we have derived shape derivatives for the relevant eigenvalues and eigenvectors.

This reduced order model framework opens up numerous possibilities for the broadband design of acoustic metascreens based on Helmholtz resonators \cite{hai2015,imerihelmholtz}, as well as electromagnetic metascreens using plasmonic or high-dielectric particles \cite{ruiz2016,bowen2020,bowen2024}. In the present work, the subwavelength resonators are passive; incorporating active elements would provide a route to tunable metascreens \cite{bryn2020,imeritunable1,imeritunable2}. Moreover, future design efforts require a deeper understanding of the fundamental limit between broadband absorption and resonator volume, especially under the physical limits inherent to passive systems. In addition, it is important to develop strategies for selecting better initial designs, since improved initialization can lead to more reliable convergence and higher quality optimal configurations.

%\section*{Acknowledgments}

\begin{appendices}

\crefalias{section}{appendix}

\section{Auxiliary proofs}

\subsection{Proof of \texorpdfstring{\cref{lem: shape derivative for general framework}}{Lemma: shape derivative for general framework}}\label{appdx: proof of shape derivative for general form}

%\subsection{Proof of \cref{lem: shape derivative for general framework}}\label{appdx: proof of shape derivative for general form}

\begin{proof}
    For a general function $w$, we have $\nabla(w \circ T_t) = (DT_t)^{\top}(\nabla w)\circ T_t$, which implies that
\begin{align*}
    J_1(D_t)
    = \int_{Y_{\infty} \setminus \overline{D_t}}
        \nabla u_{t} \cdot \nabla v_{t}\, dx 
    = \int_{Y_{\infty} \setminus \overline{D}}
        (\gamma_t DT_t^{-1} DT_t^{-\top})
        \nabla (u_{t}\circ T_t)\cdot\nabla (v_{t}\circ T_t)\, dx.
\end{align*}
Moreover, using the identity (see \cite{delfour2011shapes})
\begin{align*}
    \frac{d}{dt} \left( \gamma_t DT_t^{-1} DT_t^{-\top} \right)\Big|_{t=0}
    = \operatorname{div}\theta - D\theta - D\theta^{\top},
\end{align*}
differentiating yields
\begin{align*}
    J_1^\prime(D;\theta)
    = \int_{Y_{\infty} \setminus \overline{D}}
    \Big[
        (\operatorname{div}\theta - D\theta - D\theta^{\top})\nabla u
    \Big]\cdot\nabla v
    + \nabla\dot{u}\cdot\nabla v
    + \nabla u\cdot\nabla\dot{v}
    \, dx=\mathrm{I}+\mathrm{II}+\mathrm{III}.
\end{align*}
By a direct calculation, we obtain the following vector identities:
\begin{align*}
&\operatorname{div}((\nabla u\cdot\nabla v)\theta)
    = (\nabla u\cdot\nabla v)\operatorname{div}\theta
      + (\nabla^2 u\,\theta)\cdot\nabla v
      + (\nabla^2 v\,\theta)\cdot\nabla u,\\
&\operatorname{div}((\nabla u \cdot  \theta)\nabla v)
    = (D\theta\,\nabla u)\cdot\nabla v
      + (\nabla^2 u\,\theta)\cdot\nabla v
      + (\nabla u \cdot  \theta)\Delta v,\\
&\operatorname{div}((\nabla v \cdot  \theta)\nabla u)
    = (D\theta\nabla v)\cdot\nabla u
      + (\nabla^2 v\,\theta)\cdot \nabla u
      + (\nabla v \cdot  \theta)\Delta u,
\end{align*}
where $\nabla^2 u$ and $\nabla^2 v$ denote the Hessian matrices. Thus, we obtain
\begin{align*}
    \big[(\operatorname{div}\theta - D\theta - D\theta^{\top})\nabla u\big]
        \cdot \nabla v
    = \operatorname{div}\big((\nabla u\cdot\nabla v)\theta\big)
       - \operatorname{div}\big((\nabla u \cdot  \theta)\nabla v\big)
       - \operatorname{div}\big((\nabla v \cdot  \theta)\nabla u\big).
\end{align*}
Using integration by parts and the divergence theorem, we obtain
\begin{equation}\label{eq: integral of I}
    \mathrm{I}=\int_{\partial D}
    \Big[
        (\nabla u \cdot  \theta) \frac{\partial v}{\partial \nu}+(\nabla v \cdot  \theta)\frac{\partial u}{\partial \nu}-(\nabla u\cdot\nabla v)(\theta\cdot\nu)
    \Big]\Big{|}_{+}\, d\sigma.
\end{equation}
Since $\Delta u =\Delta v=0$ in $Y_{\infty} \setminus \overline{D}$, and $\dot{u} =\nabla f \cdot \velocity$ and $\dot{v} =\nabla g\cdot \velocity$ on $\partial D$, integration by parts gives
\begin{equation}\label{eq: integral of II}
    \mathrm{II}=\int_{Y_{\infty} \setminus \overline{D}}
        \nabla\dot{u} \cdot \nabla v\, dx
    = -\int_{\partial D} \dot{u} \frac{\partial v}{\partial \nu}\Big{|}_{+}\, d\sigma
       - \int_{Y_{\infty} \setminus \overline{D}} \dot{u}\, \Delta v\, dx
    = -\int_{\partial D} (f \cdot \velocity)\frac{\partial v}{\partial \nu}\Big{|}_{+}\, d\sigma,
\end{equation}
\begin{equation}\label{eq: integral of III}
    \mathrm{III}=  \int_{Y_{\infty} \setminus \overline{D}}\nabla u \cdot \nabla \dot{v}\, dx
    = -\int_{\partial D} \dot{v} \frac{\partial u}{\partial \nu}\Big{|}_{+}\, d\sigma
       - \int_{Y_{\infty} \setminus \overline{D}} \dot{v}\, \Delta u\, dx
    = -\int_{\partial D} (g \cdot \velocity)\frac{\partial u}{\partial \nu}\Big{|}_{+}\, d\sigma,
\end{equation}
where we have used the fact that $\nabla u$ and $\nabla v$ decay exponentially as $x_d\rightarrow +\infty$. Combining \eqref{eq: integral of I}, \eqref{eq: integral of II}, and \eqref{eq: integral of III}, we obtain
\begin{align*}
    J_1^\prime(D;\theta) 
    = &\int_{\partial D}
    \Big[
        ((\nabla u-\nabla f) \cdot  \theta )\frac{\partial v}{\partial \nu}
        +((\nabla v-\nabla g) \cdot  \theta )\frac{\partial u}{\partial \nu}
        -(\nabla u\cdot\nabla v)(\theta\cdot\nu)
    \Big]\, d\sigma.
\end{align*}
Finally, using the decomposition of the gradient into tangential and normal components in \eqref{eq: def of tangential gradient}, together with the identities $\nabla_T u=\nabla_T f$ and $\nabla_T v=\nabla_T g$ on $\partial D$, we obtain \eqref{eq: shape derivative for J1}. The proof of \eqref{eq: shape derivative for J2} follows the same argument and is therefore omitted.
\end{proof}

%\subsection{Proof of \cref{thm: shape derivative of eigenvalue}}\label{appdx: shape derivative of eigenvalue}
\subsection{Proof of \texorpdfstring{\Cref{thm: shape derivative of eigenvalue}}{Theorem: shape derivative of eigenvalue}}\label{appdx: shape derivative of eigenvalue}

\begin{proof} 
Note that $v_j=\chi_{\partial D_j}$ and $\nabla_T v_j = 0$ on $\partial D$. Then the shape derivative $C_{ij}'(D;\theta)$ follows directly from \eqref{eq: shape derivative for J1} together with \eqref{eq: def of capacity matrix by PDE}. Moreover, the derivative $V_{ij}'(D;\theta)$ is standard (see, e.g., \cite{delfour2011shapes,sokolowski1992introduction}). The eigenpair $(\lambda_j(D),u_j(D))$ satisfies
\begin{align*}
    C(D)u_j(D) = \lambda_j(D)V(D)u_j(D),\quad
    u_j^{\top}(D)V(D)u_j(D) = 1.
\end{align*}
Differentiating with respect to the domain variation in the direction $\theta$ yields
\begin{align}
\label{eq: derivative of generalized eigenvalue representation}
     &(C-\lambda_jV)u_j^\prime(D;\theta)
   + (C^\prime(D;\theta)-\lambda_jV^\prime(D;\theta))u_j
   = \lambda_j^\prime(D;\theta)Vu_j,\\
   &u_j^{\top}Vu_j^{\prime}(D;\theta)
   = -\frac{1}{2}u_j^\top V^\prime(D;\theta)u_j,
\end{align}
which implies that
\begin{align}
    \lambda_j^\prime(D;\theta)
    &= u_j^{\top}\big(C^\prime(D;\theta)-\lambda_jV^\prime(D;\theta)\big)u_j,\\ 
    u_j^\prime(D;\theta)
    &= \sum_{i\neq j}\frac{u_j^{\top}\big(C^\prime(D;\theta)-\lambda_jV^\prime(D;\theta)\big)u_i}{\lambda_j-\lambda_i} u_i
        -\frac{1}{2}\big(u_j^\top V^\prime(D;\theta) u_j\big)u_j.
\end{align}
According to the definition of $\lambda_{j,1}$ in \eqref{eq: lambdaj1 omega}, its shape derivative is given by
\begin{align*}
    \lambda_{j,1}^\prime(D;\theta)
    = \frac{2\tau_m (m^{\top}  u_j) }{|Y|}\left(u_j^{\top}m^{\prime}(D;\theta)+m^{\top} u^{\prime}_j(D;\theta)\right).
\end{align*}
To compute $m'(D;\theta)$, we rewrite $m_i$ as
\begin{align*}
    m_i
    =-\int_{\partial D} w\Big(\frac{\partial v_i}{\partial \nu}\Big{|}_+
     -\frac{\partial v_i}{\partial \nu}\Big{|}_-\Big)\, \mathrm{d}\sigma
    =\int_{Y_{\infty} \setminus \overline{D}} \nabla w \cdot \nabla v_i \, dx
     +\int_{D} \nabla w \cdot \nabla v_i \, dx,
\end{align*}
and apply \eqref{eq: shape derivative for J12} in \cref{lem: shape derivative for general framework} to obtain $m_i'(D;\theta)$. Finally, the quantities $C_{ij}'(D;\theta)$, $V_{ij}'(D;\theta)$, $\lambda_j'(D;\theta)$, $\lambda_{j,1}'(D;\theta)$, $u_j'(D;\theta)$, and $m'(D;\theta)$ can all be rewritten in Hadamard form, which completes the proof.
\end{proof}

\end{appendices}

%
%%~~~~~~~~~ Bibliography
%\bibliographystyle{unsrt}
\bibliographystyle{siam}
\bibliography{ref}

\end{document}